\documentclass[12pt]{amsart}
\pdfoutput=1

\usepackage[paperheight=11in, paperwidth=8.5in, left=1in, top=1.25in, right=1in, bottom=1.25in]{geometry}
\usepackage[linktocpage=true, linktoc=all, colorlinks=true, linkcolor=black, citecolor=black, filecolor=black, urlcolor=black, pagebackref=false, pdfstartpage={1}, pdfstartview={FitH}, pdftitle={}, pdfauthor={}, pdfsubject={}, pdfcreator={}, pdfproducer={}, pdfkeywords={}]{hyperref}
\usepackage{afterpage}
\usepackage{amsfonts}
\usepackage{amsmath}
\allowdisplaybreaks[4]
\usepackage{amssymb}
\usepackage{amsthm}
\usepackage{array}
\usepackage{bbm}
\usepackage{blkarray,bigstrut}
\usepackage[justification=centering]{caption}
\usepackage[capitalize,nameinlink,noabbrev,compress]{cleveref}
\usepackage{enumerate}
\usepackage{enumitem}
\usepackage{float}
\restylefloat{figure}
\usepackage{graphicx}
\usepackage{mathrsfs}
\usepackage[framemethod=tikz,ntheorem,xcolor]{mdframed}
\usepackage{parcolumns}
\usepackage{tabularx}
\usepackage{tikz}\pdfpageattr{/Group <</S /Transparency /I true /CS /DeviceRGB>>}
\usetikzlibrary{arrows,calc,intersections,matrix,positioning,through}
\usepackage{tikz-cd}
\tikzset{commutative diagrams/.cd,every label/.append style = {font = \normalsize}}
\usepackage[all]{xy}
\usepackage[aligntableaux=center]{ytableau}
\ytableausetup{boxsize=16pt}

\numberwithin{equation}{section}
\newtheorem{thm}{Theorem}
\AfterEndEnvironment{thm}{\noindent\ignorespaces}
\numberwithin{thm}{section}

\AfterEndEnvironment{conj}{\noindent\ignorespaces}
\newtheorem{cor}[thm]{Corollary}
\AfterEndEnvironment{cor}{\noindent\ignorespaces}
\newtheorem{lem}[thm]{Lemma}
\AfterEndEnvironment{lem}{\noindent\ignorespaces}
\newtheorem{prob}[thm]{Problem}
\AfterEndEnvironment{prob}{\noindent\ignorespaces}
\newtheorem{prop}[thm]{Proposition}
\AfterEndEnvironment{prop}{\noindent\ignorespaces}

\theoremstyle{definition}
\newtheorem{defn}[thm]{Definition}
\AfterEndEnvironment{defn}{\noindent\ignorespaces}
\newtheorem{eg_no_qed}[thm]{Example}
\newenvironment{eg}[1][]{\begin{eg_no_qed}[#1]\pushQED{\qed}}{\popQED\end{eg_no_qed}}
\AfterEndEnvironment{eg}{\noindent\ignorespaces}
\newtheorem{rmk}[thm]{Remark}
\AfterEndEnvironment{rmk}{\noindent\ignorespaces}

\theoremstyle{remark}

\AfterEndEnvironment{claim}{\noindent\ignorespaces}
\newtheorem*{claimpf_no_qed}{Proof of Claim}

\AfterEndEnvironment{claimpf}{\noindent\ignorespaces}

\crefname{thm}{Theorem}{Theorems}
\crefname{conj}{Conjecture}{Conjectures}
\crefname{cor}{Corollary}{Corollaries}
\crefname{lem}{Lemma}{Lemmas}
\crefname{prob}{Problem}{Problems}
\crefname{prop}{Proposition}{Propositions}
\crefname{defn}{Definition}{Definitions}
\crefname{eg}{Example}{Examples}
\crefname{eg_no_qed}{Example}{Examples}
\crefname{rmk}{Remark}{Remarks}

\DeclareMathOperator{\Aut}{Aut}

\newcommand{\End}{\mathfrak{gl}}

\newcommand{\Fl}{\textnormal{Fl}}

\newcommand{\GL}{\textnormal{GL}}
\newcommand{\gl}{\mathfrak{gl}}
\newcommand{\Gr}{\textnormal{Gr}}

\DeclareMathOperator{\id}{id}

\DeclareMathOperator{\rank}{rank}

\DeclareMathOperator{\soc}{soc}

\DeclareMathOperator{\SSYT}{SSYT}

\newcommand{\Stab}{\textnormal{Stab}}

\newcommand{\arrowsum}{\mathfrak{m}}
\newcommand{\bigspn}[1]{\left\langle {#1}\right\rangle}
\newcommand{\block}[2]{#1[#2]}
\newcommand{\Bor}{B}
\newcommand{\Boxes}{V}
\DeclareMathOperator{\colsum}{\mathsf{colsum}}

\newcommand{\dimcomp}[2]{\mathsf{d}(#1,#2)}
\newcommand{\dec}[1]{V_{#1}}
\newcommand{\decnil}[2]{V^{\sharp}_{#1,#2}}

\newcommand{\double}[1]{{#1}^{\sharp}}

\newcommand{\emptytab}{\varnothing}
\newcommand{\evac}{\mathsf{evac}}
\newcommand{\field}{\Bbbk}
\newcommand{\flags}[1]{\mathsf{Flags}(#1)}
\newcommand{\growth}[1]{\Gamma(#1)}
\newcommand{\indec}[2]{V_{#1,#2}}
\DeclareMathOperator{\JF}{\mathsf{JF}{\hspace*{-0.2pt}\raisebox{4pt}{$\scriptstyle\mathsf{T}$}}}

\newcommand{\Mats}[2]{\mathsf{Mat}_{#1,#2}}
\newcommand{\mat}[3]{\mathsf{mat}(#1;#2,#3)}

\newcommand{\ncon}[1]{\tilde{\mathsf{n}}(#1)}
\newcommand{\Nil}{\mathfrak{n}}
\newcommand{\nil}{N}

\newcommand{\pa}[1]{\field{#1}}
\newcommand{\Par}{P}

\newcommand{\perm}[1]{\mathsf{w}_{#1}}
\newcommand{\permdot}[1]{\mathring{\mathsf{w}}_{#1}}

\newcommand{\ppa}[1]{\Pi(#1)}
\newcommand{\ppareps}[2][]{\mathcal{V}_{#2}^{#1}}
\newcommand{\pF}[1]{\mathbb{F}_{#1}}
\newcommand{\PFl}{\textnormal{Fl}}
\renewcommand{\Pr}{\mathbb{P}}
\newcommand{\pr}[2]{\mathsf{p}_{#1}(#2)}
\newcommand{\prdisplay}[2]{\mathsf{p}_{#1}\hspace*{-2pt}\left(#2\right)}
\newcommand{\prdual}[2]{\tilde{\mathsf{p}}_{#2}(#1)}
\newcommand{\prdualdisplay}[2]{\tilde{\mathsf{p}}_{\hspace*{1pt}#2\hspace*{2pt}}\hspace*{-2pt}\Big(#1\Big)}
\newcommand{\prdualnoarg}{\tilde{\mathsf{p}}}
\newcommand{\prnoarg}{\mathsf{p}}
\newcommand{\Ptab}[1]{\mathsf{P}(#1)}
\newcommand{\qbinom}[2]{\begin{bmatrix}#1 \\ #2\end{bmatrix}_q}
\newcommand{\qbinomdisplay}[2]{\begin{bmatrix}#1 \\[4pt] #2\end{bmatrix}_q}
\newcommand{\qbinomsmall}[2]{\scalebox{0.7}{$\begin{bmatrix}{#1} \\ {#2}\end{bmatrix}$}_q}
\newcommand{\qF}{\mathbb{F}_q}
\newcommand{\qfac}[1]{[#1]_q!}
\newcommand{\qn}[1]{[#1]_q}

\newcommand{\qinvbinomsmall}[2]{\scalebox{0.7}{$\begin{bmatrix}{#1} \\ {#2}\end{bmatrix}$}_{1/q}}
\newcommand{\qinvF}{\mathbb{F}_{1/q}}

\newcommand{\qinvn}[1]{[#1]_{1/q}}
\newcommand{\Qtab}[1]{\mathsf{Q}(#1)}
\newcommand{\qv}[1]{\mathcal{M}({#1})}
\newcommand{\qvpoint}[2]{U^{\sharp}_{#1,#2}}
\newcommand{\qwt}{\operatorname{wt}_q}

\newcommand{\rect}[2]{\mathsf{Rect}(#1,#2)}
\newcommand{\relpos}[1]{\xrightarrow{\makebox[12pt]{\scriptsize{$#1$}}}}
\newcommand{\rev}[1]{\overline{#1}}
\DeclareMathOperator{\rowsum}{\mathsf{rowsum}}
\newcommand{\RPP}[2]{\mathsf{RPP}(#1,#2)}
\newcommand{\RPPbig}[2]{\mathsf{RPP}\!\left(#1,#2\right)}
\newcommand{\RPPs}[2]{\mathcal{R}_{#1,#2}}
\newcommand{\socf}[2]{\soc^{#2}(#1)}
\newcommand{\socRPP}[1]{\mathsf{RPP}(#1)}
\newcommand{\source}[1]{s(#1)}
\newcommand{\spn}[1]{\langle #1\rangle}
\DeclareMathOperator{\Sym}{\mathfrak{S}}

\newcommand{\Tabs}[2]{\mathcal{T}_{#1,#2}}
\newcommand{\target}[1]{t(#1)}

\newcommand{\transpose}[1]{\mathchoice
{{#1}{\hspace*{-0.2pt}\raisebox{4pt}{$\scriptstyle\mathsf{T}$}\hspace*{0.5pt}}}
{{#1}{\hspace*{-0.2pt}\raisebox{4pt}{$\scriptstyle\mathsf{T}$}\hspace*{0.5pt}}}
{{#1}{\hspace*{-0.2pt}\raisebox{3pt}{$\scriptscriptstyle\mathsf{T}$}\hspace*{0.5pt}}}
{{#1}{\hspace*{-0.2pt}\raisebox{3pt}{$\scriptscriptstyle\mathsf{T}$}\hspace*{0.5pt}}}
}
\newcommand{\triples}[2][]{\Upsilon_{#2}^{#1}}
\DeclareMathOperator{\type}{\mathsf{type}}

\newcommand{\Unit}[1]{E^{#1}}
\newcommand{\unit}[1]{\mathsf{e}_{#1}}
\newcommand{\wtf}{\psi}
\newcommand{\wtfdual}{\tilde{\psi}}
\newcommand{\ydiagramsmall}[1]{\ytableausetup{smalltableaux}\, \ydiagram{#1} \, \ytableausetup{nosmalltableaux}}
\newcommand{\ytableausmall}[1]{\ytableausetup{smalltableaux}\, \begin{ytableau}#1\end{ytableau}\, \ytableausetup{nosmalltableaux}}

\title{\lowercase{$q$}-Whittaker functions, finite fields, and Jordan forms}

\author{Steven N. Karp}
\address{Department of Mathematics, University of Notre Dame}
\email{\href{mailto:skarp2@nd.edu}{skarp2@nd.edu}}
\author{Hugh Thomas}
\address{LACIM, Universit\'{e} du Qu\'{e}bec \`{a} Montr\'{e}al}
\email{\href{mailto:thomas.hugh_r@uqam.ca}{thomas.hugh\_r@uqam.ca}}
\subjclass[2020]{05E05, 60C05, 14M15, 15B33, 16G20, 05A30}
\thanks{S.N.K.\ was partially supported by an NSERC postdoctoral fellowship. H.T.\ was partially supported by an NSERC Discovery Grant and the Canada Research Chairs program.}

\begin{document}

\begin{abstract}
The $q$-Whittaker function $W_\lambda(\mathbf{x};q)$ associated to a partition $\lambda$ is a $q$-analogue of the Schur function $s_\lambda(\mathbf{x})$, and is defined as the $t=0$ specialization of the Macdonald polynomial $P_\lambda(\mathbf{x};q,t)$. We show combinatorially how to expand $W_\lambda(\mathbf{x};q)$ in terms of partial flags compatible with a nilpotent endomorphism over the finite field of size $1/q$. This yields an expression analogous to a well-known formula for the Hall--Littlewood functions. We show that considering pairs of partial flags and taking Jordan forms leads to a probabilistic bijection between nonnegative-integer matrices and pairs of semistandard tableaux of the same shape, proving the Cauchy identity for $q$-Whittaker functions. We call our probabilistic bijection the {\itshape $q$-Burge correspondence}, and prove that in the limit as $q\to 0$, we recover a description of the classical Burge correspondence (also known as column RSK) due to Rosso (2012). A key step in the proof is the enumeration of an arbitrary double coset of $\GL_n$ modulo two parabolic subgroups, which we find to be of independent interest. As an application, we use the $q$-Burge correspondence to count isomorphism classes of certain modules over the preprojective algebra of a type $A$ quiver (i.e.\ a path), refined according to their socle filtrations. This develops a connection between the combinatorics of symmetric functions and the representation theory of preprojective algebras.
\end{abstract}

\maketitle
\setcounter{tocdepth}{1}
\tableofcontents

\section{Introduction}\label{sec_introduction}

\noindent The {\itshape Cauchy identity} for the Schur functions $s_\lambda(\mathbf{x})$
\begin{align}\label{cauchy_equation_intro}
\prod_{i,j \ge 1}\frac{1}{1-x_iy_j} = \sum_{\lambda}s_\lambda(\mathbf{x})s_\lambda(\mathbf{y})
\end{align}
plays an important role in the theories of symmetric functions \cite{macdonald95,stanley99} and Schur processes \cite{okounkov01,okounkov_reshetikhin03}. A bijective proof of \eqref{cauchy_equation_intro} is provided by the celebrated {\itshape Robinson--Schensted--Knuth correspondence} \cite{knuth70}, or its column-insertion variant, the {\itshape Burge correspondence} \cite{burge74} (which will be more relevant for us). Namely, we expand the left-hand side of \eqref{cauchy_equation_intro} as a sum over nonnegative-integer matrices, and expand the right-hand side as a sum over pairs of semistandard tableaux of the same shape, and these correspondences give bijections (for each coefficient) between such matrices and pairs of tableaux. For example, the coefficient of $x_1x_2y_1y_2$ on each side of \eqref{cauchy_equation_intro} is $2$, and the Burge correspondence acts by
\begin{align}\label{burge_equation_intro}
\scalebox{0.8}{$\begin{bmatrix}1 & 0 \\ 0 & 1\end{bmatrix}$} \mapsto \left(\ytableausmall{1 \\ 2},\ytableausmall{1 \\ 2}\right), \qquad \scalebox{0.8}{$\begin{bmatrix}0 & 1 \\ 1 & 0\end{bmatrix}$} \mapsto \left(\ytableausmall{1 & 2},\ytableausmall{1 & 2}\right).
\end{align}

In this paper we consider the following $q$-analogue of \eqref{cauchy_equation_intro}, where the Schur function $s_\lambda(\mathbf{x})$ is replaced with the {\itshape $q$-Whittaker function} $W_\lambda(\mathbf{x};q)$, which is the $t=0$ specialization of the Macdonald polynomial $P_\lambda(\mathbf{x};q,t)$:
\begin{thm}[{Macdonald \cite[VI.(4.13)]{macdonald95}}]\label{q-cauchy_intro}
We have the Cauchy identity for $q$-Whittaker functions
\begin{align}\label{q-cauchy_equation_intro}
\prod_{\substack{i,j \ge 1,\\ d \ge 0\hphantom{,}}}\frac{1}{1-x_iy_jq^d} = \sum_\lambda \frac{(1-q)^{-\lambda_1}}{\prod_{i\ge 1}\qfac{\lambda_i - \lambda_{i+1}}}W_\lambda(\mathbf{x};q)W_\lambda(\mathbf{y};q),
\end{align}
where the sum on the right-hand side is over all partitions $\lambda$. Equivalently, let $\alpha = (\alpha_1, \dots, \alpha_k)$ and $\beta = (\beta_1, \dots, \beta_l)$ be weak compositions of $n$. Taking the coefficient of $x_1^{\alpha_1}\cdots x_k^{\alpha_k}y_1^{\beta_1}\cdots y_l^{\beta_l}$ on each side of \eqref{q-cauchy_equation_intro}, we obtain
\begin{align}\label{q-cauchy_coefficient_intro}
\sum_{M\in\Mats{\alpha}{\beta}}\frac{(1-q)^{-n}}{\prod_{1 \le i \le k,\hspace*{1pt} 1 \le j \le l}\qfac{M_{i,j}}} = \sum_{\substack{\lambda\vdash n,\\ T\in\SSYT(\lambda,\alpha),\\ T'\in\SSYT(\lambda,\beta)}}\frac{(1-q)^{-\lambda_1}}{\prod_{i\ge 1}\qfac{\lambda_i - \lambda_{i+1}}}\qwt(T)\qwt(T'),
\end{align}
where the sum on the left-hand side is over all $k\times l$ nonnegative-integer matrices $M$ with row sums $\alpha$ and column sums $\beta$, and the sum on the right-hand side is over all pairs $(T,T')$ of semistandard tableaux of the same shape with contents $\alpha$ and $\beta$, respectively.
\end{thm}

(We refer to \cref{sec_notation} for the notation used in \cref{q-cauchy_intro}, and to \cref{lhs_equivalence} for how to obtain the left-hand side of \eqref{q-cauchy_coefficient_intro} from \eqref{q-cauchy_equation_intro}.) The $q$-Whittaker functions have received much attention recently in large part due to their appearance in the theory of Macdonald processes \cite{borodin_corwin14}. Indeed, \eqref{q-cauchy_equation_intro} may be interpreted as a partition function identity for $q$-Whittaker processes.

\cref{q-cauchy_intro} is a special case of the Cauchy identity for Macdonald polynomials, and Macdonald's original proof of \eqref{q-cauchy_equation_intro} is algebraic. Our goal is to give a geometric and combinatorial proof of the equivalent identity \eqref{q-cauchy_coefficient_intro}. In the case $q=0$, the summands on each side of \eqref{q-cauchy_coefficient_intro} are all $1$, and a bijection between the index sets of each side is provided by the Burge correspondence. However, for general $q$, there is no weight-preserving bijection between the index sets of each side of \eqref{q-cauchy_coefficient_intro}. For example, the coefficient of $x_1x_2y_1y_2$ on each side of \eqref{q-cauchy_equation_intro} is $2(1-q)^{-2}$:\setlength{\tabcolsep}{3pt}
$$
\begin{tabular}{ccccccc}
$(1-q)^{-2}$ & $+$ & $(1-q)^{-2}$ & $=$ & $(1-q)^{-1}$ & $+$ & $(1-q)^{-2}(1+q).$ \\[8pt]
\scalebox{0.8}{$\begin{bmatrix}1 & 0 \\ 0 & 1\end{bmatrix}$} & & \scalebox{0.8}{$\begin{bmatrix}0 & 1 \\ 1 & 0\end{bmatrix}$} & & $\left(\ytableausmall{1 \\ 2},\ytableausmall{1 \\ 2}\right)$ & & $\left(\ytableausmall{1 & 2},\ytableausmall{1 & 2}\right)$
\end{tabular}
$$
\setlength{\tabcolsep}{6pt}We see that for $q\neq 0$, the weights on the left-hand side are different from the weights on the right-hand side. Nevertheless, we can define a (weight-preserving) {\itshape probabilistic bijection}, for example, one which acts by
\begin{align}\label{eg_n=2_equation}
\scalebox{0.8}{$\begin{bmatrix}1 & 0 \\ 0 & 1\end{bmatrix}$} \mapsto \begin{cases}
\left(\ytableausmall{1 \\ 2},\ytableausmall{1 \\ 2}\right), & \text{with probability $1-q$}; \\[6pt]
\left(\ytableausmall{1 & 2},\ytableausmall{1 & 2}\right), & \text{with probability $q$},
\end{cases} \qquad \scalebox{0.8}{$\begin{bmatrix}0 & 1 \\ 1 & 0\end{bmatrix}$} \mapsto \left(\ytableausmall{1 & 2},\ytableausmall{1 & 2}\right).
\end{align}
That is, given a nonnegative-integer matrix $M$, we map it to each pair of semistandard tableaux $(T,T')$ with some probability $\pr{M}{T,T'}$ (which is a function of $q$), such that all probabilities of the form $\pr{M}{\cdot}$ sum to $1$. We also require that weights are preserved, i.e., the weight of each pair $(T,T')$ is equal to the sum over $M$ of $\pr{M}{T,T'}$ times the weight of $M$. This implies that the desired identity \eqref{q-cauchy_coefficient_intro} holds. (See \cref{sec_probabilistic_bijections} for the formal definition of a probabilistic bijection.) The inverse probabilistic bijection of \eqref{eg_n=2_equation} is given as follows:
\begin{align}\label{eg_n=2_equation_dual}
\left(\ytableausmall{1 \\ 2},\ytableausmall{1 \\ 2}\right) \mapsto \scalebox{0.8}{$\begin{bmatrix}1 & 0 \\ 0 & 1\end{bmatrix}$}, \qquad \left(\ytableausmall{1 & 2},\ytableausmall{1 & 2}\right) \mapsto \begin{cases}
\scalebox{0.8}{\renewcommand\arraystretch{1}$\begin{bmatrix}1 & 0 \\ 0 & 1\end{bmatrix}$}, & \text{with probability $\frac{q}{1+q}$}; \\[10pt]
\scalebox{0.8}{\renewcommand\arraystretch{1}$\begin{bmatrix}0 & 1 \\ 1 & 0\end{bmatrix}$}, & \text{with probability $\frac{1}{1+q}$}.
\end{cases}
\end{align}

Two different probabilistic bijections proving \eqref{q-cauchy_equation_intro} were given by Matveev and Petrov \cite{matveev_petrov17} using randomized insertion algorithms; see \cref{other_bijections}. We prove \eqref{q-cauchy_equation_intro} by introducing a new probabilistic bijection which we call the {\itshape $q$-Burge correspondence}, which is defined when $1/q$ is a prime power via the geometry of nilpotent endomorphisms and partial flags over the finite field $\qinvF$. In the limit $q\to 0$ (which essentially replaces $\qinvF$ by an infinite field), we recover a geometric description of the classical Burge correspondence due to Rosso \cite{rosso12}.

In order to state our result, we introduce some terminology. If $\nil$ is a nilpotent endomorphism of $\field^n$, its {\itshape (conjugated) Jordan form partition} is the conjugate partition of the sequence of sizes of the Jordan blocks of $\nil$. We say that a partial flag of subspaces
$$
F: 0 = F_0 \subseteq F_1 \subseteq \cdots \subseteq F_k = \field^n
$$
is {\itshape strictly compatible} with $\nil$ if
$$
\nil(F_i) \subseteq F_{i-1} \quad \text{ for all } 1 \le i \le k.
$$
Restricting $\nil$ to $F_0, \dots, F_k$ gives a sequence of nilpotent endomorphisms whose (conjugated) Jordan form partitions define a semistandard tableau $\JF(\nil;F)$ of size $n$ on the alphabet $\{1, \dots, k\}$. Also, if $F = (F_0, \dots, F_k)$ and $F' = (F'_0, \dots, F'_l)$ are two partial flags in $\field^n$, we define the {\itshape relative position of $(F,F')$} as the $k\times l$ nonnegative-integer matrix $M$ satisfying
$$
\dim(F_i\cap F'_j) = \sum_{1\le i'\le i,\hspace*{2pt} 1\le j'\le j}M_{i',j'} \quad \text{ for all } 0 \le i \le k \text{ and } 0 \le j \le l.
$$
We can now present the $q$-Burge correspondence (see \cref{probabilities_interpretation}, \cref{reversibility}, \cref{transpose_symmetry}, and \cref{burge_limit}):
\begin{thm}\label{cauchy_intro}
Set $\field := \qinvF$. Let $M$ be a $k\times l$ nonnegative-integer matrix with row sums $\alpha$ and column sums $\beta$ whose entries sum to $n$, and let $(T,T')$ be a pair of semistandard tableaux of the same shape of size $n$ with contents $\alpha$ and $\beta$, respectively.
\begin{enumerate}[label=(\roman*), leftmargin=*]
\item\label{cauchy_intro_bijection} Given a pair of partial flags $(F,F')$ in $\field^n$ with relative position $M$, we define $\pr{M}{T,T'}$ to be the probability that a uniformly random nilpotent endomorphism $\nil$ of $\field^n$ which is strictly compatible with both $F$ and $F'$ satisfies
\begin{align}\label{flag_tableau_equation}
\JF(\nil;F) = T \quad \text{ and } \quad \JF(\nil;F') = T'.
\end{align}
(This definition does not depend on the choice of $(F,F')$.) Then the map $\prnoarg$ defines a probabilistic bijection proving \eqref{q-cauchy_coefficient_intro} whenever $1/q$ is a prime power. (This is sufficient to prove \cref{q-cauchy_intro} for all $q$; see \cref{remark_infinitely_many}.)
\item\label{cauchy_intro_bijection_dual} Given a nilpotent endomorphism $\nil$ of $\field^n$ whose (conjugated) Jordan form partition is the shape of both $T$ and $T'$, we define $\prdual{M}{T,T'}$ to be the probability that a uniformly random pair of partial flags $(F,F')$ satisfying \eqref{flag_tableau_equation} has relative position $M$. (This definition does not depend on the choice of $\nil$.) Then the map $\prdualnoarg$ defines the inverse of the probabilistic bijection in part \ref{cauchy_intro_bijection}, and therefore proves \eqref{q-cauchy_coefficient_intro} whenever $1/q$ is a prime power.
\item\label{cauchy_intro_symmetry} Transposing the matrix corresponds to swapping the tableaux:
$$
\pr{M}{T,T'} = \pr{\transpose{M}}{T',T} \quad \text{ and } \quad \prdual{M}{T,T'} = \prdual{\transpose{M}}{T',T}.
$$
\item\label{cauchy_intro_specialization} As $q\to 0$, the probabilities $\pr{M}{T,T'}$ and $\prdual{M}{T,T'}$ both converge to $1$ if the classical Burge correspondence sends $M$ to $(T,T')$, and both converge to $0$ otherwise.
\end{enumerate}
\end{thm}

\begin{eg}\label{eg_cauchy_intro}
Let us show that $\prnoarg$ as defined in \cref{cauchy_intro}\ref{cauchy_intro_bijection} leads to \eqref{eg_n=2_equation}. For $M = \scalebox{0.8}{$\begin{bmatrix}1 & 0 \\ 0 & 1\end{bmatrix}$}$, we can take $F = F'$ to be the flag
\begin{align}\label{eg_cauchy_intro_flag}
F_0 = 0 \subseteq F_1 = \langle\unit{1}\rangle \subseteq F_2 = \langle\unit{1},\unit{2}\rangle.
\end{align}
We may write an arbitrary nilpotent endomorphism $\nil$ which is strictly compatible with $F = F'$ as
$$
\nil = \begin{bmatrix}0 & a \\ 0 & 0\end{bmatrix}, \quad \text{ where } a\in\qinvF.
$$
  If $a\neq 0$ (which occurs with probability $1-q$), then in accordance with \eqref{flag_tableau_equation} we obtain the pair of tableaux $\left(\ytableausmall{1 \\ 2},\ytableausmall{1 \\ 2}\right)$, while if $a=0$ (which occurs with probability $q$), we obtain the pair of tableaux $\left(\ytableausmall{1 & 2},\ytableausmall{1 & 2}\right)$:
$$
\prdisplay{\scalebox{0.8}{$\begin{bmatrix}1 & 0 \\ 0 & 1\end{bmatrix}$}}{\ytableausmall{1 \\ 2},\ytableausmall{1 \\ 2}} = 1-q, \qquad \prdisplay{\scalebox{0.8}{$\begin{bmatrix}1 & 0 \\ 0 & 1\end{bmatrix}$}}{\ytableausmall{1 & 2},\ytableausmall{1 & 2}} = q.
$$

For $M = \scalebox{0.8}{$\begin{bmatrix}0 & 1 \\ 1 & 0\end{bmatrix}$}$, we must take $F \neq F'$, and the only nilpotent endomorphism $\nil$ strictly compatible with both $F$ and $F'$ is $\nil = 0$. Then we obtain the pair of tableaux $\left(\ytableausmall{1 & 2},\ytableausmall{1 & 2}\right)$ (with probability $1$):
$$
\prdisplay{\scalebox{0.8}{$\begin{bmatrix}0 & 1 \\ 1 & 0\end{bmatrix}$}}{\ytableausmall{1 \\ 2},\ytableausmall{1 \\ 2}} = 0, \qquad \prdisplay{\scalebox{0.8}{$\begin{bmatrix}0 & 1 \\ 1 & 0\end{bmatrix}$}}{\ytableausmall{1 & 2},\ytableausmall{1 & 2}} = 1.
$$

Similarly, let us show that $\prdualnoarg$ as defined in \cref{cauchy_intro}\ref{cauchy_intro_bijection_dual} leads to \eqref{eg_n=2_equation_dual}. For $\left(\ytableausmall{1 \\ 2},\ytableausmall{1 \\ 2}\right)$, we can take $\nil = \scalebox{0.8}{$\begin{bmatrix}0 & 1 \\ 0 & 0\end{bmatrix}$}$. Then the only pair $(F,F')$ satisfying \eqref{flag_tableau_equation} is obtained when $F = F'$ is the flag \eqref{eg_cauchy_intro_flag}, in which case the relative position of $(F,F')$ is $\scalebox{0.8}{$\begin{bmatrix}1 & 0 \\ 0 & 1\end{bmatrix}$}$ (with probability $1$):
$$
\prdualdisplay{\scalebox{0.8}{$\begin{bmatrix}1 & 0 \\ 0 & 1\end{bmatrix}$}}{\ytableausmall{1 \\ 2},\ytableausmall{1 \\ 2}} = 1, \qquad \prdualdisplay{\scalebox{0.8}{$\begin{bmatrix}0 & 1 \\ 1 & 0\end{bmatrix}$}}{\ytableausmall{1 \\ 2},\ytableausmall{1 \\ 2}} = 0.
$$

For $\left(\ytableausmall{1 & 2},\ytableausmall{1 & 2}\right)$, we must take $\nil=0$, and every pair $(F,F')$ satisfies \eqref{flag_tableau_equation}. If $F=F'$ (which occurs with probability $\frac{q}{1+q}$), then the relative position of $(F,F')$ is $\scalebox{0.8}{$\begin{bmatrix}1 & 0 \\ 0 & 1\end{bmatrix}$}$, while if $F\neq F'$ (which occurs with probability $\frac{1}{1+q}$), the relative position of $(F,F')$ is \scalebox{0.8}{$\begin{bmatrix}0 & 1 \\ 1 & 0\end{bmatrix}$}:
$$
\prdualdisplay{\scalebox{0.8}{$\begin{bmatrix}1 & 0 \\ 0 & 1\end{bmatrix}$}}{\ytableausmall{1 & 2},\ytableausmall{1 & 2}} = \frac{q}{1+q}, \qquad \prdualdisplay{\scalebox{0.8}{$\begin{bmatrix}0 & 1 \\ 1 & 0\end{bmatrix}$}}{\ytableausmall{1 & 2},\ytableausmall{1 & 2}} = \frac{1}{1+q}.
$$
Note that in the limit $q\to 0$, we recover the classical Burge correspondence \eqref{burge_equation_intro}.
\end{eg}

As in \cref{eg_cauchy_intro}, when calculating $\pr{M}{\cdot}$ in general, we can always choose the pair of partial flags $(F,F')$ in \cref{cauchy_intro}\ref{cauchy_intro_bijection} so that the nilpotent endomorphisms $\nil$ which are strictly compatible with both $F$ and $F'$ have the following form: each entry is either zero or an arbitrary element of $\qinvF$ (see \cref{compatible_matrices_pair}).

We anticipate that $\pr{M}{T,T'}$ is given by a polynomial in $q$ (so that $\prdual{M}{T,T'}$ is given by a rational function in $q$), which we present as an open problem (see \cref{probabilities_problem}). In this case, it would be desirable to have a combinatorial interpretation of $\pr{M}{T,T'}$. We also remark that the combinatorics of our $q$-Burge correspondence is different from that of the insertion-based probabilistic bijections of Matveev and Petrov \cite{matveev_petrov17}, because it does not admit {\itshape local growth rules}; see \cref{sec_growth_diagrams}.

We highlight two of the key results we use to prove parts \ref{cauchy_intro_bijection} and \ref{cauchy_intro_bijection_dual} of \cref{cauchy_intro}. The first is a formula for the monomial expansion of $W_\lambda(\mathbf{x};q)$ (see \cref{whittaker_expansion}):
\begin{thm}\label{whittaker_expansion_intro}
Let $\nil$ be a nilpotent endomorphism of $\qinvF^n$ whose (conjugated) Jordan form partition is $\lambda$. Then for any weak composition $(\alpha_1, \dots, \alpha_k)$ of $n$, the coefficient of $x_1^{\alpha_1}\cdots x_k^{\alpha_k}$ in $W_\lambda(\mathbf{x};q)$ equals $q^{\sum_{i\ge 1}\binom{\lambda_i}{2} - \sum_{i\ge 1}\binom{\alpha_i}{2}}$ times the number of partial flags
$$
0 = F_0 \subseteq F_1 \subseteq \cdots \subseteq F_k = \qinvF^n, \quad \text{ where } \dim(F_i) = \alpha_1 + \cdots + \alpha_i \text{ for all } 0 \le i \le k,
$$
which are strictly compatible with $\nil$.
\end{thm}

We give a recursive combinatorial proof of \cref{whittaker_expansion_intro}. As Sean Griffin explained to us, \cref{whittaker_expansion_intro} also follows from known results in the theory of Springer fibers, which are proved using perverse sheaves; see \cref{remark_perverse}. We note that there is a well-known formula for the {\itshape Hall--Littlewood functions} which is similar to \cref{whittaker_expansion_intro}, where one instead enumerates partial flags which are {\itshape weakly} compatible with $\nil$; see \cref{remark_hall-littlewood}. In addition to \cref{whittaker_expansion_intro}, we prove a dual version (see \cref{whittaker_expansion_dual}), where instead of fixing $\nil$ and enumerating strictly compatible partial flags $F$, we fix $F$ and enumerate $\nil$.

\begin{eg}\label{eg_whittaker_expansion_intro}
Let us find the coefficient of $x_1^2x_2^2$ in $W_{(3,1)}(\mathbf{x};q)$, by taking $\lambda := (3,1)$ and $\alpha := (2,2)$ in \cref{whittaker_expansion_intro}. We define $\nil$ to be the nilpotent endomorphism of $\qinvF^4$ given by
$$
\nil := \begin{bmatrix}
0 & 0 & 0 & 1 \\
0 & 0 & 0 & 0 \\
0 & 0 & 0 & 0 \\
0 & 0 & 0 & 0
\end{bmatrix},
$$
which has Jordan blocks of sizes $\transpose{\lambda} = (2,1,1)$.
We must enumerate partial flags
$$
0 = F_0 \subseteq F_1 \subseteq F_2 = \qinvF^4, \quad \text{ where } \dim(F_1) = 2,
$$
which are strictly compatible with $\nil$. The condition $\nil(F_2)\subseteq F_1$ means that $F_1$ contains $\unit{1}$, and the condition $\nil(F_1)\subseteq F_0$ means that $F_1 \subseteq \langle \unit{1}, \unit{2}, \unit{3}\rangle$. Therefore $F_1$ is the direct sum of $\langle \unit{1}\rangle$ and an arbitrary $1$-dimensional subspace of $\langle \unit{2}, \unit{3}\rangle$, of which there are $\frac{1}{q} + 1$ in total. Therefore the coefficient of $x_1^2x_2^2$ in $W_{(3,1)}(\mathbf{x};q)$ equals
\begin{gather*}
q^{\binom{3}{2} + \binom{1}{2} - \binom{2}{2} - \binom{2}{2}}(\textstyle\frac{1}{q} + 1) = 1+q.\qedhere
\end{gather*}

\end{eg}

The second key result is an explicit enumeration of each orbit of a standard parabolic subgroup (i.e.\ a group of invertible block upper-triangular matrices) acting on a partial flag variety. It can be equivalently stated as a result about double cosets of $\GL_n$ modulo two standard parabolic subgroups (see \cref{double_coset_rank}):
\begin{thm}\label{double_coset_enumeration_intro}
Let $M$ be a $k\times l$ nonnegative-integer matrix with row sums $\alpha$ and column sums $\beta$, whose entries sum to $n$. Then the size of the double coset of $\Par_\alpha\backslash\GL_n/\Par_\beta$ labeled by $M$ (in the sense of \cref{defn_double_coset}) over $\qinvF$ is
\begin{gather*}
q^{-n^2 + \sum_{1 \le i < i' \le k,\hspace*{1pt} 1 \le j < j' \le l}M_{i,j}M_{i',j'}}(1-q)^n\frac{\prod_{i=1}^k\qfac{\alpha_i}\prod_{j=1}^l\qfac{\beta_j}}{\prod_{1 \le i \le k,\hspace*{1pt} 1 \le j \le l}\qfac{M_{i,j}}}.
\end{gather*}

\end{thm}

Our proof of parts \ref{cauchy_intro_bijection} and \ref{cauchy_intro_bijection_dual} of \cref{cauchy_intro} is ultimately a double-counting of triples $(F,F',\nil)$ satisfying \eqref{flag_tableau_equation}, using \cref{whittaker_expansion_intro,double_coset_enumeration_intro}.

Finally, we apply the $q$-Burge correspondence to study socle filtrations of modules over the preprojective algebra of type $A$. Namely, let $Q(k,l)$ denote the quiver
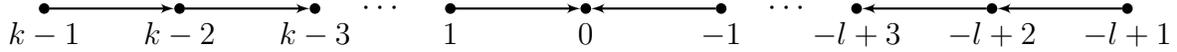
\vspace*{4pt}
$$
\hspace*{4pt}\begin{tikzpicture}[baseline=(current bounding box.center)]
\tikzstyle{out1}=[inner sep=0,minimum size=1.2mm,circle,draw=black,fill=black,semithick]
\pgfmathsetmacro{\s}{1.80};
\node[out1](v1)at(0*\s,0)[label={[below=3pt]\vphantom{$k-3$}$k-1$}]{};
\node[out1](v2)at(1*\s,0)[label={[below=3pt]\vphantom{$k-3$}$k-2$}]{};
\node[out1](v3)at(2*\s,0)[label={[below=3pt]\vphantom{$k-3$}$k-3$}]{};
\node[out1](v4)at(3*\s,0)[label={[below=3pt]\vphantom{$k-3$}$1$}]{};
\node[out1](v5)at(4*\s,0)[label={[below=3pt]\vphantom{$k-3$}$0$}]{};
\node[out1](v6)at(5*\s,0)[label={[below=3pt]\vphantom{$k-3$}$-1$}]{};
\node[out1](v7)at(6*\s,0)[label={[below=3pt]\vphantom{$k-3$}$-l+3$}]{};
\node[out1](v8)at(7*\s,0)[label={[below=3pt]\vphantom{$k-3$}$-l+2$}]{};
\node[out1](v9)at(8*\s,0)[label={[below=3pt]\vphantom{$k-3$}$-l+1$}]{};
\path[-latex',thick](v1)edge(v2) (v2)edge(v3) (v4)edge(v5) (v9)edge(v8) (v8)edge(v7) (v6)edge(v5);
\node[inner sep=0]at(2.5*\s,0){$\cdots$};
\node[inner sep=0]at(5.5*\s,0){$\cdots$};
\end{tikzpicture}\hspace*{4pt}
$$
of type $A_{k+l-1}$. Up to isomorphism, the finite-dimensional representations of $Q(k,l)$ such that the linear map associated to each arrow is injective are indexed by $k\times l$ nonnegative-integer matrices $M$.

A module over the preprojective algebra $\ppa{Q(k,l)}$ may be regarded as a representation of $Q(k,l)$ with an additional linear map associated to the reverse of each arrow, such that at every vertex $i$, the compositions of the maps associated to the two paths of length two from $i$ to itself are equal, up to a fixed sign. We consider such finite-dimensional modules where the linear map associated to each arrow of $Q(k,l)$ is injective. We show that up to isomorphism, every such module is given by a triple $(F,F',\nil)$ (which is not necessarily unique), where $(F,F')$ is a pair of partial flags in $\field^n$, and $\nil$ is a nilpotent endomorphism of $\field^n$ strictly compatible with both $F$ and $F'$:
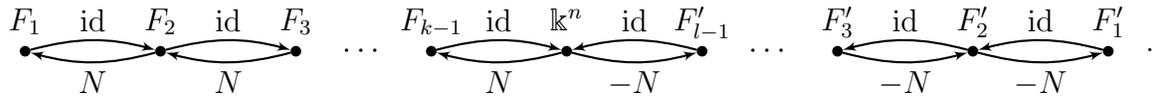
$$
\hspace*{4pt}\begin{tikzpicture}[baseline=(current bounding box.center)]
\tikzstyle{out1}=[inner sep=0,minimum size=1.2mm,circle,draw=black,fill=black,semithick]
\pgfmathsetmacro{\s}{1.80};
\node[out1](v1)at(0*\s,0)[label={[above=-2pt]\vphantom{$F'_{l-1}$}$F_1$}]{};
\node[out1](v2)at(1*\s,0)[label={[above=-2pt]\vphantom{$F'_{l-1}$}$F_2$}]{};
\node[out1](v3)at(2*\s,0)[label={[above=-2pt]\vphantom{$F'_{l-1}$}$F_3$}]{};
\node[out1](v4)at(3*\s,0)[label={[above=-2pt]\vphantom{$F'_{l-1}$}$F_{k-1}$}]{};
\node[out1](v5)at(4*\s,0)[label={[above=-2pt]\vphantom{$F'_{l-1}$}$\field^n$}]{};
\node[out1](v6)at(5*\s,0)[label={[above=-2pt]\vphantom{$F'_{l-1}$}$F'_{l-1}$}]{};
\node[out1](v7)at(6*\s,0)[label={[above=-2pt]\vphantom{$F'_{l-1}$}$F'_3$}]{};
\node[out1](v8)at(7*\s,0)[label={[above=-2pt]\vphantom{$F'_{l-1}$}$F'_2$}]{};
\node[out1](v9)at(8*\s,0)[label={[above=-2pt]\vphantom{$F'_{l-1}$}$F'_1$}]{};
\path[-latex',thick](v1)edge[bend left=15] node[above=-1pt]{$\id$}(v2) (v2)edge[bend left=15] node[below=-1pt]{$\nil$}(v1) (v2)edge[bend left=15] node[above=-1pt]{$\id$}(v3) (v3)edge[bend left=15] node[below=-1pt]{$\nil$}(v2) (v4)edge[bend left=15] node[above=-1pt]{$\id$}(v5) (v5)edge[bend left=15] node[below=-1pt]{$\nil$}(v4) (v9)edge[bend right=15] node[above=-1pt]{$\id$}(v8) (v8)edge[bend right=15] node[below=-1pt]{$-\nil$}(v9) (v8)edge[bend right=15] node[above=-1pt]{$\id$}(v7) (v7)edge[bend right=15] node[below=-1pt]{$-\nil$}(v8) (v6)edge[bend right=15] node[above=-1pt]{$\id$}(v5) (v5)edge[bend right=15] node[below=-1pt]{$-\nil$}(v6);
\node[inner sep=0]at(2.5*\s,0){$\cdots$};
\node[inner sep=0]at(5.5*\s,0){$\cdots$};
\end{tikzpicture}\hspace*{4pt}.
$$
We further show that the socle filtration of such a module (up to dimension) is equivalent to the data of the pair of tableaux $(\JF(\nil;F), \JF(\nil;F'))$ associated to $(F,F',\nil)$.

We then use the $q$-Burge correspondence to enumerate such $\ppa{Q(k,l)}$-modules. Since there are generally infinitely-many such modules, we count isomorphism classes $[V]$ of modules, each weighted by $\frac{1}{|\!\Aut(V)|}$. This is well-known in the literature as the natural way to enumerate objects in a category (cf.\ \cite{behrend93,bozec_schiffmann_vasserot20}). We can see this from the orbit-stabilizer theorem: if a group $G$ acts on a finite set $X$, then
$$
\frac{|X|}{|G|} = \sum_{[x]}\frac{1}{|\Stab(x)|},
$$
where the sum is over all equivalence classes $[x]$ of $X$ under the action of $G$, and $\Stab(x)$ is the stabilizer of $x$. If $X$ is a category, then the left-hand side above is generally ill-defined, but we can still define the right-hand side as a sum over all isomorphism classes $[x]$ of $X$, where we replace $\Stab(x)$ with the automorphism group $\Aut(x)$; this is known as the {\itshape orbifold volume} of $X$. For the modules we are considering, we obtain the following formula (see \cref{preprojective_enumeration}\ref{preprojective_enumeration_term}):
\begin{thm}\label{preprojective_enumeration_intro}
Set $\field := \qinvF$. Let $M$ be a $k\times l$ nonnegative-integer matrix with row sums $\alpha$ and column sums $\beta$ whose entries sum to $n$, and let $(T,T')$ be a pair of semistandard tableaux of the same shape $\lambda$ of size $n$ with contents $\alpha$ and $\beta$, respectively. Then
\begin{multline*}
\sum_{[V]}\frac{1}{|\!\Aut(V)|} = \frac{q^{n+\sum_{i\ge 1}\binom{\alpha_i}{2}+\sum_{i\ge 1}\binom{\beta_i}{2}}(1-q)^{-n}}{\prod_{1 \le i \le k,\hspace*{1pt} 1 \le j \le l}\qfac{M_{i,j}}}\hspace*{1pt}\pr{M}{T,T'} \\
= \frac{q^{n+\sum_{i\ge 1}\binom{\alpha_i}{2}+\sum_{i\ge 1}\binom{\beta_i}{2}}(1-q)^{-\lambda_1}}{\prod_{i\ge 1}\qfac{\lambda_i - \lambda_{i+1}}}\qwt(T)\qwt(T')\hspace*{2pt}\prdual{M}{T,T'},
\end{multline*}
where the sum is over all isomorphism classes $[V]$ of $\ppa{Q(k,l)}$-modules such that the linear map associated to each arrow of $Q(k,l)$ is injective, such that $V$ is indexed by $M$ as a representation of $Q(k,l)$ and has socle filtration corresponding to $(T,T')$.
\end{thm}

By summing the identity in \cref{preprojective_enumeration_intro} over all $M$ or over all $(T,T')$, we obtain less refined formulas; see \cref{preprojective_enumeration}. Our work develops a connection between the combinatorics of symmetric functions and the representation theory of preprojective algebras, which we expect to be part of a more general phenomenon worthy of further study; see \cref{dynkin_remark}. We also apply \cref{whittaker_expansion_intro} to enumerate the points of certain Nakajima quiver varieties of type $A$; see \cref{quiver_variety_enumeration}.

We note that the enumeration of quiver representations over finite fields dates back at least to work of Kac \cite{kac83}. He showed with the help of Richard Stanley that for every quiver $Q$, the number of isomorphism classes of absolutely indecomposable representations of $Q$ over $\qF$ is a polynomial in $q$, which Hausel, Letellier, and Rodriguez-Villegas \cite{hausel_letellier_rodriguez-villegas13} showed has nonnegative coefficients. We refer to the survey of Schiffmann \cite{schiffmann18} for an overview of the literature on Kac polynomials. We also mention that Reineke \cite{reineke06} showed that point counts over finite fields of moduli spaces of stable quiver representations are given by polynomials.

For enumerative results specifically about modules over the preprojective algebra $\Pi(Q)$ of a quiver $Q$, we refer to the detailed survey in Vernet's Ph.D.\ thesis \cite[Introduction]{vernet}. In particular, Crawley-Boevey and Van den Bergh \cite{crawley-boevey_van_den_bergh04} enumerated modules over a deformation of $\Pi(Q)$, and Mozgovoy \cite[Theorem 5.1]{mozgovoy} proved a generating-function identity relating counts of modules over $\Pi(Q)$ to the Kac polynomial of $Q$. A key novelty in our setting is that we are enumerating modules over $\Pi(Q)$ refined according to their socle filtration. It would be interesting to explore connections between our work and the papers cited above.

\subsection*{Outline} In \cref{sec_notation,sec_jordan_background}, we give background on partitions, tableaux, $q$-Whittaker functions, and nilpotent endomorphisms. In \cref{sec_pairs_partial} we introduce partial flags and relative position, and enumerate each orbit of a partial flag variety under the action of a standard parabolic subgroup. We also discuss applications to the study of double cosets of $\GL_n$ and of the symmetric group $\Sym_n$. In \cref{sec_nilpotent}, we prove a formula for the $q$-Whittaker function $W_\lambda(\mathbf{x};q)$ in terms of nilpotent endomorphisms and partial flags. In \cref{sec_q_burge}, we introduce our $q$-Burge correspondence, and show that it gives a probabilistic bijection proving the Cauchy identity for $q$-Whittaker functions. In \cref{sec_combinatorics}, we recall the classical Burge correspondence, and show that it is the $q\to 0$ limit of the $q$-Burge correspondence. We also discuss the combinatorial aspects of the $q$-Burge correspondence, including polynomiality and growth diagrams. In \cref{sec_quiver}, we introduce quiver representations and the preprojective algebra, and apply the $q$-Burge correspondence to enumerate isomorphism classes of $\ppa{Q(k,l)}$-modules. In \cref{sec_nakajima}, we similarly enumerate the points of certain Nakajima quiver varieties of type $A$.

\subsection*{Acknowledgments}
This work arose as part of a working group at LaCIM during the academic year 2019--2020. We thank the other participants of the working group, namely Fran\c{c}ois Bergeron, Benjamin Dequ\^{e}ne, Gabriel Frieden, Christophe Reutenauer, Franco Saliola, and Florian Schreier-Aigner, for many enlightening discussions. We are especially grateful to Gabriel Frieden for many helpful conversations and for providing the example presented in \cref{sec_growth_diagrams}. We thank Leonid Petrov for valuable input at many stages of this project. We also thank Dan Betea, Alexei Borodin, Sean Griffin, Joel Kamnitzer, Joel Lewis, and Alejandro Morales for valuable feedback, and an anonymous referee for several suggestions which improved the exposition.

\addtocontents{toc}{\protect\setcounter{tocdepth}{2}}
\section{Notation and background}\label{sec_notation}

\noindent In this section, we introduce some notation and background material. For further details, we refer to \cite{macdonald95,stanley12,stanley99}. We let $\mathbb{N}$ denote the set of nonnegative integers, and for $n\in\mathbb{N}$, we define $[n] := \{1, 2, \dots, n\}$. The symbol $\cong$ means ``is in bijection with''.

\subsection{Vector spaces}\label{sec_notation_vector}
Throughout the paper, vector spaces are defined over a fixed field $\field$. We will often specialize $\field$ to be $\qinvF$, the finite field with $1/q$ elements, where $1/q$ is a prime power. All vector spaces are assumed to be finite-dimensional.
\begin{defn}\label{defn_GL_gl}
Let $V$ be a vector space over $\field$. For $S\subseteq V$, we let $\spn{S}$ denote the linear span of $S$. Let $\GL(V)$ denote the group of automorphisms of $V$, and let $\End(V)$ denote the algebra of linear endomorphisms of $V$. We let $\GL(V)$ act on $\gl(V)$ by conjugation. In general, if $\GL(V)$ acts on the set $X$, we define the {\itshape stabilizer} of $x\in X$ by
$$
\Stab_{\GL(V)}(x) := \{g\in\GL(V) : g\cdot x = x\}.
$$
For $n\in\mathbb{N}$, we let $\GL_n$ and $\gl_n$ denote, respectively, $\GL(\field^n)$ and $\End(\field^n)$. We let $I_n\in\GL_n$ denote the $n\times n$ identity matrix. For $i\in [n]$, we let $\unit{i}\in\field^n$ denote the $i$th unit vector, whose $i$th entry is $1$ and whose other entries are $0$. 
\end{defn}

\subsection{Weak compositions and partitions}\label{sec_notation_partition}
We introduce weak compositions and partitions, and consider some associated functions.
\begin{defn}\label{defn_composition}
A {\itshape weak composition} is a finite sequence $\alpha = (\alpha_1, \dots, \alpha_l)$ of elements of $\mathbb{N}$, modulo trailing zeros (i.e.\ we consider $\alpha$ and $(\alpha_1, \dots, \alpha_l, 0)$ to be the same weak composition). For convenience, we set $\alpha_i := 0$ for $i > l$. The {\itshape size} of $\alpha$ is $\alpha_1 + \cdots + \alpha_l$, denoted $|\alpha|$. If $|\alpha| = n$, we write $\alpha\models n$. If $\alpha_1 \ge \cdots \ge \alpha_l$, we call $\alpha$ a {\itshape partition}, and write $\alpha\vdash n$.
\end{defn}

\begin{defn}\label{defn_young_diagram}
Given a partition $\lambda$, we define its {\itshape Young diagram} to be the diagram of left-justified boxes with $\lambda_i$ boxes in row $i$, for all $i\ge 1$; see \cref{eg_young_diagram}. Given another partition $\mu$, we write $\mu\subseteq\lambda$ if the Young diagram of $\mu$ is contained in the Young diagram of $\lambda$, or equivalently, if $\mu_i \le \lambda_i$ for all $i\ge 1$. In this case, we call $\lambda/\mu$ a {\itshape skew partition}, whose Young diagram is defined to be the set of boxes in the Young diagram of $\lambda$ which are not contained in the Young diagram of $\mu$. We call $\lambda/\mu$ a {\itshape horizontal strip} if its Young diagram contains at most one box in every column, or equivalently, if $\lambda_{i+1} \le \mu_i$ for all $i\ge 1$.

We define the {\itshape south border} of $\lambda$ to be the set of boxes $b$ in its Young diagram such that there are no boxes below $b$ in the same column. Note that $\lambda/\mu$ is a horizontal strip if and only if its Young diagram is contained in the south border of $\lambda$.
\end{defn}

\begin{eg}\label{eg_young_diagram}
Let $\lambda := (5,5,3,2)$ and $\mu := (4,4,1,1)$. Then $\mu\subseteq\lambda$, and the Young diagrams of $\lambda$, $\mu$, and $\lambda/\mu$ are, respectively,
$$
\hspace*{4pt}\,\ydiagram{5,5,3,2}\,\hspace*{4pt}, \quad \hspace*{4pt}\,\ydiagram{4,4,1,1}\,\hspace*{4pt}, \quad\text{ and }\quad \hspace*{4pt}\,
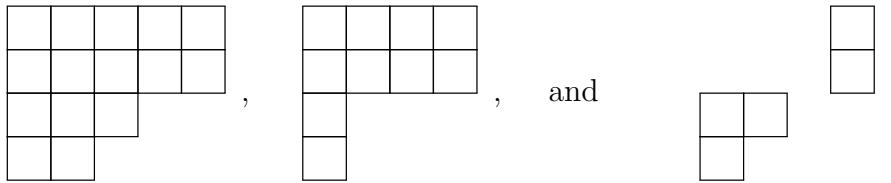
\begin{ytableau}
\none & \none & \none & \none & ~ \\
\none & \none & \none & \none & ~ \\
\none & ~ & ~ \\
\none & ~
\end{ytableau}\,\hspace*{4pt}.
$$
The south border of $\lambda = (5,5,3,2)$ consists of the shaded boxes below:
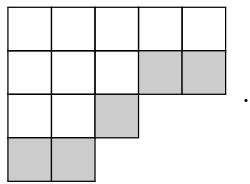
\begin{gather*}
\hspace*{4pt}\,\begin{ytableau}
{} & & & & \\
 & & & *({black!20!}) & *({black!20!}) \\
 & & *({black!20!}) \\
*({black!20!}) & *({black!20!})
\end{ytableau}\,\hspace*{4pt}.\qedhere
\end{gather*}

\end{eg}

\begin{defn}\label{defn_conjugate}
Given a partition $\lambda$, we define its {\itshape conjugate $\transpose{\lambda}$} to be the partition whose Young diagram is the transpose of the Young diagram of $\lambda$. That is, $(\transpose{\lambda})_i$ is the number of boxes in column $i$ of the Young diagram of $\lambda$, for all $i\ge 1$.
\end{defn}

\begin{eg}\label{eg_conjugate}
Let $\lambda := (5,5,3,2)$. Then its conjugate partition is $\transpose{\lambda} = (4,4,3,2,2)$.
\end{eg}

\begin{defn}\label{defn_n}
Let $\alpha = (\alpha_1, \dots, \alpha_l)$ be a weak composition. We define
\begin{gather*}
\ncon{\alpha} := \sum_{i=1}^l\binom{\alpha_i}{2} \quad \text{ and } \quad e_2(\alpha) := \sum_{1 \le i < j \le l}\alpha_i\alpha_j.
\end{gather*}

\end{defn}

\begin{eg}\label{eg_n}
Let $\lambda := (5,5,3,2)$. Then
\begin{gather*}
\ncon{\lambda} = 10 + 10 + 3 + 1 = 24 \quad \text{ and } \quad e_2(\lambda) = 25 + 15 + 10 + 15 + 10 + 6 = 81.\qedhere
\end{gather*}

\end{eg}

We remark that when $\lambda$ is a partition, the quantity $\ncon{\transpose{\lambda}}$ is what is denoted $n(\lambda)$ by Macdonald \cite[I.(1.6)]{macdonald95}. We now prove two identities which will be useful later.
\begin{lem}\label{n_formula}
For $\alpha\models n$, we have $\ncon{\alpha} + e_2(\alpha) = \binom{n}{2}$.
\end{lem}

\begin{proof}
This follows by writing $\binom{n}{2}$ as $\frac{(\alpha_1 + \cdots + \alpha_l)(\alpha_1 + \cdots + \alpha_l-1)}{2}$, and expanding.
\end{proof}

\begin{lem}\label{maximum_degree_recursion}
Let $\lambda/\mu$ be a skew partition. Then
\begin{gather*}
\ncon{\lambda} - \ncon{\mu} - \binom{|\lambda| - |\mu|}{2} =  \sum_{1 \le i \le j}(\lambda_i - \mu_i)(\mu_j - \lambda_{j+1}).
\end{gather*}

\end{lem}

\begin{proof}
Define the weak composition $\alpha := (\lambda_1 - \mu_1, \lambda_2 - \mu_2, \dots)$. Then by \cref{n_formula}, the left-hand side above equals $\ncon{\lambda} - \ncon{\mu} - \ncon{\alpha} - e_2(\alpha)$. We can expand this as
$$
\sum_{i\ge 1}\Bigg(\binom{\lambda_i}{2} - \binom{\mu_i}{2} - \binom{\lambda_i - \mu_i}{2} - \sum_{j\ge i+1}(\lambda_i - \mu_i)(\lambda_j - \mu_j)\Bigg).
$$
This simplifies to the right-hand side, using the fact that $\binom{\lambda_i}{2} - \binom{\mu_i}{2} - \binom{\lambda_i - \mu_i}{2} = (\lambda_i - \mu_i)\mu_i$.
\end{proof}

\begin{rmk}\label{identities_remark}
We note that \cref{n_formula,maximum_degree_recursion} hold more generally when the entries of $\alpha$, $\lambda$, and $\mu$ are regarded as indeterminates, because the associated identities only involve polynomials in the entries.
\end{rmk}

\subsection{Tableaux}\label{sec_notation_tableaux}
We introduce certain fillings of Young diagrams.

\begin{defn}\label{defn_tableau}
Let $\lambda\vdash n$. A {\itshape tableau $T$ of shape $\lambda$} is a filling of the boxes of the Young diagram of $\lambda$ with entries in $\mathbb{Z}_{>0}$. We call $n$ the {\itshape size} of $T$. We say that $\alpha\models n$ is the {\itshape content} of $T$ if the number of $i$'s in $T$ is $\alpha_i$, for all $i\ge 1$. We say that $T$ is a {\itshape semistandard Young tableau} (or just {\itshape semistandard}) if its entries weakly increase along rows (from left to right) and strictly increase along columns (from top to bottom). In this case, for $i\in\mathbb{N}$, we let $T^{(i)}$ denote the partition formed by the boxes containing the entries $1, \dots, i$ in $T$. Note that the skew partition $T^{(i)}/T^{(i-1)}$ is a horizontal strip for all $i\ge 1$. We let $\SSYT(\lambda)$ denote the set of all semistandard tableaux of shape $\lambda$, and for $\alpha\models n$, we let $\SSYT(\lambda,\alpha)$ denote the subset of $\SSYT(\lambda)$ of all tableaux with content $\alpha$. We call a tableau $T$ of size $n$ {\itshape standard} if $T$ is semistandard and has content $(1, \dots, 1)$, i.e., the entries of $T$ are $1, \dots, n$, with each appearing exactly once.
\end{defn}

\begin{eg}\label{eg_tableau}
Consider the tableaux
$$
T = \hspace*{4pt}\begin{ytableau}
2 & 4 & 4 & 5 & 7 \\
3 & 5 & 8 & 8 & 9 \\
7 & 7 & 9 \\
8 & 8
\end{ytableau}\hspace*{4pt} \quad \text{ and } \quad T' = \hspace*{4pt}\begin{ytableau}
2 & 4 & 5 & 5 & 7 \\
3 & 4 & 8 & 8 & 9 \\
7 & 7 & 9 \\
8 & 8
\end{ytableau}\hspace*{4pt},
$$
both of shape $(5,5,3,2)$ and content $(0,1,1,2,2,0,3,4,2)$. We can verify that $T$ is semistandard. Note that $T'$ is not semistandard, because its second column is not strictly increasing.
\end{eg}

\subsection{\texorpdfstring{$q$}{q}-analogues}\label{sec_notation_q}
We recall several $q$-analogues which will appear throughout the paper. Here $q$ is a formal variable, which we will often specialize to be a real number such that $1/q$ is a prime power.
\begin{defn}\label{defn_q-analogue}
Given $n\in\mathbb{N}$, we define the {\itshape $q$-number} and {\itshape $q$-factorial}
$$
\qn{n} := 1 + q + \cdots + q^{n-1} = \frac{1-q^n}{1-q} \quad \text{ and } \quad \qfac{n} := \qn{n}\qn{n-1}\cdots\qn{1}.
$$
For a weak composition $\alpha = (\alpha_1, \dots, \alpha_l)\models n$, we define the {\itshape $q$-multinomial coefficient}
$$
\qbinom{n}{\alpha} := \frac{\qfac{n}}{\qfac{\alpha_1}\qfac{\alpha_2}\cdots\qfac{\alpha_l}}.
$$
In the case that $\alpha = (k, n-k)$ for some $0 \le k \le n$, we obtain the {\itshape $q$-binomial coefficient}
$$
\qbinom{n}{k} := \frac{\qfac{n}}{\qfac{k}\qfac{n-k}}.
$$
We also define
$$
\qn{\infty} := (1-q)^{-1} \quad \text{ and } \quad \qbinom{\infty}{k} := \frac{(1-q)^{-k}}{\qfac{k}} \text{ for all } k\in\mathbb{N},
$$
which equal, respectively, $\lim_{n\to\infty}\qn{n}$ and $\lim_{n\to\infty}\qbinomsmall{n}{k}$, whenever $q\in\mathbb{C}$ with $|q| < 1$.

Note that when $q=1$, we have $\qn{n} = n$, $\qfac{n} = n!$, $\qbinomsmall{n}{\alpha} = \binom{n}{\alpha}$, and $\qbinomsmall{n}{k} = \binom{n}{k}$. On the other hand, when $q=0$, all the $q$-analogues above specialize to $1$, except for $\qn{0} = 0$.
\end{defn}

\begin{eg}\label{eg_q-analogue}
We have $\qn{3} = 1 + q + q^2$, $\qfac{3} = 1 + 2q + 2q^2 + q^3$, and
\begin{gather*}
\qbinom{4}{2} = \qbinom{4}{(2,2)} =  1 + q + 2q^2 + q^3 + q^4.\qedhere
\end{gather*}

\end{eg}

\begin{rmk}\label{remark_infinitely_many}
Throughout the paper, we will consider various identities of rational functions in $q$ (such as \eqref{q-cauchy_coefficient_intro}). To establish such an identity, it suffices to prove it whenever $1/q$ is a prime power, because every nonzero rational function has only finitely many zeros.
\end{rmk}

\subsection{\texorpdfstring{$q$}{q}-Whittaker functions}\label{sec_notation_whittaker}
We introduce the $q$-Whittaker symmetric functions $W_\lambda(\mathbf{x};q)$. They are $q$-analogues of the Schur symmetric functions $s_\lambda(\mathbf{x})$, and the $t=0$ specialization of the Macdonald polynomial $P_\lambda(\mathbf{x};q,t)$. See \cite{bergeron} for a survey of $q$-Whittaker functions. Throughout the paper, $\mathbf{x} = (x_1, x_2, \dots)$ denotes a family of commuting variables indexed by $\mathbb{Z}_{>0}$. For a weak composition $\alpha = (\alpha_1, \dots, \alpha_l)$, we define $\mathbf{x}^\alpha := x_1^{\alpha_1}x_2^{\alpha_2}\cdots x_l^{\alpha_l}$, and we let $[\mathbf{x}^\alpha]$ denote the operator which extracts the coefficient of $\mathbf{x}^\alpha$ from a formal power series in $\mathbf{x}$.
\begin{defn}[{cf.\ \cite[VI.(7.13')]{macdonald95}, \cite[Section 3.1]{aigner_frieden}}]\label{defn_whittaker}
Let $\lambda\vdash n$. Given $T\in\SSYT(\lambda)$, we define the {\itshape q-weight of $T$} as
$$
\qwt(T) := \prod_{i,j\ge 1}\qbinomdisplay{T^{(j)}_i - T^{(j)}_{i+1}}{T^{(j)}_i - T^{(j-1)}_i} \in \mathbb{N}[q].
$$
We define the {\itshape $q$-Whittaker function} indexed by $\lambda$ to be the formal power series in $\mathbf{x}$
$$
W_\lambda(\mathbf{x};q) := \sum_{T\in\SSYT(\lambda)}\qwt(T)\mathbf{x}^T,
$$
where $\mathbf{x}^T := \prod_{i\ge 1}x_i^{\#\text{$i$'s in $T$}}$. It is symmetric in the variables $\mathbf{x}$ \cite[Section VI.4]{macdonald95}.
\end{defn}

\begin{eg}\label{eg_whittaker_full}
\setlength{\tabcolsep}{3pt}Let $\lambda := (2,2)$. Then $W_\lambda(\mathbf{x};q)$ equals
\begin{gather*}
\begin{tabular}{cccccccccc}
$\displaystyle\sum_{1 \le i_1 < i_2}\hspace*{-4pt}$ & $x_{i_1}^2x_{i_2}^2$ & $+$ & $\displaystyle\sum_{1 \le i_1 < i_2 < i_3}\hspace*{-2pt}\Big(\hspace*{-6pt}$ & $(1+q)x_{i_1}^2x_{i_2}x_{i_3}$ & $+$ & $(1+q)x_{i_1}x_{i_2}^2x_{i_3}$ & $+$ & $(1+q)x_{i_1}x_{i_2}x_{i_3}^2$ & $\hspace*{-4pt}\Big)$ \\[14pt]
& $\ytableausmall{i_1 & i_1 \\ i_2 & i_2}$ & & & $\ytableausmall{i_1 & i_1 \\ i_2 & i_3}$ & & $\ytableausmall{i_1 & i_2 \\ i_2 & i_3}$ & & $\ytableausmall{i_1 & i_2 \\ i_3 & i_3}$
\end{tabular}\hspace*{36pt} \\[12pt]
\hspace*{108pt}\begin{tabular}{cccccc}
$+$ & $\displaystyle\sum_{1 \le i_1 < i_2 < i_3 < i_4}\hspace*{-2pt}\Big(\hspace*{-6pt}$ & $(1+q)^2x_{i_1}x_{i_2}x_{i_3}x_{i_4}$ & $+$ & $(1+q)x_{i_1}x_{i_2}x_{i_3}x_{i_4}$ & $\hspace*{-4pt}\Big).$ \\[14pt]
& & $\ytableausmall{i_1 & i_2 \\ i_3 & i_4}$ & & $\ytableausmall{i_1 & i_3 \\ i_2 & i_4}$ & 
\end{tabular}\qedhere
\end{gather*}
\setlength{\tabcolsep}{6pt}
\end{eg}

\begin{eg}\label{eg_whittaker}
Let us calculate the coefficient of $x_1^2x_2x_3^3$ in $W_\lambda(\mathbf{x};q)$, where $\lambda := (4,2)$. We have
$$
\qwt\Big(\ytableausmall{1 & 1 & 2 & 3 \\ 3 & 3}\Big) = \qbinom{3}{1}\qbinom{2}{1} = 1 + 2q + 2q^2 + q^3, \quad \qwt\Big(\ytableausmall{1 & 1 & 3 & 3 \\ 2 & 3}\Big) = \qbinom{2}{1} = 1 + q.
$$
Therefore
\begin{gather*}
[x_1^2x_2x_3^3]W_\lambda(\mathbf{x};q) = 2 + 3q + 2q^2 + q^3.\qedhere
\end{gather*}

\end{eg}

\begin{rmk}\label{macdonald_remark}
Given a partition $\lambda$, let $P_\lambda(\mathbf{x};q,t)$ and $\widetilde{H}_\lambda(\mathbf{x};q,t)$ denote the {\itshape Macdonald polynomial} and {\itshape modified Macdonald polynomial}, respectively, and let $\omega$ denote the standard involution on symmetric functions. Then we have \cite[Section 3]{bergeron}
$$
W_\lambda(\mathbf{x};q) = P_\lambda(\mathbf{x};q,0) = q^{\ncon{\lambda}}\omega(\widetilde{H}_\lambda(\mathbf{x};1/q,0)),
$$
where $\ncon{\lambda}$ is the degree of $W_\lambda(\mathbf{x};q)$ as a polynomial in $q$. In particular, we have the following specializations of $W_\lambda(\mathbf{x};q)$ \cite[(3.2)]{bergeron}:
$$
W_\lambda(\mathbf{x};0) = s_\lambda(\mathbf{x}), \quad W_\lambda(\mathbf{x};1) = e_{\transpose{\lambda}}(\mathbf{x}),
$$
where $s_\lambda(\mathbf{x})$ is the {\itshape Schur symmetric function} indexed by $\lambda$, and $e_{\transpose{\lambda}}(\mathbf{x})$ is the {\itshape elementary symmetric function} indexed by $\transpose{\lambda}$. The $q$-Whittaker functions are so named because they specialize to the {\itshape $\gl_n$-Whittaker functions} in the limit $q \to 1$ with a certain specialization of the variables $\mathbf{x}$; see \cite[Section 3]{gerasimov_lebedev_oblezin12} and \cite[Theorem 4.7]{borodin_corwin14}.
\end{rmk}

\begin{rmk}\label{remark_dual}
We mention that there exist the {\itshape dual $q$-Whittaker functions}, defined by
$$
\widehat{W}_\lambda(\mathbf{x};q) := \frac{(1-q)^{-\lambda_1}}{\prod_{i\ge 1}\qfac{\lambda_i - \lambda_{i+1}}}W_\lambda(\mathbf{x};q) \quad \text{ for all partitions } \lambda.
$$
The dual $q$-Whittaker functions are dual to the $q$-Whittaker functions with respect to a natural inner product \cite[Section VI]{macdonald95}, and we can write the right-hand side of the Cauchy identity \eqref{q-cauchy_equation_intro} as
$$
\sum_\lambda W_\lambda(\mathbf{x};q)\widehat{W}_\lambda(\mathbf{y};q).
$$
We will not use the dual $q$-Whittaker functions in the rest of the paper, and prefer to instead write out the term $\frac{(1-q)^{-\lambda_1}}{\prod_{i\ge 1}\qfac{\lambda_i - \lambda_{i+1}}}$ wherever necessary. However, we find it helpful to keep in mind that there is often a dual perspective to several topics we will discuss. For example, we have the following dual description of the $q$-weights of \cref{defn_whittaker}.
\end{rmk}

\begin{prop}[{Macdonald \cite[VI.(7.13)]{macdonald95}; cf.\ \cite[Section 3.1]{aigner_frieden}}]\label{whittaker_weight_dual}
Let $T\in\SSYT(\lambda)$. Then
$$
\frac{(1-q)^{-\lambda_1}}{\prod_{i\ge 1}\qfac{\lambda_i - \lambda_{i+1}}}\qwt(T) = \prod_{i,j\ge 1}\qbinomdisplay{T^{(j-1)}_{i-1} - T^{(j-1)}_i}{T^{(j)}_i - T^{(j-1)}_i},
$$
where when $i=1$, we take $T^{(j-1)}_{i-1} - T^{(j-1)}_i$ to be $\infty$.
\end{prop}

\begin{rmk}\label{weight_standard}
We point out that for a standard tableau $T$, the formulas for $\qwt(T)$ in \cref{defn_whittaker,whittaker_weight_dual} simplify as follows. Let $\lambda\vdash n$ be the shape of $T$, and for $1 \le i \le n$, let $r_i$ index the row of $T$ containing the entry $i$. Then
$$
\qwt(T) = \prod_{i=1}^n\qn{T^{(i)}_{r_i} - T^{(i)}_{r_i+1}} \quad \text{ and } \quad \frac{(1-q)^{-\lambda_1}}{\prod_{i\ge 1}\qfac{\lambda_i - \lambda_{i+1}}}\qwt(T) = \prod_{i=1}^n\qn{T^{(i-1)}_{r_i-1} - T^{(i-1)}_{r_i}},
$$
where when $r_i=1$ in the second product, we take $T^{(i-1)}_{r_i-1} - T^{(i-1)}_{r_i}$ to be $\infty$.
\end{rmk}

\section{Nilpotent endomorphisms and Jordan forms}\label{sec_jordan_background}

\noindent In this section, we provide some background on nilpotent endomorphisms and their Jordan forms. While this material is well-known, we will provide proofs of the results we need, both for completeness and for illustrating our combinatorial approach to studying nilpotent endomorphisms.
\begin{defn}\label{defn_JF}
For $n\in\mathbb{N}$, define the {\itshape Jordan block}
$$
J_n := \begin{bmatrix}
0 & 1 & 0 & \cdots \\
0 & 0 & 1 & \cdots \\
0 & 0 & 0 & \cdots \\
\vdots & \vdots & \vdots & \ddots
\end{bmatrix}  \in \gl_n,
$$
and for a partition $\lambda = (\lambda_1, \dots, \lambda_k)$, define the block-diagonal matrix
$$
J_\lambda := \begin{bmatrix}
J_{\lambda_1} & 0 & \cdots & 0 \\
0 & J_{\lambda_2} & \cdots & 0 \\
\vdots & \vdots & \ddots & \vdots \\
0 & 0 & \cdots & J_{\lambda_k}
\end{bmatrix} \in \gl_{|\lambda|}.
$$

Let $V$ be a (finite-dimensional) vector space. We say that an endomorphism $\nil\in\End(V)$ is {\itshape nilpotent} if $\nil^m = 0$ for some $m\in\mathbb{N}$. In this case, there exists a unique $\lambda\vdash\dim(V)$ such that $\nil$ is represented by the matrix $J_\lambda$ with respect to some basis of $V$ (see e.g.\ \cite[Section VII.7]{gantmacher59}). We call $\transpose{\lambda}$ the {\itshape (conjugated) Jordan form partition} of $\nil$, denoted $\JF(\nil)$.
\end{defn}

\begin{eg}
We have
\begin{gather*}
\JF\hspace*{-2pt}\left(\scalebox{0.8}{$\begin{bmatrix}0 & 0 & 0 \\ 0 & 0 & 0 \\ 0 & 0 & 0\end{bmatrix}$}\right) = \ydiagramsmall{3},\quad
\JF\hspace*{-2pt}\left(\scalebox{0.8}{$\begin{bmatrix}1 & -1 & 0 \\ 1 & -1 & 0 \\ 1 & -1 & 0\end{bmatrix}$}\right) = \ydiagramsmall{2,1},\quad
\JF\hspace*{-2pt}\left(\scalebox{0.8}{$\begin{bmatrix}0 & 1 & 0 \\ 0 & 0 & 0 \\ 1 & 0 & 0\end{bmatrix}$}\right) = \ydiagramsmall{1,1,1}.\hspace*{4pt}\qedhere
\end{gather*}

\end{eg}

We will need the following fact about the Jordan form:
\begin{lem}[{\cite[Section VII.7]{gantmacher59}}]\label{transitive_action}
Let $\lambda\vdash n$. Then $\GL_n$ acts transitively on the set of nilpotent endomorphisms $\nil\in\gl_n$ with $\JF(\nil) = \lambda$.
\end{lem}

We will find it convenient to think of a nilpotent endomorphism as acting on the boxes of the Young diagram of a partition.
\begin{defn}\label{defn_boxes}
Given $\lambda\vdash n$, let $\Boxes_\lambda$ be the $n$-dimensional vector space with a basis formed by the boxes of the Young diagram of $\lambda$. Let $\nil_\lambda\in\End(\Boxes_\lambda)$ denote the endomorphism which maps each box $b$ to the box immediately above it (if $b$ is not in the first row), or to zero (if $b$ is in the first row). Then $\nil_\lambda$ is nilpotent, and $\JF(\nil_\lambda) = \lambda$.
\end{defn}

\begin{eg}
Let $\lambda := (5,5,3,2)$. Then the action of $\nil_\lambda$ on $\Boxes_\lambda$ is given by the following picture:
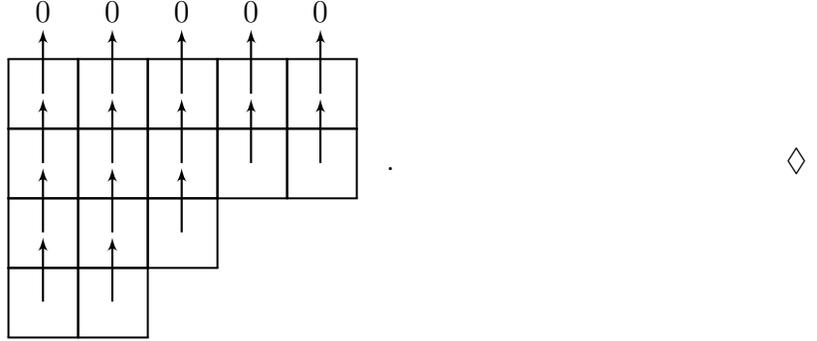
\begin{gather*}
\quad\begin{tikzpicture}[baseline=(current bounding box.center)]
\pgfmathsetmacro{\u}{0.922};
\useasboundingbox(0,0.9*\u)rectangle(5*\u,-4.1*\u);
\coordinate (vepsilon)at(0,0.075*\u);
\node[inner sep=0]at(0,0){\scalebox{1.6}{\begin{ytableau}
\none \\
\none \\
\none \\
\none \\
\none & \none & \none & \none & \none & & & & & \\
\none & \none & \none & \none & \none & & & & & \\
\none & \none & \none & \none & \none & & & \\
\none & \none & \none & \none & \none & &
\end{ytableau}}};
\foreach \y in {1,...,4}{
\path[thick,-latex']($(1*\u,-\y*\u)+(-0.5*\u,0.5*\u)$)edge($(1*\u,-\y*\u+\u)+(-0.5*\u,0.5*\u)-(vepsilon)$);}
\foreach \y in {1,...,4}{
\path[thick,-latex']($(2*\u,-\y*\u)+(-0.5*\u,0.5*\u)$)edge($(2*\u,-\y*\u+\u)+(-0.5*\u,0.5*\u)-(vepsilon)$);}
\foreach \y in {1,...,3}{
\path[thick,-latex']($(3*\u,-\y*\u)+(-0.5*\u,0.5*\u)$)edge($(3*\u,-\y*\u+\u)+(-0.5*\u,0.5*\u)-(vepsilon)$);}
\foreach \y in {1,...,2}{
\path[thick,-latex']($(4*\u,-\y*\u)+(-0.5*\u,0.5*\u)$)edge($(4*\u,-\y*\u+\u)+(-0.5*\u,0.5*\u)-(vepsilon)$);}
\foreach \y in {1,...,2}{
\path[thick,-latex']($(5*\u,-\y*\u)+(-0.5*\u,0.5*\u)$)edge($(5*\u,-\y*\u+\u)+(-0.5*\u,0.5*\u)-(vepsilon)$);}
\foreach \x in {1,...,5}{
\node[inner sep=0]at($(\x*\u-0.5*\u,0.68*\u)$){$0$};}
\end{tikzpicture}\quad.\qedhere
\end{gather*}

\end{eg}

Note that if $\nil\in\End(V)$ is nilpotent and $\JF(\nil) = \lambda$, then there exists an isomorphism $V\to \Boxes_\lambda$ which takes $\nil$ to $\nil_\lambda$. Therefore in studying nilpotent endomorphisms with conjugated Jordan form $\lambda$, it suffices to work with $\nil_\lambda$. We now prove several properties about nilpotent endomorphisms and their Jordan forms.
\begin{lem}\label{matrix_to_partition}
Let $\nil\in\End(V)$ be nilpotent, and set $\lambda := \JF(\nil)$. Then
$$
\lambda_1 + \cdots + \lambda_i = \dim(\ker(\nil^i)) = \dim(V) - \rank(\nil^i) \quad \text{ for all } i \in \mathbb{N}.
$$
Therefore
\begin{gather*}
\lambda_i = \dim(\ker(\nil^i)/\ker(\nil^{i-1})) = \dim(\nil^{i-1}(V)/\nil^i(V)) \quad \text{ for all } i \ge 1.
\end{gather*}

\end{lem}

\begin{proof}
We may assume that $V = \Boxes_\lambda$ and $\nil = \nil_\lambda$. Then for $i\in\mathbb{N}$, the boxes in rows $1, \dots, i$ of $\lambda$ form a basis for $\ker(\nil^i)$, so $\lambda_1 + \cdots + \lambda_i = \dim(\ker(\nil^i))$. The remaining results follow.
\end{proof}

In particular, if $\nil\in\End(V)$ is nilpotent, then $\nil^{\dim(V)} = 0$.
\begin{lem}\label{flag_to_tableau}
Let $\nil\in\End(V)$ be nilpotent, and let $W$ be a subspace of $V$ such that $\nil(V)\subseteq W$. Let $\lambda := \JF(\nil)$ and $\mu := \JF(\nil|_W)$. Then $\mu\subseteq\lambda$ and $\lambda/\mu$ is a horizontal strip.
\end{lem}

\begin{proof}
We must show that
$$
\lambda_i \ge \mu_i \ge \lambda_{i+1} \quad \text{ for all } i \ge 1.
$$
We will use the descriptions of $\lambda$ and $\mu$ given in \cref{matrix_to_partition}.

First we show that $\lambda_i \ge \mu_i$. Let $B\subseteq W$ be a set of $\mu_i$ vectors which form a basis of $\ker((\nil|_W)^i)/\ker((\nil|_W)^{i-1})$. Then $\spn{B}\cap\ker((\nil|_W)^{i-1}) = 0$. But $\spn{B}\subseteq W$, so $\spn{B}\cap\ker(\nil^{i-1}) = 0$. Therefore $B$ is linearly independent in $\ker(\nil^i)/\ker(\nil^{i-1})$, so $\lambda_i \ge \mu_i$.

Now we show that $\mu_i\ge\lambda_{i+1}$. We may assume that $V = \Boxes_\lambda$ and $\nil = \nil_\lambda$. Let $A$ be the set of boxes in row $i$ of the Young diagram of $\lambda$ not in the south border. We have $|A| = \lambda_{i+1}$ and $A\subseteq\nil(V)\subseteq W$. Since the set of all boxes in row $i$ forms a basis of $\ker(\nil^i)/\ker(\nil^{i-1})$, we see that $A$ is linearly independent in $\ker((\nil|_W)^i)/\ker((\nil|_W)^{i-1})$. Therefore $\lambda_{i+1} \le \mu_i$.
\end{proof}

We recall that the symbol $\cong$ means ``is in bijection with''.
\begin{prop}[{cf.\ \cite[Section II.1]{macdonald95}, \cite[Theorem 1.10.7]{stanley12}}]\label{jordan_stabilizer}
Let $\lambda\vdash n$, and let $\nil\in\gl_n$ be nilpotent with $\JF(\nil) = \lambda$. Then
\begin{gather*}
\Stab_{\GL_n}\hspace*{-1pt}(\nil) = \{g\in\GL_n : g\nil g^{-1} = \nil\} \cong \prod_{i\ge 1}\field^{(\lambda_1 + \cdots + \lambda_{i-1} + \lambda_{i+1})(\lambda_i - \lambda_{i+1})} \times \GL_{\lambda_i - \lambda_{i+1}}.
\end{gather*}

\end{prop}

\begin{proof}
We may assume that $V = \Boxes_\lambda$ and $\nil = \nil_\lambda$. Note that an arbitrary $g\in\End(V)$ is uniquely determined by where it sends each box of the Young diagram of $\lambda$. We depict this by labeling each box $b$ by $g(b)$. Then $g\nil = \nil g$ if and only if $\nil$ acts on the labeled diagram just as in \cref{defn_boxes}; that is, for each box $b$, $\nil$ maps $g(b)$ to $g(b')$ (if $b$ is not in the first row and $b'$ is above $b$), or to zero (if $b$ is in the first row): 
$$
\quad\begin{tikzpicture}[baseline=(current bounding box.center)]
\pgfmathsetmacro{\u}{0.922};
\useasboundingbox(0,0*\u)rectangle(1*\u,-2*\u);
\node[inner sep=0]at(0,0){\scalebox{1.6}{\begin{ytableau}
\none \\
\none \\
\none & \\
\none &
\end{ytableau}}};
\path[thick,-latex']($(0.5*\u,-1.25*\u)$)edge($(0.5*\u,-0.7*\u)$);
\node[inner sep=0]at($(0.5*\u,-1.5*\u)$){$g(b)$};
\node[inner sep=0]at($(0.5*\u,-0.5*\u)$){$g(b')$};
\end{tikzpicture}\qquad,\qquad
\begin{tikzpicture}[baseline=(current bounding box.center)]
\pgfmathsetmacro{\u}{0.922};
\useasboundingbox(0,-2*\u)rectangle(1*\u,-0*\u);
\node[inner sep=0]at(0,0){\scalebox{1.6}{\begin{ytableau}
\none \\
\none \\
\none \\
\none &
\end{ytableau}}};
\path[thick,-latex']($(0.5*\u,-1.25*\u)$)edge($(0.5*\u,-0.7*\u)$);
\node[inner sep=0]at($(0.5*\u,-1.5*\u)$){$g(b)$};
\node[inner sep=0]at($(0.5*\u,-0.5*\u)$){$0$};
\end{tikzpicture}\quad.
$$
It suffices to retain only the labels in the south border of the Young diagram of $\lambda$; the labels of the other boxes are uniquely determined by applying powers of $\nil$. For $i\ge 1$, let $A_i$ be the set of boxes in rows $1, \dots, i$, and let $B_i$ be the set of boxes in row $i$ in the south border. Note that $|A_i| = \lambda_1 + \cdots + \lambda_i$ and $|B_i| = \lambda_i - \lambda_{i+1}$. Then $g\in\Stab_{\GL_n}\hspace*{-1pt}(\nil)$ if and only if for all $i\ge 1$, the labels of the boxes $B_i$ form a basis of $\spn{A_i}/\spn{A_i\setminus B_i}$. We obtain these labels precisely by applying an arbitrary element of $\GL(\spn{B_i})$ to the boxes $B_i$, and then adding to each box in $B_i$ an arbitrary linear combination of $A_i\setminus B_i$.
\end{proof}

\section{Pairs of partial flags and relative position}\label{sec_pairs_partial}

\noindent In this section, we introduce partial flag varieties, and study the action of a parabolic subgroup $\Par_\alpha$ of $\GL_n$ on a partial flag variety $\PFl_\beta$. The orbits turn out to correspond to double cosets $\Par_\alpha\backslash\GL_n/\Par_\beta$, as well as to the notion of relative position between a pair of partial flags. The associated combinatorics is that of the double cosets $\Sym_\alpha\backslash\Sym_n/\Sym_\beta$ of the symmetric group $\Sym_n$ modulo two parabolic subgroups, or equivalently, nonnegative-integer matrices with prescribed row and column sums (also known as contingency tables). For background we refer to \cite{brion05} for partial flag varieties, to \cite[Section IV.2.5]{bourbaki68} for double cosets of $\GL_n$, and to \cite{diaconis_gangolli95,billey_konvalinka_petersen_slofstra_tenner18} and the references therein for double cosets of $\Sym_n$.

\subsection{Partial flag varieties}\label{sec_partial_flags}
We introduce parabolic subgroups of $\GL_n$ and partial flag varieties.
\begin{defn}\label{defn_composition_block}
Let $\alpha = (\alpha_1, \dots, \alpha_k) \models n$. For $1 \le i \le k$, we define the {\itshape $i$th block} of $\alpha$ as the interval of integers
\begin{gather*}
\block{\alpha}{i} := \{\alpha_1 + \cdots + \alpha_{i-1} + 1, \dots, \alpha_1 + \cdots + \alpha_i\}.
\end{gather*}

\end{defn}

\begin{eg}\label{eg_composition_block}
Let $\alpha := (3,2)\models 5$. Then the blocks of $\alpha$ are $\block{\alpha}{1} = \{1,2,3\}$ and $\block{\alpha}{2} = \{4,5\}$.
\end{eg}

\begin{defn}\label{defn_lie}
Let $\alpha = (\alpha_1, \dots, \alpha_k)\models n$. We let $\Par_\alpha$ denote the subgroup of $\GL_n$ of matrices such that every entry with row index in $\block{\alpha}{i}$ and column index in $\block{\alpha}{j}$ is zero, for all $1 \le j < i \le k$. That is, $\Par_\alpha$ consists of all invertible block-upper-triangular matrices with diagonal blocks of sizes $\alpha_1, \dots, \alpha_k$. We also let $\Nil_\alpha$ denote the subspace of $\gl_n$ of matrices such that every entry with row index in $\block{\alpha}{i}$ and column index in $\block{\alpha}{j}$ is zero, for all $1 \le j \le i \le k$. That is, $\Nil_\alpha$ consists of all block-upper-triangular matrices with diagonal blocks of sizes $\block{\alpha}{1}, \dots, \block{\alpha}{k}$, where all diagonal blocks are zero.

In the case that $\alpha = (1, \dots, 1)$, we denote $\Par_\alpha$ and $\Nil_\alpha$ by $\Bor_n$ and $\Nil_n$, respectively. That is, $\Bor_n$ is the group of all $n\times n$ invertible upper-triangular matrices, and $\Nil_n$ is the algebra of all $n\times n$ strictly upper-triangular matrices.
\end{defn}

\begin{eg}\label{eg_lie}
Let $\alpha := (3,2)\models 5$. Then
$$
\Par_\alpha = \left\{g = \begin{bmatrix}
\ast & \ast & \ast & \ast & \ast \\
\ast & \ast & \ast & \ast & \ast \\
\ast & \ast & \ast & \ast & \ast \\
0 & 0 & 0 & \ast & \ast \\
0 & 0 & 0 & \ast & \ast
\end{bmatrix} : g \text{ is invertible}\right\} \quad \text{ and } \quad \Nil_\alpha = \left\{\begin{bmatrix}
0 & 0 & 0 & \ast & \ast \\
0 & 0 & 0 & \ast & \ast \\
0 & 0 & 0 & \ast & \ast \\
0 & 0 & 0 & 0 & 0 \\
0 & 0 & 0 & 0 & 0
\end{bmatrix}\right\},
$$
where $\ast$ denotes an arbitrary element of $\field$.
\end{eg}

\begin{defn}\label{defn_partial_flag}
Let $\alpha = (\alpha_1, \dots, \alpha_k)\models n$. Define the {\itshape partial flag variety} $\PFl_\alpha$ as the set of tuples $F = (F_0, F_1, \dots, F_k)$ of subspaces of $\field^n$, such that
$$
F_0 \subseteq F_1 \subseteq \cdots \subseteq F_k \quad \text{ and } \quad \dim(F_i) = \alpha_1 + \cdots + \alpha_i \text{ for all } 0 \le i \le k.
$$
We say that $F$ is {\itshape represented} by the matrix $g\in\GL_n$ if for all $0 \le i \le k$, the subspace $F_i$ is spanned by the first $\alpha_1 + \cdots + \alpha_i$ columns of $g$. Note that $\GL_n$ acts transitively on $\PFl_\alpha$ by left multiplication, and the stabilizer of the partial flag represented by $I_n$ is $\Par_\alpha$. Therefore we may regard $\PFl_\alpha$ as the quotient $\GL_n/\Par_\alpha$.

We mention two special cases of $\PFl_\alpha$ which are of particular interest. When $\alpha = (1, \dots, 1)$, we obtain the {\itshape complete flag variety}, denoted $\Fl_n$. When $\alpha = (k,n-k)$ for some $0 \le k \le n$, we obtain the {\itshape Grassmannian} of all $k$-dimensional subspaces of $\field^n$, denoted $\Gr_{k,n}$.
\end{defn}

\begin{eg}\label{eg_partial_flag}
Let $g := \scalebox{0.8}{$\begin{bmatrix}2 & 5 & -1 \\ -3 & 3 & 4 \\ 1 & -4 & -2\end{bmatrix}$}\in\GL_3$. Then $g$ represents the flag $F = (F_0, F_1, F_2, F_3)\in\Fl_{(1,1,1)} = \Fl_3$, where
$$
F_0 = 0, \quad F_1 = \bigspn{\scalebox{0.8}{$\begin{bmatrix}2 \\ -3 \\ 1\end{bmatrix}$}}, \quad F_2 = \bigspn{\scalebox{0.8}{$\begin{bmatrix}2 \\ -3 \\ 1\end{bmatrix}$}, \scalebox{0.8}{$\begin{bmatrix}5 \\ 3 \\ -4\end{bmatrix}$}}, \quad F_3 = \bigspn{\scalebox{0.8}{$\begin{bmatrix}2 \\ -3 \\ 1\end{bmatrix}$}, \scalebox{0.8}{$\begin{bmatrix}5 \\ 3 \\ -4\end{bmatrix}$}, \scalebox{0.8}{$\begin{bmatrix}-1 \\ 4 \\ -2\end{bmatrix}$}} = \field^3.
$$
Also, $g$ represents the flag $F' = (F'_0,F'_1,F'_2)\in\Fl_{(2,1)}$, where
$$
F'_0 = 0, \quad F'_1 = \bigspn{\scalebox{0.8}{$\begin{bmatrix}2 \\ -3 \\ 1\end{bmatrix}$}, \scalebox{0.8}{$\begin{bmatrix}5 \\ 3 \\ -4\end{bmatrix}$}}, \quad F'_2 = \bigspn{\scalebox{0.8}{$\begin{bmatrix}2 \\ -3 \\ 1\end{bmatrix}$}, \scalebox{0.8}{$\begin{bmatrix}5 \\ 3 \\ -4\end{bmatrix}$}, \scalebox{0.8}{$\begin{bmatrix}-1 \\ 4 \\ -2\end{bmatrix}$}} = \field^3.
$$
If we regard $\Fl_{(2,1)}$ as the Grassmannian $\Gr_{2,3}$ of all $2$-dimensional subspaces of $\field^3$, then $g$ represents the subspace $F'_1$.
\end{eg}

We recall the enumeration over finite fields of various spaces introduced so far. We point out that $\qinvn{n} = q^{1-n}\qn{n}$ for all $n\in\mathbb{N}$.
\begin{prop}[{\cite[Propositions 1.10.1 and 1.7.2]{stanley12}; cf.\ \cite[Section 1.3]{kirillov95}}]\label{q_enumeration}
Set $\field := \qinvF$.
\begin{enumerate}[label=(\roman*), leftmargin=*, itemsep=2pt]
\item\label{q_enumeration_GL} For $n\in\mathbb{N}$, we have $|\GL_n| = \prod_{i=1}^n(q^{-n} - q^{-i+1}) = q^{-n^2}(1-q)^n\qfac{n}$.\vspace*{4pt}
\item\label{q_enumeration_Gr} For $0 \le k \le n$, we have $|\Gr_{k,n}| = \qinvbinomsmall{n}{k} = q^{-k(n-k)}\qbinomsmall{n}{k}$.
\item\label{q_enumeration_matrices} For $0 \le k \le \min(m,n)$, the number of $m\times n$ matrices of rank $k$ is
$$
\frac{q^{-k(m+n-k)}(1-q)^k\qfac{m}\qfac{n}}{\qfac{k}\qfac{m-k}\qfac{n-k}}.
$$
\item\label{q_enumeration_P} For $\alpha\models n$, we have $|\PFl_\alpha| = \qinvbinomsmall{n}{\alpha} = q^{-e_2(\alpha)}\qbinomsmall{n}{\alpha}$.
\end{enumerate}

\end{prop}

\begin{proof}
Part \ref{q_enumeration_GL} follows from \cite[Proposition 1.10.1]{stanley12}, and part \ref{q_enumeration_Gr} follows from \cite[Proposition 1.7.2]{stanley12}. For part \ref{q_enumeration_matrices}, note that the row span of an $m\times n$ matrix is an arbitrary subspace of $\field^n$ of rank at most $m$, of which there are $|\Gr_{k,n}|$ subspaces of rank $k$. Therefore the number of $m\times n$ matrices of rank $k$ is equal to $|\Gr_{k,n}|$ times the number of $m\times k$ matrices of rank $k$. Now observe that $\GL_k$ acts freely by right multiplication on the set of all $m\times k$ matrices of rank $k$, and the quotient can be identified with $\Gr_{k,m}$, so the number of such matrices is $|\Gr_{k,m}|\cdot |\GL_k|$. The product $|\Gr_{k,n}|\cdot |\Gr_{k,m}|\cdot |\GL_k|$ gives the desired expression. For part \ref{q_enumeration_P}, the first equality follows from part \ref{q_enumeration_Gr}, using the fact that if $\alpha = (\alpha_1, \dots, \alpha_k)$, we have
\begin{align}\label{partial_flag_induction_equation}
\PFl_\alpha \cong \Gr_{\alpha_1, n} \times \PFl_{(\alpha_2, \dots, \alpha_k)}.
\end{align}
The second equality follows by \cref{n_formula}.
\end{proof}

\subsection{Schubert decomposition}\label{sec_schubert_decomposition}
We recall the Schubert decomposition of a partial flag variety. We refer to \cite{brion05} for further background on the Schubert decomposition of flag varieties of $\GL_n$, and to \cite{bjorner_brenti05} for background on the symmetric group and its parabolic subgroups. We begin by introducing the symmetric group $\Sym_n$.
\begin{defn}\label{defn_inversion}
Given $n\in\mathbb{N}$, we let $\Sym_n$ denote the group of permutations of $[n]$, with identity $\id$. Let $w\in\Sym_n$. An {\itshape inversion} of $w$ is a pair $(i,j) \in [n]\times [n]$ such that $i < j$ and $w(i) > w(j)$. We define the {\itshape length} $\ell(w)$ of $w$ as the number of inversions of $w$.
\end{defn}

\begin{eg}\label{eg_inversion}
Let $w := 23514\in\Sym_5$. The inversions of $w$ are $(1,4)$, $(2,4)$, $(3,4)$, and $(3,5)$, so $\ell(w) = 4$.
\end{eg}

\begin{defn}[{\cite[Section 2.4]{bjorner_brenti05}}]\label{defn_parabolic_quotient}
Given $\alpha = (\alpha_1, \dots, \alpha_k) \models n$, we let $\Sym_\alpha$ denote the {\itshape parabolic subgroup} of $\Sym_n$ of all permutations which setwise preserve $\block{\alpha}{i}$ for all $1 \le i \le k$. Note that $\Sym_\alpha \cong \Sym_{\alpha_1} \times \cdots \times \Sym_{\alpha_k}$. Each coset of $\Sym_n/\Sym_\alpha$ contains a unique element of minimal length; we denote the set of minimal-length coset representatives by $\Sym^\alpha$. For $w\in\Sym_n$, we have $w\in\Sym^\alpha$ if and only if for all $1 \le i \le k$, we have
\begin{gather*}
w(a) < w(b) \quad \text{ for all } a < b \text{ in } \block{\alpha}{i}.
\end{gather*}

\end{defn}

\begin{eg}\label{eg_parabolic_quotient}
Let $\alpha := (3,2)\models 5$. Then for $w\in\Sym_5$, we have $w\in\Sym^\alpha$ if and only if
$$
w(1) < w(2) < w(3) \quad \text{ and } \quad w(4) < w(5).
$$
The permutation $52341$ is not in $\Sym^\alpha$; its minimal-length coset representative modulo the right action of $\Sym_\alpha$ is $23514\in\Sym^\alpha$.
\end{eg}

We now recall the Schubert decomposition of a partial flag variety.
\begin{defn}\label{defn_schubert}
Given $w\in\Sym_n$, we define the {\itshape permutation matrix} $\mathring{w}\in\GL_n$ by $\mathring{w}_{i,j} := \delta_{i,w(j)}$ for all $1 \le i,j \le n$. Equivalently, $\mathring{w}\unit{j} = \unit{w(j)}$ for all $1 \le j \le n$. We obtain in this way a group representation of $\Sym_n$; that is, for all $v,w\in\Sym_n$, the permutation matrix of $vw$ is $\mathring{v}\mathring{w}$.

Given $\alpha = (\alpha_1, \dots, \alpha_k)\models n$ and $w\in\Sym_n$, let $\Unit{w}\in\PFl_\alpha$ denote the flag represented by $\mathring{w}$. That is,
$$
\Unit{w}_i = \spn{\unit{w(1)}, \dots, \unit{w(\alpha_1 + \cdots + \alpha_i)}} \quad \text{ for all } 0 \le i \le k.
$$
Note that $\Unit{w} = \mathring{w}\Unit{\id}$. Let $X_w := \Bor_n\Unit{w}\subseteq\PFl_\alpha$ be the {\itshape Schubert cell} associated to $w$. Note that $\Unit{w}$ and $X_w$ only depend on $w$ modulo the right action of $\Sym_\alpha$, so it suffices to restrict to $w\in\Sym^\alpha$. We emphasize that the definitions of $\Unit{w}$ and $X_w$ depend on $\alpha$.
\end{defn}

\begin{prop}[Schubert decomposition {\cite[Section 1.2]{brion05}}]\label{schubert_decomposition}
Let $\alpha\models n$.
\begin{enumerate}[label=(\roman*), leftmargin=*]
\item\label{schubert_decomposition_matrix} For $w\in\Sym^\alpha$ and $F\in X_w$, there exists a unique $g\in\GL_n$ representing $F$ of the following form: $g$ is obtained from $\mathring{w}$ by replacing each $0$ which does not lie below a $1$ in the same column, nor to the right of a $1$ in the same row, by an element of $\field$. In particular, we have $X_w\cong\field^{\ell(w)}$.
\item\label{schubert_decomposition_decomposition} We have the decomposition into $\Bor_n$-orbits
\begin{gather*}
\PFl_\alpha = \bigsqcup_{w\in\Sym^\alpha}X_w.
\end{gather*}
\end{enumerate}

\end{prop}

\begin{eg}\label{eg_schubert}
Let $\alpha := (3,2)\models 5$ and $w := 23514\in\Sym^\alpha$. Then
\begin{gather*}
\mathring{w} = \begin{bmatrix}
0 & 0 & 0 & 1 & 0 \\
1 & 0 & 0 & 0 & 0 \\
0 & 1 & 0 & 0 & 0 \\
0 & 0 & 0 & 0 & 1 \\
0 & 0 & 1 & 0 & 0
\end{bmatrix}, \quad \text{ and } \quad
X_w = \left\{\begin{bmatrix}
a & b & c & 1 & 0 \\
1 & 0 & 0 & 0 & 0 \\
0 & 1 & 0 & 0 & 0 \\
0 & 0 & d & 0 & 1 \\
0 & 0 & 1 & 0 & 0
\end{bmatrix} : a,b,c,d\in\field\right\}\subseteq\PFl_\alpha.\qedhere
\end{gather*}

\end{eg}

\subsection{Parabolic subgroup decomposition}\label{sec_parabolic_decomposition}
We study the orbits of a parabolic subgroup $\Par_\alpha$ acting on a partial flag variety $\PFl_\beta$. We will label the orbits by nonnegative-integer matrices, and express each orbit as a disjoint union of Schubert cells of $\PFl_\beta$.

\begin{defn}\label{defn_nonnegative_integer_matrix}
Let $M\in\mathbb{N}^{k\times l}$ be a $k\times l$ matrix with entries in $\mathbb{N}$. Let $\rowsum(M)$ denote the weak composition $\alpha = (\alpha_1, \dots, \alpha_k)$ of row sums of $M$, i.e.,
$$
\alpha_i := \sum_{j=1}^l M_{i,j} \quad \text{ for } 1 \le i \le k.
$$
Similarly, let $\colsum(M)$ denote the weak composition $\beta = (\beta_1, \dots, \beta_l)$ of column sums of $M$, i.e.,
$$
\beta_j := \sum_{i=1}^k M_{i,j} \quad \text{ for } 1 \le j \le l.
$$
We call $(\alpha, \beta)$ the {\itshape type} of $M$, denoted $\type(M)$. We let $\Mats{\alpha}{\beta}$ denote the set of all nonnegative-integer matrices of type $(\alpha,\beta)$.
\end{defn}

\begin{eg}\label{eg_nonnegative_integer_matrix}
Let $M := \scalebox{0.8}{$\begin{bmatrix}1 & 1 \\ 2 & 1\end{bmatrix}$}$. Then $\rowsum(M) = (2,3)$ and $\colsum(M) = (3,2)$, so $\type(M) = ((2,3), (3,2))$ and $M\in\Mats{(2,3)}{(3,2)}$.
\end{eg}

\begin{defn}\label{defn_permutation_to_matrix}
Let $\alpha = (\alpha_1, \dots, \alpha_k), \beta = (\beta_1, \dots, \beta_l)\models n$. Given $w\in\Sym_n$, define the matrix $\mat{w}{\alpha}{\beta}\in\mathbb{N}^{k\times l}$ by
$$
\mat{w}{\alpha}{\beta}_{i,j} := \sum_{a\in\block{\alpha}{i},\hspace*{2pt} b\in\block{\beta}{j}}\mathring{w}_{a,b}\quad \text{ for all } 1 \le i \le k \text{ and } 1 \le j \le l.
$$
Note that $\mat{w}{\alpha}{\beta}$ has type $(\alpha,\beta)$.
\end{defn}

\begin{eg}\label{eg_permutation_to_matrix}
Let $\alpha := (2,3)$, $\beta := (3,2)$, and $w := 23514\in\Sym_5$. Then using \cref{eg_schubert}, we find $\mat{w}{\alpha}{\beta} = \scalebox{0.8}{$\begin{bmatrix}1 & 1 \\ 2 & 1\end{bmatrix}$}$.
\end{eg}

\begin{lem}\label{permutation_to_matrix_transpose}
Let $\alpha, \beta \models n$ and $w\in\Sym_n$. Then $\mat{w}{\alpha}{\beta} = \transpose{\mat{w^{-1}}{\beta}{\alpha}}$.
\end{lem}

\begin{proof}
This follows from the fact that the permutation matrix of $w^{-1}$ is $\transpose{\mathring{w}}$.
\end{proof}

Note that for $\alpha,\beta\models n$, the function $w\mapsto\mat{w}{\alpha}{\beta}$ is constant on the double cosets of $\Sym_\alpha\backslash\Sym_n/\Sym_\beta$. In fact, we have the following classification result:
\begin{prop}[{\cite[Section 3]{diaconis_gangolli95}}]\label{double_cosets}
Let $\alpha,\beta\models n$. Then the map
$$
\Mats{\alpha}{\beta} \to \Sym_\alpha\backslash\Sym_n/\Sym_\beta, \quad M \mapsto \{w\in\Sym_n : \mat{w}{\alpha}{\beta} = M\}
$$
is a bijection from nonnegative-integer matrices of type $(\alpha,\beta)$ to the set of double cosets of $\Sym_n$ modulo the left action of $\Sym_\alpha$ and the right action of $\Sym_\beta$.
\end{prop}

In order to make certain constructions more concrete, we introduce canonical double coset representatives.
\begin{defn}[{cf.\ \cite[Proof of Theorem 7.1]{garsia_stanton84}, \cite[Proof of Theorem 3.1]{diaconis_gangolli95}}]\label{defn_canonical_perm}
Let $\alpha,\beta\models n$, and let $M\in\mathbb{N}^{k\times l}$ have type $(\alpha,\beta)$. Let $\perm{M}\in\Sym_n$ denote the permutation whose permutation matrix $\permdot{M}\in\GL_n$ is constructed from the $n\times n$ zero matrix as follows. For all $1 \le i \le k$ and $1 \le j \le l$, we place a copy of $I_{M_{i,j}}$ in the contiguous $M_{i,j}\times M_{i,j}$ submatrix whose lower-right entry is indexed by
$$
(\alpha_1 + \cdots + \alpha_{i-1} + M_{i,1} + \cdots + M_{i,j}, \beta_1 + \cdots + \beta_{j-1} + M_{1,j} + \cdots + M_{i,j}).
$$
Equivalently, we read all the entries $M_{i,j}$ of $M$, in any order subject to the condition that we read $M_{i,j}$ before both $M_{i+1,j}$ and $M_{i,j+1}$; when we read $M_{i,j}$, we place a contiguous copy of $I_{M_{i,j}}$ somewhere in rows $\block{\alpha}{i}$ and in columns $\block{\beta}{j}$, as northwest as possible so as not to overlap with any row or column with a $1$ already in it. Note that $\perm{M}$ is the minimal-length representative of the double coset of $\Sym_\alpha\backslash\Sym_n/\Sym_\beta$ corresponding to $M$ as in \cref{double_cosets}.
\end{defn}

\begin{eg}\label{eg_canonical_perm}
Let $M := \scalebox{0.8}{$\begin{bmatrix}1 & 1 \\ 2 & 1\end{bmatrix}$}$. Then
\begin{gather*}
\permdot{M} = \begin{bmatrix}
1 & 0 & 0 & 0 & 0 \\
0 & 0 & 0 & 1 & 0 \\
0 & 1 & 0 & 0 & 0 \\
0 & 0 & 1 & 0 & 0 \\
0 & 0 & 0 & 0 & 1
\end{bmatrix}, \quad \text{ and so } \perm{M} = 13425.\qedhere
\end{gather*}

\end{eg}

We now use nonnegative-integer matrices of type $(\alpha,\beta)$ to index the $\Par_\alpha$-orbits of $\PFl_\beta$.
\begin{defn}\label{defn_orbits}
Given $M\in\Mats{\alpha}{\beta}$, we define
$$
X_M := \bigsqcup_{w\in\Sym^\beta,\hspace*{2pt} \mat{w}{\alpha}{\beta} = M}X_w \subseteq \PFl_\beta.
$$
\end{defn}

\begin{eg}\label{eg_orbits}
As in \cref{eg_permutation_to_matrix}, let $\alpha := (2,3)$ and $\beta := (3,2)$, and let $M := \scalebox{0.8}{$\begin{bmatrix}1 & 1 \\ 2 & 1\end{bmatrix}$}\in\Mats{\alpha}{\beta}$. Then
$$
X_M = X_{13425} \sqcup X_{13524} \sqcup X_{14523} \sqcup X_{23415} \sqcup X_{23514} \sqcup X_{24513} \subseteq\PFl_\beta.
$$
In particular, by \cref{schubert_decomposition}\ref{schubert_decomposition_matrix}, we have
\begin{gather*}
X_M \cong \field^2 \sqcup \field^3 \sqcup \field^4 \sqcup \field^3 \sqcup \field^4 \sqcup \field^5.\qedhere
\end{gather*}

\end{eg}

Using \cref{double_cosets}, we obtain that the sets $X_M$ are the orbits of $\Par_\alpha$ acting on $\PFl_\beta$:
\begin{prop}[{\cite[Section IV.2.5]{bourbaki68}}]\label{matrix_to_orbit}
Let $\alpha,\beta\models n$. Then the map
$$
\Mats{\alpha}{\beta} \to \Par_\alpha\backslash\PFl_\beta, \quad M \mapsto X_M
$$
is a bijection from nonnegative-integer matrices of type $(\alpha,\beta)$ to the $\Par_\alpha$-orbits of $\PFl_\beta$. In particular, we have the decomposition into $\Par_\alpha$-orbits
\begin{align}\label{orbit_decomposition_equation}
\PFl_\beta = \bigsqcup_{M\in\Mats{\alpha}{\beta}}X_M.
\end{align}

\end{prop}

\begin{rmk}\label{orbit_symmetry}
Note that the $\Par_\alpha$-orbits of $\PFl_\beta$ correspond to the double cosets $\Par_\alpha\backslash\GL_n/\Par_\beta$, which we will discuss further in \cref{sec_double_cosets}. They also correspond to the orbits of $\GL_n$ acting diagonally on the direct product $\PFl_\alpha \times \PFl_\beta$, since $\GL_n$ acts transitively on $\PFl_\alpha$ and the stabilizer of $\Unit{\id}\in\PFl_\alpha$ is $\Par_\alpha$.

While we prefer to work inside $\PFl_\beta$, ultimately $\alpha$ and $\beta$ play interchangeable roles in our analysis, as we make explicit in \cref{double_coset_symmetry}. However, we caution that in spite of \cref{permutation_to_matrix_transpose}, in general we have $X_M \ncong X_{\transpose{M}}$; indeed, $X_M\times\Par_\beta \cong X_{\transpose{M}}\times\Par_\alpha$ (cf.\ \cref{double_coset_decomposition}). For example, if $M = \scalebox{0.8}{$\begin{bmatrix}0 & 1 \\ 0 & 1\end{bmatrix}$}$, then $X_M$ is a point and $X_{\transpose{M}}\cong\Gr_{1,2}$.
\end{rmk}

\begin{prop}\label{alternative_orbit_decomposition}
Let $M\in\mathbb{N}^{k\times l}$ have type $(\alpha,\beta)$. Then
\begin{gather*}
X_M \cong \prod_{1 \le i' < i \le k,\hspace*{2pt} 1 \le j < j' \le l}\field^{M_{i,j}M_{i',j'}}\times\prod_{i=1}^k\PFl_{(M_{i,1},\dots,M_{i,l})}.
\end{gather*}

\end{prop}

\begin{proof}
We proceed by induction on $l$, where the base case $l=0$ is clear. Suppose that $l\ge 1$ and that the result holds for $l-1$. Let $\beta' := (\beta_2, \dots, \beta_l)$, and let $M'\in\mathbb{N}^{k\times (l-1)}$ denote the matrix obtained from $M$ by deleting the first column. Let us represent each $F\in\PFl_\beta$ by the unique matrix $g\in\GL_n$ as in \cref{schubert_decomposition}\ref{schubert_decomposition_matrix}. We call the entries of $g$ coming from the $1$'s of the corresponding permutation matrix the {\itshape pivot 1's}. Let $g'$ be obtained from $g$ by deleting columns $\block{\beta}{1}$ and the rows of the pivot $1$'s in columns $\block{\beta}{1}$. Note that $g'\in\GL_{n-\beta_1}$ is the unique representative of some partial flag $F'\in\PFl_{\beta'}$ as in \cref{schubert_decomposition}\ref{schubert_decomposition_matrix}. Then $F\in X_M$ if and only if $F'\in X_{M'}$ and for all $1 \le i \le k$, there are exactly $M_{i,1}$ pivot $1$'s in rows $\block{\alpha}{i}$ and columns $\block{\beta}{1}$ of $g$. That is, when we restrict $g$ to rows $\block{\alpha}{i}$ and columns $\block{\beta}{1}$, we must see the following:
\begin{itemize}[leftmargin=24pt, itemsep=2pt]
\item the first $M_{1,1} + \cdots + M_{i-1,1}$ columns are zero;
\item the next $M_{i,1}$ columns contain $M_{i,1}$ pivot $1$'s, zeros below and to the right of these pivot $1$'s, and arbitrary entries elsewhere, so that these columns span an arbitrary element of $\Gr_{M_{i,1},\alpha_i}$; and
\item the last $M_{i+1,1} + \cdots + M_{k,1}$ columns contain arbitrary entries, except in the rows of the pivot $1$'s, where the entries are zeros.
\end{itemize}
Also, note that every entry which is deleted from $g$ to form $g'$ whose column index is not in $\block{\beta}{1}$ is zero. Therefore
$$
X_M \cong \prod_{i=1}^k\big(\field^{(M_{i+1,1} + \cdots + M_{k,1})(\alpha_i - M_{i,1})}\times\Gr_{M_{i,1},\alpha_i}\big) \times X_{M'}.
$$
We can check that applying the induction hypothesis to $X_{M'}$ gives the desired expression for $X_M$, using \eqref{partial_flag_induction_equation}.
\end{proof}

\begin{eg}\label{eg_alternative_orbit_decomposition}
Let $M := \scalebox{0.8}{$\begin{bmatrix}1 & 1 \\ 2 & 1\end{bmatrix}$}$. Then by \cref{alternative_orbit_decomposition}, we have
$$
X_M \cong \field^2 \times \PFl_{(1,1)} \times \PFl_{(2,1)}.
$$
We can verify that this agrees with \cref{eg_orbits}, by decomposing $\PFl_{(1,1)}$ and $\PFl_{(2,1)}$ according to \cref{schubert_decomposition}\ref{schubert_decomposition_decomposition}.
\end{eg}

\begin{cor}\label{alternative_orbit_decomposition_q}
Set $\field := \qinvF$. Let $M\in\mathbb{N}^{k\times l}$ have type $(\alpha,\beta)$. Then
\begin{gather*}
|X_M| = q^{-\sum_{1 \le i' \le i \le k,\hspace*{1pt} 1 \le j < j' \le l}M_{i,j}M_{i',j'}}\frac{\prod_{i=1}^k\qfac{\alpha_i}}{\prod_{1 \le i \le k,\hspace*{1pt} 1 \le j \le l}\qfac{M_{i,j}}}.
\end{gather*}

\end{cor}

\begin{proof}
This follows from \cref{alternative_orbit_decomposition,q_enumeration}\ref{q_enumeration_P}. (Note that we allow $i' = i$ in the first index set.)
\end{proof}

\subsection{Relative position}\label{sec_relative_position}
We use the decomposition \eqref{orbit_decomposition_equation} to define the notion of relative position between two partial flags. By \cref{relative_position_properties}, our definition coincides with that of Rosso \cite[(1.1)]{rosso12}.

\begin{defn}\label{defn_relative_position}
Let $\alpha,\beta\models n$, and let $F\in\PFl_\alpha$ and $F'\in\PFl_\beta$. Take $g$ and $g'$ in $\GL_n$ representing $F$ and $F'$, respectively. The {\itshape relative position of $(F,F')$} is the (unique) $M\in\Mats{\alpha}{\beta}$ such that $X_M$ contains $g^{-1}g'$ (regarded as an element of $\PFl_\beta$). (Note that $M$ does not depend on the choice of $g$ and $g'$, by \cref{matrix_to_orbit}.) We write $F\relpos{M}F'$.
\end{defn}

\begin{lem}\label{relative_position_invariance}
Let $\alpha,\beta\models n$, let $M\in\Mats{\alpha}{\beta}$, and let $F,\tilde{F}\in\PFl_\alpha$ and $F',\tilde{F'}\in\PFl_\beta$.
\begin{enumerate}[label=(\roman*), leftmargin=*]
\item\label{relative_position_invariance_schubert}
Let $g\in\GL_n$ represent $F$. Then $F\relpos{M}F'$ if and only if $F'\in gX_M$.
\item\label{relative_position_invariance_action}
The pairs $(F,F')$ and $(\tilde{F},\tilde{F}')$ have the same relative position if and only if there exists $h\in\GL_n$ such that $hF = \tilde{F}$ and $hF' = \tilde{F}'$.
\item\label{relative_position_invariance_inverse} We have $F\relpos{M}F'$ if and only if $F'\relpos{\transpose{M}}F$.
\end{enumerate}
\end{lem}

\begin{proof}
Part \ref{relative_position_invariance_schubert} follows from \cref{defn_relative_position}. Part \ref{relative_position_invariance_action} follows from \cref{matrix_to_orbit,orbit_symmetry}. Part \ref{relative_position_invariance_inverse} follows from \cref{double_cosets,permutation_to_matrix_transpose}.
\end{proof}

\begin{prop}\label{relative_position_properties}
Let $\alpha,\beta\models n$, let $M\in\mathbb{N}^{k\times l}$ have type $(\alpha,\beta)$, and let $F\in\PFl_\alpha$ and $F'\in\PFl_\beta$. Then the following statements are equivalent.
\begin{enumerate}[label=(\roman*), leftmargin=*]
\item\label{relative_position_properties1} The relative position of $(F,F')$ is $M$.
\item\label{relative_position_properties2} There exist $V_{i,j}\in\Gr_{M_{i,j},n}$ for $1 \le i \le k$ and $1 \le j \le l$ such that
$$
F_i = \spn{V_{i',j'} : 1 \le i' \le i,\hspace*{2pt} 1 \le j' \le l} \quad \text{ and } \quad F'_j = \spn{V_{i',j'} : 1 \le i' \le k,\hspace*{2pt} 1 \le j' \le j}
$$
for all $0 \le i \le k$ and $0 \le j \le l$. (In particular, $\field^n = \bigoplus_{1 \le i \le k,\hspace*{1pt} 1 \le j \le l}V_{i,j}$.)
\item\label{relative_position_properties3} For all $0 \le i \le k$ and $0 \le j \le l$, we have
\begin{gather*}
\dim(F_i\cap F'_j) = \sum_{1\le i'\le i,\hspace*{2pt} 1\le j'\le j}M_{i',j'}.
\end{gather*}
\end{enumerate}

\end{prop}

\begin{proof}
\ref{relative_position_properties1} $\Rightarrow$ \ref{relative_position_properties2}: Suppose that $F\relpos{M}F'$, and take $g\in\GL_n$ which represents $F$. By \cref{relative_position_invariance}\ref{relative_position_invariance_schubert}, we have $F'\in gX_M$, and so $F'$ is represented by $gh\mathring{w}$ for some $h\in\Bor_n$ and $w\in\Sym_n$ with $\mat{w}{\alpha}{\beta} = M$. Let $v_1, \dots, v_n\in\field^n$ denote the columns of $gh$, which also represents $F$. Then we may take
$$
V_{i,j} := \spn{v_{i'} : i'\in\block{\alpha}{i},\hspace*{2pt} w^{-1}(i')\in\block{\beta}{j}} \quad \text{ for }1 \le i \le k \text{ and } 1 \le j \le l.
$$

\ref{relative_position_properties2} $\Rightarrow$ \ref{relative_position_properties3}: Suppose that $V_{i,j}\in\Gr_{M_{i,j},n}$ are as in part \ref{relative_position_properties2}. Then for all $0 \le i \le k$ and $0 \le j \le l$, we have
$$
F_i\cap F'_j = \spn{V_{i',j'} : 1 \le i' \le i,\hspace*{2pt} 1 \le j' \le j}.
$$
Taking dimensions gives the result.

\ref{relative_position_properties3} $\Rightarrow$ \ref{relative_position_properties1}: Suppose that \ref{relative_position_properties3} holds. Let $M'\in\mathbb{N}^{k\times l}$ be the relative position of $(F,F')$. Then since \ref{relative_position_properties1} $\Rightarrow$ \ref{relative_position_properties3}, we have
$$
\sum_{1\le i'\le i,\hspace*{2pt} 1\le j'\le j}M_{i',j'} = \sum_{1\le i'\le i,\hspace*{2pt} 1\le j'\le j}M'_{i',j'} \quad \text{ for all } 0 \le i \le k \text{ and } 0 \le j \le l.
$$
Therefore $M = M'$.
\end{proof}

\subsection{Parabolic double cosets of \texorpdfstring{$\GL_n$}{GL(n)}}\label{sec_double_cosets}
We now use the results of this section to index and enumerate the double cosets of $\GL_n$ modulo two standard parabolic subgroups. This leads to a $q$-analogue of the enumeration of the double cosets of $\Sym_n$ modulo two parabolic subgroups. While we will not make further use of the results of this subsection, we state them since we find them to be of independent interest. Also, this discussion will make manifest that $\alpha$ and $\beta$ play interchangeable roles.
\begin{defn}\label{defn_double_coset}
Let $\alpha,\beta\models n$. Given a nonnegative-integer matrix $M$ of type $(\alpha,\beta)$, let $Y_M\subseteq\GL_n$ denote the double coset of $\Par_\alpha\backslash\GL_n/\Par_\beta$ generated by $\mathring{w}$, for any $w\in\Sym_n$ such that $\mat{w}{\alpha}{\beta} = M$. We may take $\permdot{M}$ as a canonical double coset representative. By \cref{double_cosets,matrix_to_orbit}, the definition of $Y_M$ does not depend on the choice of $w$, and the distinct double cosets of $\Par_\alpha\backslash\GL_n/\Par_\beta$ are precisely $Y_M$ for $M\in\mathbb{N}^{k\times l}$ of type $(\alpha,\beta)$.
\end{defn}

\begin{eg}\label{eg_double_coset}
Let $\alpha := (2,3)$ and $\beta := (3,2)$. Then there are precisely $3$ double cosets of $\Par_\alpha\backslash\GL_5/\Par_\beta$, namely, $Y_M$ for
$$
M = \begin{bmatrix}2 & 0 \\ 1 & 2\end{bmatrix},\hspace*{2pt} \begin{bmatrix}1 & 1 \\ 2 & 1\end{bmatrix},\hspace*{2pt} \begin{bmatrix}0 & 2 \\ 3 & 0\end{bmatrix}.
$$
They are each generated by $\permdot{M}$, where
$$
\perm{M} = 12345,\hspace*{2pt} 13425,\hspace*{2pt} 34512 \in\Sym_5,
$$
respectively.
\end{eg}

\begin{rmk}\label{double_coset_symmetry}
Let $\alpha,\beta\models n$. By \cref{permutation_to_matrix_transpose}, the involution $g\mapsto g^{-1}$ on $\GL_n$ induces a bijection $Y_M \mapsto Y_{\transpose{M}}$ from the set of double cosets of $\Par_\alpha\backslash\GL_n/\Par_\beta$ to the set of double cosets of $\Par_\beta\backslash\GL_n/\Par_\alpha$. Similarly, the involution $w \mapsto w^{-1}$ on $\Sym_n$ induces a bijection from the set of double cosets of $\Sym_\alpha\backslash\Sym_n/\Sym_\beta$ to the set of double cosets of $\Sym_\beta\backslash\Sym_n/\Sym_\alpha$.
\end{rmk}

We have the following two descriptions of the double coset $Y_M$. Below, for a matrix $g$, we let $g_{I,J}$ denote the submatrix of $g$ using rows $I$ and columns $J$.
\begin{prop}\label{double_coset_rank}
Let $\alpha,\beta\models n$, and let $M\in\mathbb{N}^{k\times l}$ have type $(\alpha,\beta)$. Then $Y_M$ is the set of $g\in\GL_n$ such that
\begin{gather*}
\rank(g_{\block{\alpha}{i+1} \cup \cdots \cup \block{\alpha}{k},\hspace*{1pt}\block{\beta}{1}\cup\cdots\cup\block{\beta}{j}}) = \sum_{i < i' \le k,\hspace*{2pt} 1\le j'\le j}M_{i',j'} \quad \text{ for all } 0 \le i \le k \text{ and } 0 \le j \le l.
\end{gather*}

\end{prop}

\begin{proof}
Given $g\in\GL_n$, let $F'\in\PFl_\beta$ denote the partial flag represented by $g$. Then $g\in Y_M$ if and only if $F'\in X_M$. By \cref{relative_position_invariance}\ref{relative_position_invariance_schubert}, we have $F'\in X_M$ if and only if $\Unit{\id}\relpos{M}F'$. Finally, by \cref{relative_position_properties}, we have $\Unit{\id}\relpos{M}F'$ if and only if
$$
\dim(\spn{\unit{1}, \dots, \unit{\alpha_1 + \cdots + \alpha_i}}\cap F'_j) = \sum_{1\le i'\le i,\hspace*{2pt} 1\le j'\le j}M_{i',j'} \quad \text{ for all } 0 \le i \le k \text{ and } 0 \le j \le l.
$$
Note that $\spn{\unit{1}, \dots, \unit{\alpha_1 + \cdots + \alpha_i}}\cap F'_j$ is the kernel of the projection map $F'_j \to \spn{\unit{\alpha_1 + \cdots + \alpha_i + 1}, \dots, \unit{n}}$ which forgets the first $\alpha_1 + \cdots + \alpha_i$ coordinates. Therefore
$$
\dim(\spn{\unit{1}, \dots, \unit{\alpha_1 + \cdots + \alpha_i}}\cap F'_j) = \beta_1 + \cdots + \beta_j - \rank(g_{\block{\alpha}{i+1} \cup \cdots \cup \block{\alpha}{k},\hspace*{1pt}\block{\beta}{1}\cup\cdots\cup\block{\beta}{j}}),
$$
and the result follows.
\end{proof}

\begin{prop}\label{double_coset_decomposition}
Let $\alpha,\beta\models n$, and let $M\in\mathbb{N}^{k\times l}$ have type $(\alpha,\beta)$. Then
$$
Y_M \cong X_M \times \Par_\beta.
$$
If $\field = \qinvF$, then
\begin{align}\label{double_coset_decomposition_q}
|Y_M| = q^{-n^2 + \sum_{1 \le i < i' \le k,\hspace*{1pt} 1 \le j < j' \le l}M_{i,j}M_{i',j'}}(1-q)^n\frac{\prod_{i=1}^k\qfac{\alpha_i}\prod_{j=1}^l\qfac{\beta_j}}{\prod_{1 \le i \le k,\hspace*{1pt} 1 \le j \le l}\qfac{M_{i,j}}}.
\end{align}

\end{prop}

\begin{proof}
Since $\PFl_\beta = \GL_n/\Par_\beta$, and $\Par_\beta$ acts freely on $\GL_n$, we have $Y_M \cong X_M \times \Par_\beta$. Now set $\field := \qinvF$. By \cref{alternative_orbit_decomposition_q,q_enumeration}, we have
$$
|Y_M| = \frac{|X_M|\cdot|\GL_n|}{|\PFl_\beta|} = q^{-n^2 + e_2(\beta) - \sum_{1 \le i' \le i \le k,\hspace*{1pt} 1 \le j < j' \le l}M_{i,j}M_{i',j'}}(1-q)^n\frac{\prod_{i=1}^k\qfac{\alpha_i}\prod_{j=1}^l\qfac{\beta_j}}{\prod_{1 \le i \le k,\hspace*{1pt} 1 \le j \le l}\qfac{M_{i,j}}},
$$
where $e_2(\beta)$ was defined in \cref{defn_n}. This simplifies to the desired formula, using the fact that
\begin{equation}\label{beta_sum_matrix}
e_2(\beta) = \sum_{\substack{1 \le i,i' \le k, \\ 1 \le j < j' \le l}}M_{i,j}M_{i',j'} = \sum_{\substack{1 \le i < i' \le k, \\ 1 \le j < j' \le l}}M_{i,j}M_{i',j'} + \sum_{\substack{1 \le i' \le i \le k, \\ 1 \le j < j' \le l}}M_{i,j}M_{i',j'}.\qedhere
\end{equation}

\end{proof}

\begin{rmk}\label{double_coset_permutations}
Note that by dividing the right-hand side of \eqref{double_coset_decomposition_q} by $(1-q)^n$ and setting $q=1$, we obtain
\begin{align}\label{contingency_probability}
\frac{\prod_{i=1}^k\alpha_i!\prod_{j=1}^l\beta_j!}{\prod_{1 \le i \le k,\hspace*{1pt} 1 \le j \le l}M_{i,j}!}.
\end{align}
This is equal to the size of the double coset of $\Sym_\alpha\backslash\Sym_n/\Sym_\beta$ corresponding to $M$ as in \cref{double_cosets}. We can verify this directly, or by an argument similar to the proof of \cref{alternative_orbit_decomposition}; also see the discussion of contingency tables in \cref{probabilistic_interpretation}. For example, if $M = \scalebox{0.8}{$\begin{bmatrix}1 & 1 \\ 2 & 1\end{bmatrix}$}$, then the corresponding double coset of $\Sym_{(2,3)}\backslash\Sym_5/\Sym_{(3,2)}$ has size $\frac{2!\hspace*{1pt}3!\hspace*{1pt}3!\hspace*{1pt}2!}{1!\hspace*{1pt}1!\hspace*{1pt}2!\hspace*{1pt}1!} = 72$.
\end{rmk}

\section{\texorpdfstring{$q$}{q}-Whittaker functions via partial flags}\label{sec_nilpotent}

\noindent In this section we present two characterizations of the $q$-Whittaker function $W_\lambda(\mathbf{x};q)$ in terms of nilpotent endomorphisms acting on partial flags (\cref{whittaker_expansion,whittaker_expansion_dual}). We give some background on (strict) compatibility in \cref{sec_compatibility}, and then present our two characterizations in \cref{sec_expansion,sec_expansion_dual}.

\subsection{Compatibility of nilpotent endomorphisms and partial flags}\label{sec_compatibility}
In this subsection, we discuss two notion of compatibility between nilpotent endomorphisms and partial flags.
\begin{defn}\label{defn_compatible}
Let $\alpha = (\alpha_1, \dots, \alpha_k)\models n$, $F\in\PFl_\alpha$, and $\nil\in\gl_n$. We say that $F$ and $\nil$ are {\itshape strictly compatible} if
$$
\nil(F_i) \subseteq F_{i-1} \quad \text{ for all } 1 \le i \le k.
$$
In this case $\nil$ is nilpotent. By \cref{flag_to_tableau}, we may define a semistandard tableau $T$ of shape $\JF(\nil)$ and content $\alpha$ such that $T^{(i)} = \JF(\nil|_{F_i})$ for $1 \le i \le k$. We denote $T$ by $\JF(\nil;F)$.
\end{defn}

\begin{eg}\label{eg_compatible}
Let $\alpha := (2,1,3)$, and let
$$
\nil := \begin{bmatrix}
0 & 0 & 1 & 6 & -4 & 2 \\
0 & 0 & 3 & -2 & 7 & 5 \\
0 & 0 & 0 & 0 & 0 & 0 \\
0 & 0 & 0 & 0 & 0 & 0 \\
0 & 0 & 0 & 0 & 0 & 0 \\
0 & 0 & 0 & 0 & 0 & 0
\end{bmatrix}\in\gl_6.
$$
Then $\nil$ is strictly compatible with $\Unit{\id}\in\PFl_\alpha$. Let us calculate $T := \JF(\nil;\Unit{\id})$. We have
$$
T^{(1)} = \JF\hspace*{-2pt}\Big(\scalebox{0.8}{$\begin{bmatrix}0 & 0 \\ 0 & 0\end{bmatrix}$}\Big) = \ydiagramsmall{2}, \quad T^{(2)} = \JF\hspace*{-2pt}\left(\scalebox{0.8}{$\begin{bmatrix}0 & 0 & 1 \\ 0 & 0 & 3 \\ 0 & 0 & 0\end{bmatrix}$}\right) = \ydiagramsmall{2,1}, \quad T^{(3)} = \JF(\nil) = \ydiagramsmall{4,2},
$$
so
\begin{gather*}
T = \hspace*{4pt}\begin{ytableau}
1 & 1 & 3 & 3 \\
2 & 3
\end{ytableau}\hspace*{4pt}.\qedhere
\end{gather*}

\end{eg}

\begin{rmk}\label{weakly_vs_strictly_compatible}
Following \cref{defn_compatible}, for $\nil$ nilpotent, we say that $F$ and $\nil$ are {\itshape weakly compatible} if 
$$
\nil(F_i) \subseteq F_i \quad \text{ for all } 0 \le i \le k.
$$
Note that if $F$ and $\nil$ are strictly compatible, then they are also weakly compatible. The two notions coincide for the complete flag variety $\Fl_n$, but they differ in general. Both notions lead to natural generalizations of {\itshape Springer fibers} from $\Fl_n$ to $\PFl_\alpha$ \cite{spaltenstein76,steinberg76}; see \cite[Example 1]{fresse16} for further discussion. Weak compatibility also arises in a characterization of {\itshape Hall--Littlewood functions}, which we discuss in \cref{remark_hall-littlewood}. However, weak compatibility is not suitable for our purposes, because it does not produce semistandard tableaux. Namely, for a weakly compatible pair $F$ and $\nil$, the tableau $T$ defined in \cref{defn_compatible} is weakly increasing along rows and columns, but not necessarily strictly increasing along columns. (Such a tableau $T$ is a {\itshape reverse plane partition}; we will encounter reverse plane partitions in a different context in \cref{sec_quiver}.)

For example, consider the case $\alpha := (n)$, and let $\lambda := \JF(\nil)$. Then there is a unique partial flag $F\in\PFl_\alpha$, and it is weakly compatible with $\nil$. The associated tableau $T$ is the filling of the Young diagram of $\lambda$ with all $1$'s. However, this tableau is not semistandard unless $\lambda = (n)$, or equivalently, $\nil = 0$ (which is also the only case when $F$ and $\nil$ are strictly compatible).
\end{rmk}

We observe that \cref{defn_compatible} is compatible with the action of $\GL_n$:
\begin{lem}\label{compatible_matrices_adjoint}
Let $\alpha\models n$, $F\in\PFl_\alpha$, $\nil\in\gl_n$, and $g\in\GL_n$. Then $\nil$ is strictly compatible with $F$ if and only if $g\nil g^{-1}$ is strictly compatible with $gF$. In this case, $\JF(\nil;F) = \JF(g\nil g^{-1};gF)$.
\end{lem}

\begin{proof}
This follows from the definitions.
\end{proof}

We have the following characterization of strictly compatible pairs, where we recall that the subspace $\Nil_\alpha$ of $\gl_n$ was defined in \cref{defn_lie}:
\begin{prop}\label{compatible_matrices}
Let $F\in\PFl_\alpha$ be represented by $g\in\GL_n$. Then the set of nilpotent endomorphisms strictly compatible with $F$ is $g\Nil_\alpha g^{-1}$.
\end{prop}

\begin{proof}
This follows from \cref{compatible_matrices_adjoint}, since $\Nil_\alpha$ is the set of nilpotent endomorphisms strictly compatible with $\Unit{\id} \in \PFl_\alpha$.
\end{proof}

\subsection{Expansion via partial flags}\label{sec_expansion}
In this subsection, we present an expansion of $W_\lambda(\mathbf{x};q)$ in terms of nilpotent endomorphisms and partial flags. To state it, we recall that $\ncon{\cdot}$ was defined in \cref{defn_n}, and $\qwt(\cdot)$ was defined in \cref{defn_whittaker}.
\begin{thm}\label{whittaker_expansion}
Set $\field := \qinvF$. Let $\lambda\vdash n$, and let $\nil\in\gl_n$ be nilpotent with $\JF(\nil) = \lambda$.
\begin{enumerate}[label=(\roman*), leftmargin=*]
\item\label{whittaker_expansion_tableau} For $\alpha\models n$ and $T\in\SSYT(\lambda,\alpha)$, we have
$$
\qwt(T) = q^{\ncon{\lambda} - \ncon{\alpha}}|\{F\in\PFl_\alpha : F \text{ is strictly compatible with $\nil$, and } \JF(\nil;F) = T\}|.
$$
\item\label{whittaker_expansion_coefficient} For $\alpha\models n$, we have
$$
[\mathbf{x}^\alpha]W_\lambda(\mathbf{x};q) = q^{\ncon{\lambda} - \ncon{\alpha}}|\{F\in\PFl_\alpha : F \text{ is strictly compatible with } \nil\}|.
$$
\item\label{whittaker_expansion_formula}
We have
$$
W_\lambda(\mathbf{x};q) = q^{\ncon{\lambda}}\sum_F \prod_{i\ge 1}q^{-\binom{\dim(F_i/F_{i-1})}{2}}x_i^{\dim(F_i/F_{i-1})},
$$
where the sum is over all partial flags $F$ of $\field^n$ strictly compatible with $\nil$.
\end{enumerate}
\end{thm}

We consider an example and make several remarks about \cref{whittaker_expansion}, before proving it in the remainder of this subsection.
\begin{eg}\label{eg_whittaker_expansion}
Set $\field := \qinvF$, and let $\lambda := (4,2)$ and $\alpha := (2,1,3)$. Let us use \cref{whittaker_expansion}\ref{whittaker_expansion_coefficient} (and \cref{q_enumeration}\ref{q_enumeration_Gr}) to find the coefficient of $\mathbf{x}^\alpha = x_1^2x_2x_3^3$ in $W_\lambda(\mathbf{x};q)$ (cf. \cref{eg_whittaker}). We set $\nil := \nil_\lambda$, which acts as in the following picture:
$$
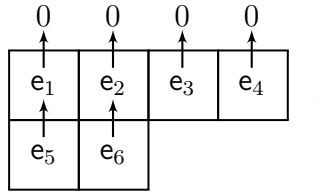
\quad\begin{tikzpicture}[baseline=(current bounding box.center)]
\pgfmathsetmacro{\u}{0.922};
\useasboundingbox(0,0.7*\u)rectangle(4*\u,-2.2*\u);
\coordinate (vepsilon)at(0,0.2*\u);
\coordinate (vepsilonup)at(0,0.25*\u);
\node[inner sep=0]at(0,0){\scalebox{1.6}{\begin{ytableau}
\none \\
\none \\
\none & \none & \none & \none & & & & \\
\none & \none & \none & \none & &
\end{ytableau}}};
\foreach \y in {1,...,2}{
\path[thick,-latex']($(1*\u,-\y*\u)+(-0.5*\u,0.5*\u)+(vepsilonup)$)edge($(1*\u,-\y*\u+\u)+(-0.5*\u,0.5*\u)-(vepsilon)$);}
\foreach \y in {1,...,2}{
\path[thick,-latex']($(2*\u,-\y*\u)+(-0.5*\u,0.5*\u)+(vepsilonup)$)edge($(2*\u,-\y*\u+\u)+(-0.5*\u,0.5*\u)-(vepsilon)$);}
\foreach \y in {1}{
\path[thick,-latex']($(3*\u,-\y*\u)+(-0.5*\u,0.5*\u)+(vepsilonup)$)edge($(3*\u,-\y*\u+\u)+(-0.5*\u,0.5*\u)-(vepsilon)$);}
\foreach \y in {1}{
\path[thick,-latex']($(4*\u,-\y*\u)+(-0.5*\u,0.5*\u)+(vepsilonup)$)edge($(4*\u,-\y*\u+\u)+(-0.5*\u,0.5*\u)-(vepsilon)$);}
\foreach \x in {1,...,4}{
\node[inner sep=0]at($(\x*\u-0.5*\u,0.5*\u)$){$0$};}
\node[inner sep=0]at($(0.5*\u,-0.5*\u)$){$\unit{1}$};
\node[inner sep=0]at($(1.5*\u,-0.5*\u)$){$\unit{2}$};
\node[inner sep=0]at($(2.5*\u,-0.5*\u)$){$\unit{3}$};
\node[inner sep=0]at($(3.5*\u,-0.5*\u)$){$\unit{4}$};
\node[inner sep=0]at($(0.5*\u,-1.5*\u)$){$\unit{5}$};
\node[inner sep=0]at($(1.5*\u,-1.5*\u)$){$\unit{6}$};
\end{tikzpicture}\quad.
$$

Let us enumerate all $F = (0,F_1,F_2,\field^6)\in\PFl_\alpha$ strictly compatible with $\nil$. Note that in particular, $\spn{\unit{1},\unit{2}}\subseteq F_2$. We consider two cases, depending on whether $F_2$ is contained in $\spn{\unit{1},\unit{2},\unit{3},\unit{4}}$.

{\itshape Case 1: $F_2 \subseteq \spn{\unit{1},\unit{2},\unit{3},\unit{4}}$.} In this case, $F_2/\spn{\unit{1},\unit{2}}$ is an arbitrary $1$-dimensional subspace of $\spn{\unit{1},\unit{2},\unit{3},\unit{4}}/\spn{\unit{1},\unit{2}}$, so there are $|\Gr_{1,2}| = q^{-1} + 1$ choices for $F_2$. Then $\nil(F_2) = 0$, so $F_1$ is an arbitrary $2$-dimensional subspace of $F_2$, for which there are $|\Gr_{2,3}| = q^{-2} + q^{-1} + 1$ choices. The total number of partial flags $F$ in this case is
$$
(q^{-1} + 1)(q^{-2} + q^{-1} + 1) = q^{-3} + 2q^{-2} + 2q^{-1} + 1.
$$

{\itshape Case 2: $F_2 \nsubseteq \spn{\unit{1},\unit{2},\unit{3},\unit{4}}$.} In this case, we can write $F_2 = \spn{\unit{1},\unit{2},x+y}$, where $x$ spans a $1$-dimensional subspace of $\spn{\unit{5},\unit{6}}$ (for which there are $|\Gr_{1,2}| = q^{-1} + 1$ choices), and $y$ is any vector in $\spn{\unit{3},\unit{4}}$ (for which there are $q^{-2}$ choices). Since $\nil(F_1) = 0$, we must have $F_1 = \spn{\unit{1},\unit{2}}$. The total number of partial flags $F$ in this case is
$$
(q^{-1} + 1)q^{-2} = q^{-3} + q^{-2}.
$$

In conclusion,
$$
|\{F\in\PFl_\alpha : F \text{ is strictly compatible with } \nil\}| = 2q^{-3} + 3q^{-2} + 2q^{-1} + 1.
$$
Multiplying by $q^{\ncon{\lambda} - \ncon{\alpha}} = q^3$, we obtain
$$
[x_1^2x_2x_3^3]W_\lambda(\mathbf{x};q) = 2 + 3q + 2q^2 + q^3,
$$
in agreement with \cref{eg_whittaker}.
\end{eg}

\begin{rmk}\label{remark_hall-littlewood}
We mention that the description of $W_\lambda(\mathbf{x};q)$ in \cref{whittaker_expansion}\ref{whittaker_expansion_coefficient} resembles a well-known description of the {\itshape Hall--Littlewood functions} \cite[Chapter II]{macdonald95}, using the notion of weak compatibility rather than strict compatibility (cf.\ \cref{weakly_vs_strictly_compatible}). We present this description using the {\itshape modified Hall--Littlewood functions}, following \cite[Section 1]{kirillov00} (cf.\ \cite[Section 3]{haiman03}). Namely, given $\lambda\vdash n$, we define
$$
H_\lambda(\mathbf{x};q) := q^{\ncon{\transpose{\lambda}}}\widetilde{H}_\lambda(\mathbf{x};0,1/q), \quad \text{ i.e.}, \quad H_{\transpose{\lambda}}(\mathbf{x};q) = q^{\ncon{\lambda}}\widetilde{H}_\lambda(\mathbf{x};1/q,0) = \omega(W_\lambda(\mathbf{x};q)),
$$
where $\widetilde{H}_\lambda$ is the modified Macdonald polynomial discussed in \cref{macdonald_remark}. Then letting $\field = \qinvF$ and $\nil\in\gl_n$ be nilpotent with $\JF(\nil) = \lambda$, we have \cite[Section 1.4]{kirillov00}
\begin{align}\label{hall-littlewood_formula}
[\mathbf{x}^\alpha]H_{\transpose{\lambda}}(\mathbf{x};q) = q^{\ncon{\lambda}}|\{F\in\PFl_\alpha : F \text{ is weakly compatible with } \nil\}| \quad \text{ for all } \alpha\models n.
\end{align}
Despite the superficial similarity between \eqref{hall-littlewood_formula} and \cref{whittaker_expansion}\ref{whittaker_expansion_coefficient}, we do not know how to deduce either formula from the other.

We also note that the left-hand sides of \cref{whittaker_expansion}\ref{whittaker_expansion_coefficient} and \eqref{hall-littlewood_formula} can be written in terms of the standard inner product $\langle\cdot,\cdot\rangle$ on symmetric functions. Namely, we have
\begin{gather*}
[\mathbf{x}^\alpha]W_\lambda(\mathbf{x},q) = \langle W_\lambda(\mathbf{x},q), h_\alpha\rangle \quad \text{ and } \quad [\mathbf{x}^\alpha]H_{\transpose{\lambda}}(\mathbf{x};q) = \langle W_\lambda(\mathbf{x},q), e_\alpha\rangle.
\end{gather*}

\end{rmk}

\begin{rmk}\label{remark_perverse}
The result \cref{whittaker_expansion}\ref{whittaker_expansion_coefficient} also follows from known results on Springer fibers; we briefly sketch the details.\footnote{We thank Sean Griffin for informing us of this and for supplying the argument used here.} Let $\nil\in\gl_n$ be nilpotent with $\JF(\nil) = \lambda$, and let $\alpha\models n$. Define the varieties over $\field$
\begin{align*}
\mathcal{X}^\lambda &:= \{F\in\Fl_n : F \text{ is strictly compatible with } \nil\}, \\
\mathcal{X}^\lambda_\alpha &:= \{F\in\Fl_\alpha : F \text{ is strictly compatible with } \nil\}.
\end{align*}
Then \cref{whittaker_expansion}\ref{whittaker_expansion_coefficient} is equivalent to the identity
\begin{align}\label{perverse_equation}
|\mathcal{X}^\lambda_\alpha(\qinvF)| = q^{\ncon{\alpha}}[\mathbf{x}^\alpha]\omega(\widetilde{H}_\lambda(\mathbf{x}; 1/q, 0)).
\end{align}

Springer \cite{springer78} defined an action of $\Sym_n$ on the cohomology $H^*(\mathcal{X}^\lambda; \mathbb{Q})$ of $\mathcal{X}^\lambda$. (We caution that our conventions follow \cite{borho_macpherson83} in defining the $\Sym_n$-action to be Springer's action multiplied by the sign character.) Using perverse sheaves, Borho and MacPherson \cite[Corollary 3.4(b)]{borho_macpherson83} (cf.\ \cite[Theorem 1.2]{brundan_ostrik11}) showed that
$$
H^i(\mathcal{X}^\lambda_\alpha; \mathbb{Q}) \cong H^{i + 2\ncon{\alpha}}(\mathcal{X}^\lambda; \mathbb{Q})^\text{$\Sym_\alpha$-anti} \quad \text{ for all } i\in\mathbb{Z},
$$
where $\cdot^\text{$\Sym_\alpha$-anti}$ takes $\Sym_\alpha$-anti-invariants. On the other hand, Brundan and Ostrik \cite[Theorem 2.4]{brundan_ostrik11} showed that $\mathcal{X}^\lambda_\alpha$ has an affine paving, so
$$
|\mathcal{X}^\lambda_\alpha(\qinvF)| = \sum_{i\ge 0}\dim(H^{2i}(\mathcal{X}^\lambda_\alpha; \mathbb{Q}))q^{-i} = q^{\ncon{\alpha}}\sum_{j\ge 0}\dim(H^{2j}(\mathcal{X}^\lambda; \mathbb{Q})^\text{$\Sym_\alpha$-anti})q^{-j}.
$$
Finally, by work of Hotta and Springer \cite{hotta_springer77} (cf.\ \cite[Theorem 3.4.19]{haiman03}), the Frobenius characteristic map sends $H^*(\mathcal{X}^\lambda; \mathbb{Q})$ (with grading parameter $q^{-1/2}$) to $\widetilde{H}_\lambda(\mathbf{x}; 1/q, 0)$. This gives \eqref{perverse_equation}, where we apply $[\mathbf{x}^\alpha]\omega$ since we are taking $\Sym_\alpha$-anti-invariants. A similar argument can be used to prove \eqref{hall-littlewood_formula}, using \cite[Corollary 3.4(a)]{borho_macpherson83}.
\end{rmk}

\begin{rmk}\label{symmetry_remark}
Recall that $W_\lambda(\mathbf{x};q)$ is symmetric in $\mathbf{x}$. By \cref{whittaker_expansion}\ref{whittaker_expansion_coefficient}, this is equivalent to the fact that
\begin{align}\label{symmetry_quantity}
|\{F\in\PFl_\alpha : F \text{ is strictly compatible with } \nil\}|
\end{align}
is invariant under permuting the entries of the weak composition $\alpha\models n$. We can see this directly, as follows. It suffices to show that \eqref{symmetry_quantity} is invariant under switching two consecutive entries of $\alpha$, which reduces to the case that $\alpha = (k,n-k)$, and we switch $k$ and $n-k$. This case follows from the fact that
$$
|\{V\in\Gr_{k,n} : \nil(\field^n) \subseteq V \subseteq \ker(\nil)\}| \cong |\{W\in\Gr_{n-k,n} : \nil(\field^n) \subseteq W \subseteq \ker(\nil)\}|,
$$
which is a consequence of the rank-nullity theorem.
\end{rmk}

In order to prove \cref{whittaker_expansion}, we will need the following enumerative statement, which follows from the results of \cref{sec_pairs_partial}:
\begin{lem}\label{intersection_dimension_count}
Let $V_1, \dots, V_k$ be vector spaces of dimensions $\alpha_1, \dots, \alpha_k$, and set $V := V_1 \oplus \cdots \oplus V_k$. For $1 \le i \le k$, take $\gamma_i\in\mathbb{N}$ with $0 \le \gamma_i \le \alpha_i$. Let $\mathcal{W}$ denote the set of subspaces $W\subseteq V$ such that
$$
\dim((V_1\oplus\cdots\oplus V_i) \cap W) = \gamma_1 + \cdots + \gamma_i \quad \text{ for all } 1 \le i \le k.
$$
Then
\begin{gather*}
\mathcal{W} \cong \prod_{1 \le i < j \le k}\field^{(\alpha_i-\gamma_i)\gamma_j} \times \prod_{i=1}^k\Gr_{\gamma_i,\alpha_i}.
\end{gather*}

\end{lem}

\begin{proof}
Let $n := \alpha_1 + \cdots + \alpha_k$. We may assume that $V_i = \field^{\block{\alpha}{i}}$ for $1 \le i \le k$. Then for the partial flag $\Unit{\id}\in\PFl_\alpha$, we have $\Unit{\id}_i = V_1 \oplus\cdots\oplus V_i$ for all $0 \le i \le k$. Now define $M\in\mathbb{N}^{k\times 2}$ by
$$
M_{i,1} := \gamma_i \hspace*{4pt}\text{ and }\hspace*{4pt} M_{i,2} := \alpha_i - \gamma_i \quad \text{ for } 1 \le i \le k.
$$
Note that $M$ has type $(\alpha,\beta)$, where $\beta := (\gamma_1 + \cdots + \gamma_k, n - (\gamma_1 + \cdots + \gamma_k))\models n$. Given an arbitrary subspace $W\subseteq V$ with $\dim(W) = \gamma_1 + \cdots + \gamma_k$, let $F'\in\PFl_\beta$ denote the partial flag $(0,W,V)$. By \cref{relative_position_properties}, we have $W\in\mathcal{W}$ if and only if $\Unit{\id}\relpos{M}F'$. This is in turn equivalent to $F' \in X_M$, by taking $g = I_n$ (which represents $\Unit{\id}$) in \cref{relative_position_invariance}\ref{relative_position_invariance_schubert}. Hence $\mathcal{W}\cong X_M$, and the result follows from \cref{alternative_orbit_decomposition}.
\end{proof}

The following result is the key to the induction step in the proof of \cref{whittaker_expansion}:
\begin{lem}\label{whittaker_induction_step}
Let $\lambda/\mu$ be a horizontal strip, and let $\nil\in\gl_n$ be nilpotent with $\JF(\nil) = \lambda$. Then
\begin{gather*}
\{V\subseteq\field^n : \nil(\field^n)\subseteq V \text{ and } \JF(\nil|_V) = \mu\} \cong \prod_{1 \le i < j}\field^{(\lambda_i-\mu_i)(\mu_j-\lambda_{j+1})} \times \prod_{i\ge 1}\Gr_{\mu_i-\lambda_{i+1},\lambda_i-\lambda_{i+1}}.
\end{gather*}

\end{lem}

\begin{proof}
We may assume that $V = \Boxes_\lambda$ and $\nil = \nil_\lambda$ from \cref{defn_boxes}. Let $A$ denote the set of boxes of the Young diagram of $\lambda$ not in the south border. Also, for $i\ge 1$, let $B_i$ denote the set of boxes in row $i$ in the south border; note that $|B_i| = \lambda_i - \lambda_{i+1}$. Let $V$ denote an arbitrary subspace of $\field^n$ such that $\nil(\field^n)\subseteq V$ and $\JF(\nil|_V) = \mu$. Then by \cref{matrix_to_partition}, $V$ is the direct sum of $\spn{A}$ and an arbitrary subspace $W$ of $\spn{\bigsqcup_{i\ge 1}B_i}$ satisfying
$$
\dim(\spn{B_1 \sqcup \cdots \sqcup B_i} \cap W) = \sum_{j=1}^i(\mu_j - \lambda_{j+1}) \quad \text{ for all } i\ge 1.
$$
Such $W$ are enumerated by \cref{intersection_dimension_count}.
\end{proof}

\begin{proof}[Proof of \cref{whittaker_expansion}]
We prove part \ref{whittaker_expansion_tableau}, whence part \ref{whittaker_expansion_coefficient} follows by summing over all $T\in\SSYT(\lambda,\alpha)$, and part \ref{whittaker_expansion_formula} follows by summing over all $\alpha\models n$. Write $\alpha = (\alpha_1, \dots, \alpha_k)$. We proceed by induction on $k$ (for all $n\in\mathbb{N}$ and $\lambda\vdash n$), where the base case $k=0$ is clear. Suppose that $k\ge 1$ and that part \ref{whittaker_expansion_tableau} holds for $k-1$. Let $\lambda\vdash n$ and $T\in\SSYT(\lambda,\alpha)$. Set $\mu := T^{(k-1)}$ and $\beta := (\alpha_1, \dots, \alpha_{k-1})$, and let $T'\in\SSYT(\mu,\beta)$ be obtained from $T$ by deleting all the entries $k$. Then
$$
\ncon{\alpha} = \ncon{\beta} + \binom{\alpha_k}{2}  \quad \text{ and } \quad \qwt(T) = \qwt(T')\prod_{i=1}^{k-1}\qbinom{\lambda_i - \lambda_{i+1}}{\lambda_i - \mu_i}.
$$

Let us enumerate all $F\in\PFl_\alpha$ such that $F$ is strictly compatible with $\nil$ and $\JF(\nil;F) = T$. The number of $F_{k-1}\subseteq\field^n$ such that $\nil(\field^n)\subseteq F_{k-1}$ and $\JF(\nil|_{F_{k-1}}) = \mu$ is
$$
q^{-\sum_{1 \le i \le j}(\lambda_i - \mu_i)(\mu_j - \lambda_{j+1})}\prod_{i=1}^{k-1}\qbinom{\lambda_i - \lambda_{i+1}}{\mu_i - \lambda_{i+1}} = q^{-\ncon{\lambda}+\ncon{\mu}+\binom{\alpha_k}{2}}\prod_{i=1}^{k-1}\qbinom{\lambda_i - \lambda_{i+1}}{\lambda_i - \mu_i},
$$
by \cref{whittaker_induction_step}, \cref{q_enumeration}\ref{q_enumeration_Gr}, and \cref{maximum_degree_recursion}. By the induction hypothesis, the number of ways to complete $F_{k-1}$ to $F$ is $$
q^{-\ncon{\mu}+\ncon{\beta}}\qwt(T').
$$
Therefore the number of such $F$ is
$$
q^{-\ncon{\lambda}+\ncon{\beta}+\binom{\alpha_k}{2}}\qwt(T')\prod_{i=1}^{k-1}\qbinom{\lambda_i - \lambda_{i+1}}{\lambda_i - \mu_i} = q^{-\ncon{\lambda}+\ncon{\alpha}}\qwt(T),
$$
as desired.
\end{proof}

\subsection{Dual expansion via nilpotent endomorphisms}\label{sec_expansion_dual}
We now state a dual formulation of \cref{whittaker_expansion}, where instead of fixing a nilpotent endomorphism $\nil$ and counting partial flags $F$, we fix $F$ and count $\nil$. For the sake of concreteness, we take $F = \Unit{\id}\in\PFl_\alpha$, so that by \cref{compatible_matrices}, the set of nilpotent endomorphisms strictly compatible with $F$ is $\Nil_\alpha$. Our result will naturally arise as a consequence of more general results in \cref{sec_q_burge} (see \cref{probabilities_diagonal}), so we defer the proof to \cref{sec_polynomiality}.
\begin{thm}\label{whittaker_expansion_dual}
Set $\field := \qinvF$. Let $\lambda\vdash n$ and $\alpha\models n$.
\begin{enumerate}[label=(\roman*), leftmargin=*]
\item\label{whittaker_expansion_dual_tableau} For $T\in\SSYT(\lambda,\alpha)$, we have
$$
\qwt(T) = \frac{q^{\binom{n}{2} - \ncon{\lambda}}(1-q)^{-n+\lambda_1}\prod_{i\ge 1}\qfac{\lambda_i - \lambda_{i+1}}}{\prod_{i\ge 1}\qfac{\alpha_i}}|\{\nil\in\Nil_\alpha : \JF(\nil;\Unit{\id}) = T\}|.
$$
\item\label{whittaker_expansion_dual_coefficient} We have
\begin{gather*}
[\mathbf{x}^\alpha]W_\lambda(\mathbf{x};q) = \frac{q^{\binom{n}{2} - \ncon{\lambda}}(1-q)^{-n+\lambda_1}\prod_{i\ge 1}\qfac{\lambda_i - \lambda_{i+1}}}{\prod_{i\ge 1}\qfac{\alpha_i}}|\{\nil\in\Nil_\alpha : \JF(\nil) = \lambda\}|.
\end{gather*}
\end{enumerate}

\end{thm}

\begin{eg}\label{eg_whittaker_expansion_dual}
Set $\field := \qinvF$, and let $\lambda := (4,2)$ and $\alpha := (2,1,3)$, as in \cref{eg_whittaker_expansion}. Let us use \cref{whittaker_expansion_dual}\ref{whittaker_expansion_dual_coefficient} to find the coefficient of $\mathbf{x}^\alpha = x_1^2x_2x_3^3$ in $W_\lambda(\mathbf{x};q)$. While this calculation is ultimately equivalent to the one in \cref{eg_whittaker_expansion} by the orbit-stabilizer theorem, it proceeds by a rather different route.

We can write an arbitrary element $\nil\in\Nil_\alpha$ as
$$
\nil = \begin{bmatrix}
0 & 0 & a & b & c & d \\
0 & 0 & e & f & g & h \\
0 & 0 & 0 & i & j & k \\
0 & 0 & 0 & 0 & 0 & 0 \\
0 & 0 & 0 & 0 & 0 & 0 \\
0 & 0 & 0 & 0 & 0 & 0
\end{bmatrix}, \quad \text{ where } a,b,c,d,e,f,g,h,i,j,k\in\qinvF.
$$
We enumerate all $\nil$ such that $\JF(\nil) = \lambda$. By \cref{matrix_to_partition}, this is equivalent to $\rank(\nil) = 2$ and $\nil^2 = 0$. Note that
$$
\nil^2 = \begin{bmatrix}
0 & 0 & 0 & ai & aj & ak \\
0 & 0 & 0 & ei & ej & ek \\
0 & 0 & 0 & 0 & 0 & 0 \\
0 & 0 & 0 & 0 & 0 & 0 \\
0 & 0 & 0 & 0 & 0 & 0 \\
0 & 0 & 0 & 0 & 0 & 0
\end{bmatrix}.
$$
We consider two cases, depending on whether $a$ and $e$ are both zero.

{\itshape Case 1: $a\neq 0$ or $e\neq 0$.} In this case, since $\nil^2 = 0$, we have $i = j = k = 0$. Then $\scalebox{0.8}{$\begin{bmatrix}a & b & c & d \\ e & f & g & h\end{bmatrix}$}$ is an arbitrary $2\times 4$ matrix of rank $2$ whose first column is nonzero. The number of such matrices is the number of $2\times 4$ matrices of rank $2$ minus the number of $2\times 3$ matrices of rank $2$, which by \cref{q_enumeration}\ref{q_enumeration_matrices} is
$$
q^{-8}(1-q)^2\qn{4}\qn{3} - q^{-6}(1-q)^2\qn{3}\qn{2} = q^{-8}(1-q)^2(1 + q)(1 + q + q^2).
$$

{\itshape Case 2: $a=e=0$.} In this case, we have $\nil^2 = 0$, and so $\scalebox{0.8}{$\begin{bmatrix}b & c & d \\ f & g & h \\ i & j & k\end{bmatrix}$}$ is an arbitrary $3\times 3$ matrix of rank $2$. By \cref{q_enumeration}\ref{q_enumeration_matrices}, the number of such matrices is
$$
q^{-8}(1-q)^2\qn{3}\qn{3}\qn{2} = q^{-8}(1-q)^2(1+q)(1+q+q^2)^2.
$$

In conclusion,
$$
|\{\nil\in\Nil_\alpha : \JF(\nil) = \lambda\}| = q^{-8}(1-q)^2(1 + q)(1 + q + q^2)(2 + q + q^2).
$$
Multiplying by
$$
\frac{q^{\binom{6}{2} - \ncon{\lambda}}(1-q)^{-6+\lambda_1}\prod_{i\ge 1}\qfac{\lambda_i - \lambda_{i+1}}}{\prod_{i\ge 1}\qfac{\alpha_i}} = \frac{q^8(1-q)^{-2}}{1 + q + q^2},
$$
we obtain
\begin{gather*}
[x_1^2x_2x_3^3]W_\lambda(\mathbf{x};q) = (1 + q)(2 + q + q^2) = 2 + 3q + 2q^2 + q^3,
\end{gather*}
in agreement with \cref{eg_whittaker_expansion,eg_whittaker}.
\end{eg}

\begin{rmk}\label{whittaker_dual_remark}
We mention that in the case $\alpha := (1, \dots, 1)$, whence $\Nil_\alpha = \Nil_n$ is the set of $n\times n$ strictly upper-triangular matrices, \cref{whittaker_expansion_dual}\ref{whittaker_expansion_dual_tableau} recovers a known result in the literature. Namely, 
set $\field := \qinvF$. Then \cref{whittaker_expansion_dual}\ref{whittaker_expansion_dual_tableau} concerns standard tableaux $T$ of shape $\lambda\vdash n$, and states that
$$
|\{\nil\in\Nil_n : \JF(\nil;\Unit{\id}) = T\}| = q^{-\binom{n}{2} + \ncon{\lambda}}(1-q)^n\prod_{i=1}^n\qn{T^{(i-1)}_{r_i-1} - T^{(i-1)}_{r_i}},
$$
where we have used the second description of $\qwt(T)$ from \cref{weight_standard} (recall that $r_i$ indexes the row of $T$ containing the entry $i$). We can rewrite this as
$$
|\{\nil\in\Nil_n : \JF(\nil;\Unit{\id}) = T\}| = q^{-\binom{n}{2}}\prod_{i=1}^n(q^{T^{(i-1)}_{r_i}} - q^{T^{(i-1)}_{r_i-1}}),
$$
where when $r_i = 1$, we take $q^{T^{(i-1)}_{r_i-1}}$ to be zero. This formula is implicit in work of Kirillov \cite[Theorem 2.3.2]{kirillov95}, and was stated more explicitly by Borodin \cite[Section 2]{borodin97} (cf.\ \cite{borodin95}). It was connected by Fulman \cite{fulman00} to the Hall--Littlewood functions, and its combinatorics was further studied by Yip \cite{yip18}.
\end{rmk}

\section{The \texorpdfstring{$q$}{q}-Burge correspondence}\label{sec_q_burge}

\noindent In this section we give a probabilistically bijective proof of the Cauchy identity for $q$-Whittaker functions (\cref{q-cauchy_intro}), using the algebra of nilpotent endomorphisms acting on partial flags. We begin by reviewing the notion of a probabilistic bijection in \cref{sec_probabilistic_bijections}. We define our probabilistic bijection in \cref{sec_probabilities}, and show that it proves the Cauchy identity in \cref{sec_double_counting}.

\subsection{Background on probabilistic bijections}\label{sec_probabilistic_bijections}
The following probabilistic generalization of the notion of a bijection, called a {\itshape bijectivization}, was employed by Matveev and Petrov \cite{matveev_petrov17} and by Bufetov and Matveev \cite{bufetov_matveev18}, and introduced more formally by Bufetov and Petrov \cite[Section 2]{bufetov_petrov19}. We prefer to call it a {\itshape probabilistic bijection}, following Aigner and Frieden \cite[Section 4.1]{aigner_frieden}.
\begin{defn}[{\cite[Definition 2.1]{bufetov_petrov19}, \cite[Definition 4.1.1]{aigner_frieden}}]\label{defn_probabilistic_bijection}
Let $K$ be a field (which we will later take to be $\mathbb{Q}$ or $\mathbb{Q}(q)$). A {\itshape weighted set} is a pair $(A,\wtf)$, where $A$ is a set and $\wtf$ is a function from $A$ to $K$. Let $(A,\wtf)$ and $(B,\wtfdual)$ be weighted finite sets, and let $p : A\times B\to K$ and $\tilde{p} : B\times A \to K$ be functions. We say that $(p,\tilde{p})$ give a {\itshape probabilistic bijection} from $(A,\wtf)$ to $(B,\wtfdual)$ if the following three conditions are satisfied:
\begin{enumerate}[label={(B\arabic*)}, leftmargin=42pt, itemsep=2pt]
\item\label{defn_probabilistic_bijection_forward} for all $a\in A$, we have $\sum_{b\in B}p(a,b) = 1$;
\item\label{defn_probabilistic_bijection_backward} for all $b\in B$, we have $\sum_{a\in A}\tilde{p}(b,a) = 1$; and
\item\label{defn_probabilistic_bijection_reversibility} for all $a\in A$ and $b\in B$, we have $\wtf(a)p(a,b) = \wtfdual(b)\tilde{p}(b,a)$.
\end{enumerate}
We call $p(a,b)$ the {\itshape forward probabilities}, $\tilde{p}(b,a)$ the {\itshape backward probabilities}, and condition \ref{defn_probabilistic_bijection_reversibility} {\itshape reversibility}. (While we do not require $p$ and $\tilde{p}$ to take real values in the interval $[0,1]$, we find it convenient to think of them as defining probabilities.) By summing the reversibility condition over all $a\in A$ and $b\in B$, we obtain
\begin{align}\label{probabilistic_bijection_sum}
\sum_{a\in A}\wtf(a) = \sum_{b\in B}\wtfdual(b).
\end{align}
Note that when $\wtf$ and $\wtfdual$ are identically $1$, and both $p$ and $\tilde{p}$ are required to take values in $\{0,1\}$, we recover the notion of a bijection.
\end{defn}

\begin{rmk}\label{backward_probabilities_remark}
While we will find it convenient to work with both the forward probabilities $p$ and the backward probabilities $\tilde{p}$, we point out that there is an alternative way to define the notion of a probabilistic bijection only in terms of $p$ (without $\tilde{p}$). Namely, we adopt the setup of \cref{defn_probabilistic_bijection}, and assume that $\wtfdual$ is never zero. Then requiring conditions \ref{defn_probabilistic_bijection_forward}, \ref{defn_probabilistic_bijection_backward}, and \ref{defn_probabilistic_bijection_reversibility} is equivalent to requiring condition \ref{defn_probabilistic_bijection_forward} and the following condition \ref{defn_probabilistic_bijection_extra}:
\begin{enumerate}[label={(B\arabic*)}, leftmargin=42pt, itemsep=2pt]\setcounter{enumi}{3}
\item\label{defn_probabilistic_bijection_extra} for all $b\in B$, we have $\sum_{a\in A}\wtf(a)p(a,b) = \wtfdual(b)$.
\end{enumerate}
To see this, note that conditions \ref{defn_probabilistic_bijection_backward} and \ref{defn_probabilistic_bijection_reversibility} imply condition \ref{defn_probabilistic_bijection_extra}, by summing over all $a\in A$. Conversely, suppose that conditions \ref{defn_probabilistic_bijection_forward} and \ref{defn_probabilistic_bijection_extra} hold. We define $\tilde{p} : B\times A\to K$ so that condition \ref{defn_probabilistic_bijection_reversibility} holds, i.e.,
$$
\tilde{p}(b,a) := \frac{\wtf(a)p(a,b)}{\wtfdual(b)} \quad \text{ for all } a\in A \text{ and } b\in B.
$$
Then condition \ref{defn_probabilistic_bijection_extra} implies condition \ref{defn_probabilistic_bijection_backward}.
\end{rmk}

\subsection{The \texorpdfstring{$q$}{q}-Burge probabilities}\label{sec_probabilities}
We introduce the forward and backward probabilities of the $q$-Burge correspondence.
\begin{defn}\label{defn_triples}
Let $\alpha,\beta\models n$. We let $\triples{\alpha,\beta}$ denote the set of triples $(F,F',\nil)$, where $F\in\PFl_\alpha$, $F'\in\PFl_\beta$, and $\nil\in\gl_n$ is a nilpotent endomorphism strictly compatible with both $F$ and $F'$. Given $M\in\Mats{\alpha}{\beta}$, we let $\triples{M}$ denote the subset of $\triples{\alpha,\beta}$ of all $(F,F',\nil)$ such that $F\relpos{M}F'$.

Now let
$$
\Tabs{\alpha}{\beta} := \bigsqcup_{\lambda\vdash n}\SSYT(\lambda,\alpha)\times\SSYT(\lambda,\beta)
$$
denote the set of all pairs of semistandard tableaux of the same shape with contents $\alpha$ and $\beta$, respectively. Given $(T,T')\in\Tabs{\alpha}{\beta}$, we let $\triples[T,T']{\alpha,\beta}$ and $\triples[T,T']{M}$ denote the subsets of $\triples{\alpha,\beta}$ and $\triples{M}$, respectively, of all $(F,F',\nil)$ such that $\JF(\nil;F) = T$ and $\JF(\nil;F') = T'$.
\end{defn}

\begin{eg}\label{eg_triples}
Let $M := \scalebox{0.8}{$\begin{bmatrix}0 & 1 \\ 1 & 0\end{bmatrix}$} \in \Mats{\alpha}{\beta}$, where $\alpha, \beta := (1,1)$. Then $\triples[T,T']{M}$ equals
$$
\{(F,F',\nil) : F,F'\in\Fl_2, F \neq F', \text{ and } \nil = 0\in\gl_2\}
$$
if $T = T' = \ytableausmall{1 & 2}$, and is empty otherwise.
\end{eg}

\begin{defn}\label{defn_weights}
Let $\alpha = (\alpha_1, \dots, \alpha_k), \beta = (\beta_1, \dots, \beta_l) \models n$. We define weight functions $\wtf$ on $\Mats{\alpha}{\beta}$ and $\wtfdual$ on $\Tabs{\alpha}{\beta}$ over the field $\mathbb{Q}(q)$ as follows:
$$
\wtf(M) := \frac{(1-q)^{-n}}{\prod_{1 \le i \le k,\hspace*{1pt} 1 \le j \le l}\qfac{M_{i,j}}} \quad \text{ and } \quad \wtfdual(T,T') := \frac{(1-q)^{-\lambda_1}}{\prod_{i\ge 1}\qfac{\lambda_i - \lambda_{i+1}}}\qwt(T)\qwt(T').
$$
We will also regard these weight functions as being defined over $\mathbb{Q}$, by specializing $q$ to any element of $\mathbb{Q}$ other than $\pm 1$.
\end{defn}

The following result shows that the weights $\wtf(M)$ (which appear explicitly in \eqref{q-cauchy_coefficient_intro}) are obtained from the Cauchy identity for $q$-Whittaker functions \eqref{q-cauchy_equation_intro}:
\begin{lem}[{cf.\ \cite[Section 3.2]{aigner_frieden}}]\label{lhs_equivalence}
For all $\alpha = (\alpha_1, \dots, \alpha_k), \beta = (\beta_1, \dots, \beta_l)\models n$, we have
\begin{gather*}
[\mathbf{x}^\alpha\mathbf{y}^\beta]\prod_{\substack{i,j \ge 1,\\ d \ge 0\hphantom{,}}}\frac{1}{1-x_iy_jq^d} = \sum_{M\in\Mats{\alpha}{\beta}}\frac{(1-q)^{-n}}{\prod_{1 \le i \le k,\hspace*{1pt} 1 \le j \le l}\qfac{M_{i,j}}}.
\end{gather*}

\end{lem}

\begin{proof}
First we claim that
\begin{align}\label{lhs_equivalence_claim}
\prod_{d\ge 0}\frac{1}{1-zq^d} = \sum_{m\ge 0}\frac{z^m}{(1-q)^m\qfac{m}}.
\end{align}
To see this, let $m\in\mathbb{N}$, and note that
$$
[z^m]\prod_{d\ge 0}\frac{1}{1-zq^d} = [z^m]\prod_{d\ge 0}\hspace*{2pt}\sum_{a_d\ge 0}(zq^d)^{a_d} = \sum_{\lambda = (\lambda_1, \dots, \lambda_m)}q^{|\lambda|},
$$
where $\lambda$ denotes the partition with at most $m$ nonzero parts with exactly $a_i$ parts equal to $i$, for all $i\ge 1$. Letting $\mu$ denote $\transpose{\lambda}$, we have
$$
\sum_{\lambda = (\lambda_1, \dots, \lambda_m)}q^{|\lambda|} = \sum_{\substack{\mu, \\ \mu_1 \le m}}q^{|\mu|} = \prod_{i=1}^m\frac{1}{1-q^i} = \frac{1}{(1-q)^m\qfac{m}}.
$$
This proves \eqref{lhs_equivalence_claim}. Applying \eqref{lhs_equivalence_claim} with $z = x_iy_j$, we obtain
\begin{gather*}
[\mathbf{x}^\alpha\mathbf{y}^\beta]\prod_{\substack{i,j \ge 1,\\ d \ge 0\hphantom{,}}}\frac{1}{1-x_iy_jq^d} = [\mathbf{x}^\alpha\mathbf{y}^\beta]\prod_{1 \le i \le k,\hspace*{1pt} 1 \le j \le l}\hspace*{2pt}\sum_{M_{i,j}\ge 0}\frac{(x_iy_j)^{M_{i,j}}}{(1-q)^{M_{i,j}}\qfac{M_{i,j}}} \hspace*{108pt}\\
\hspace*{48pt}= \sum_{M\in\Mats{\alpha}{\beta}}\hspace*{2pt}\prod_{1 \le i \le k,\hspace*{1pt} 1 \le j \le l}\frac{1}{(1-q)^{M_{i,j}}\qfac{M_{i,j}}} = \sum_{M\in\Mats{\alpha}{\beta}}\frac{(1-q)^{-n}}{\prod_{1 \le i \le k,\hspace*{1pt} 1 \le j \le l}\qfac{M_{i,j}}}.\qedhere
\end{gather*}

\end{proof}

\begin{defn}\label{defn_probabilities}
Let $1/q$ be a prime power, and set $\field := \qinvF$. For all $M\in\Mats{\alpha}{\beta}$ and $(T,T')\in\Tabs{\alpha}{\beta}$, we define (as functions of $q$)
$$
\pr{M}{T,T'} := \frac{|\triples[T,T']{M}\hspace*{-1pt}|}{|\triples{M}|} \quad \text{ and } \quad \prdual{M}{T,T'} := \frac{|\triples[T,T']{M}\hspace*{-1pt}|}{|\triples[T,T']{\alpha,\beta}\hspace*{-1pt}|},
$$
which we call the {\itshape forward probabilities} and {\itshape backward probabilities}, respectively.
\end{defn}

Our goal is to show that $(\prnoarg,\prdualnoarg)$ gives a probabilistic bijection from $(\Mats{\alpha}{\beta},\wtf)$ to $(\Tabs{\alpha}{\beta},\wtfdual)$ whenever $1/q$ is a prime power (cf.\ \cref{remark_infinitely_many}). Note that by definition, conditions \ref{defn_probabilistic_bijection_forward} and \ref{defn_probabilistic_bijection_backward} of \cref{defn_probabilistic_bijection} hold. Therefore it only remains to show that condition \ref{defn_probabilistic_bijection_reversibility} holds, which we will do in \cref{sec_double_counting} by enumerating $\triples[T,T']{\alpha,\beta}$ in two ways. In the remainder of this subsection, we give alternative characterizations of the forward and backward probabilities, and consider some examples.

\begin{prop}\label{probabilities_interpretation}
Set $\field := \qinvF$. Let $\alpha,\beta\models n$, let $M\in\Mats{\alpha}{\beta}$, and let $(T,T')\in\Tabs{\alpha}{\beta}$.
\begin{enumerate}[label=(\roman*), leftmargin=*]
\item\label{probabilities_interpretation_forward} Let $F\in\PFl_\alpha$ and $F'\in\PFl_\beta$ with $F\relpos{M}F'$. Then $\pr{M}{T,T'}$ equals the probability that for a uniformly random nilpotent endomorphism $\nil$ strictly compatible with both $F$ and $F'$, we have
$$
\JF(\nil;F) = T \quad \text{ and } \quad \JF(\nil;F') = T'.
$$
\item\label{probabilities_interpretation_backward} Let $\nil\in\gl_n$ be nilpotent with $\JF(\nil) = \lambda$. Then $\prdual{M}{T,T'}$ equals the probability that for a uniformly random pair of partial flags $(F,F')\in\PFl_\alpha\times\PFl_\beta$ which are both strictly compatible with $\nil$ and satisfy $\JF(\nil;F) = T$ and $\JF(\nil;F') = T'$, we have $F\relpos{M}F'$.
\end{enumerate}

\end{prop}

\begin{proof}
This follows from \cref{transitive_action,compatible_matrices_adjoint,relative_position_invariance}\ref{relative_position_invariance_action}.
\end{proof}

We give an example of calculating the backward probabilities. We will give an example of calculating the forward probabilities in \cref{eg_probabilities_forward}, after establishing \cref{forward_probabilities_concrete}.
\begin{eg}\label{eg_probabilities_backward}
Set $\field := \qinvF$. We use \cref{probabilities_interpretation}\ref{probabilities_interpretation_backward} to calculate the backward probabilities $\prdual{\cdot}{T,T'}$, where
$$
T := \hspace*{4pt}\begin{ytableau}
1 & 1 & 2 & 2 \\
2
\end{ytableau}\hspace*{4pt} \quad \text{ and } \quad T' := \hspace*{4pt}\begin{ytableau}
1 & 1 & 1 & 2 \\
2
\end{ytableau}\hspace*{4pt}.
$$
Note that $T$ and $T'$ both have shape $\lambda := (4,1)$, and they have contents $\alpha := (2,3)$ and $\beta := (3,2)$, respectively. We take
$$
\nil := \begin{bmatrix}
0 & 0 & 0 & 0 & 1 \\
0 & 0 & 0 & 0 & 0 \\
0 & 0 & 0 & 0 & 0 \\
0 & 0 & 0 & 0 & 0 \\
0 & 0 & 0 & 0 & 0
\end{bmatrix},
$$
so that $\nil$ is nilpotent with $\JF(\nil) = \lambda$.

Now let $F\in\PFl_\alpha$ and $F'\in\PFl_\beta$ denote partial flags which are both strictly compatible with $\nil$, and satisfy $\JF(\nil;F) = T$ and $\JF(\nil;F') = T'$. That is, $F_1\in\Gr_{2,5}$ and $F'_1\in\Gr_{3,5}$ such that
$$
\spn{\unit{1}} \subseteq F_1, F'_1 \subseteq \spn{\unit{1},\unit{2},\unit{3},\unit{4}}.
$$
Let $M\in\Mats{\alpha}{\beta}$ denote the relative position of $(F,F')$. Note that $M$ is uniquely determined by its top-left entry:
$$
M = \begin{bmatrix}
M_{1,1} & 2 - M_{1,1} \\[2pt]
3 - M_{1,1} & M_{1,1}
\end{bmatrix}.
$$
By \cref{relative_position_properties}\ref{relative_position_properties3}, we have $M_{1,1} = \dim(F_1 \cap F'_1)$. In order to determine $M_{1,1}$, we consider two cases, depending on whether $F_1/\spn{\unit{1}} \subseteq F'_1/\spn{\unit{1}}$. 

{\itshape Case 1: $F_1/\spn{\unit{1}} \subseteq F'_1/\spn{\unit{1}}$.} By \cref{q_enumeration}\ref{q_enumeration_Gr}, this occurs with probability
$$
\frac{|\Gr_{1,2}|}{|\Gr_{1,3}|} = \frac{q^{-1}\qbinomsmall{2}{1}}{q^{-2}\qbinomsmall{3}{1}} = \frac{q + q^2}{1 + q + q^2}.
$$
In this case, we have $M_{1,1} = 2$, so
$$
\prdualdisplay{\scalebox{0.8}{$\begin{bmatrix}2 & 0 \\ 1 & 2\end{bmatrix}$}}{\ytableausmall{1 & 1 & 2 & 2 \\ 2},\hspace*{2pt}\ytableausmall{1 & 1 & 1 & 2 \\ 2}} = q(1+q)(1+q+q^2)^{-1}.
$$

{\itshape Case 2: $F_1/\spn{\unit{1}} \nsubseteq F'_1/\spn{\unit{1}}$.} In this case, we have $M_{1,1} = 1$, so
$$
\prdualdisplay{\scalebox{0.8}{$\begin{bmatrix}1 & 1 \\ 2 & 1\end{bmatrix}$}}{\ytableausmall{1 & 1 & 2 & 2 \\ 2},\hspace*{2pt}\ytableausmall{1 & 1 & 1 & 2 \\ 2}} = 1 - q(1+q)(1+q+q^2)^{-1} = (1 + q + q^2)^{-1}.
$$

The remaining backward probability of the form $\prdual{\cdot}{T,T'}$ is
\begin{gather*}
\prdualdisplay{\scalebox{0.8}{$\begin{bmatrix}0 & 2 \\ 3 & 0\end{bmatrix}$}}{\ytableausmall{1 & 1 & 2 & 2 \\ 2},\hspace*{2pt}\ytableausmall{1 & 1 & 1 & 2 \\ 2}} = 0.\qedhere
\end{gather*}

\end{eg}

\begin{rmk}\label{fixed_flag_remark}
\cref{probabilities_interpretation}\ref{probabilities_interpretation_backward} gives an interpretation for the backward probabilities $\prdualnoarg$, by fixing $\nil$ and considering $F$ and $F'$. It is natural to wonder whether we can also fix $F$; we give an example to show that this is not possible. Specifically, we show that for fixed $M$, $T$, $T'$, and $\nil$, and for a given $F\in\PFl_\alpha$ strictly compatible with $\nil$ satisfying $\JF(\nil;F) = T$, whether or not there exists $F'\in\PFl_\beta$ strictly compatible with $\nil$ satisfying $F\relpos{M}F'$ and $\JF(\nil;F') = T'$ can depend on the choice of $F$. A consequence is that $\Stab_{\GL_n}\hspace*{-1pt}(\nil)\subseteq\GL_n$ does not necessarily act transitively on the set of $F\in\PFl_\alpha$ such that $F$ is strictly compatible with $\nil$ and $\JF(\nil;F) = T$.

For our example, we let $\lambda := (2,1)$ and $\alpha,\beta := (1,1,1)$, and take
$$
M := \begin{bmatrix}
1 & 0 & 0 \\
0 & 0 & 1 \\
0 & 1 & 0
\end{bmatrix},\quad T := \hspace*{4pt}\begin{ytableau}
1 & 2 \\
3
\end{ytableau}\hspace*{4pt},\quad T' := \hspace*{4pt}\begin{ytableau}
1 & 3 \\
2
\end{ytableau}\hspace*{4pt},\quad \text{ and } \quad \nil := \begin{bmatrix}
0 & 0 & 1 \\
0 & 0 & 0 \\
0 & 0 & 0
\end{bmatrix}.
$$
Then a complete flag $F\in\Fl_3$ is strictly compatible with $\nil$ and satisfies $\JF(\nil;F) = T$ if and only if $F_2 = \spn{\unit{1},\unit{2}}$. Similarly, a complete flag $F'\in\Fl_3$ is strictly compatible with $\nil$ and satisfies $\JF(\nil;F') = T'$ if and only if $F'_1 = \spn{\unit{1}}$ and $F'_2 = \spn{\unit{1}, a\unit{2} + \unit{3}}$ for some $a\in\field$. If $F_1 = \spn{\unit{1}}$, then we have $F\relpos{M}F'$ for all such $F'$. However, if $F_1\neq\spn{\unit{1}}$, then there does not exist such an $F'$ with $F\relpos{M}F'$.
\end{rmk}

We now give a more concrete description of the forward probabilities, which we will also need to prove our probabilistic bijection.
\begin{defn}\label{defn_inversion_matrices}
Let $M\in\Mats{\alpha}{\beta}$. We define
\begin{gather*}
\Nil_M := \Nil_\alpha\cap\permdot{M}\Nil_\beta\permdot{M}^{-1},
\end{gather*}
where we recall that $\Nil_\alpha$ was defined in \cref{defn_lie}, and $\permdot{M}$ was defined in \cref{defn_canonical_perm}.
\end{defn}

\begin{eg}\label{eg_inversion_matrices}
Let $M := \scalebox{0.8}{$\begin{bmatrix}1 & 1 \\ 2 & 1\end{bmatrix}$} \in \Mats{\alpha}{\beta}$, where $\alpha := (2,3)$ and $\beta := (3,2)$. Recall from \cref{eg_canonical_perm} that $\perm{M} = 13425$. We calculate that $\permdot{M}\Nil_\beta\permdot{M}^{-1}$ equals
$$
\left\{\begin{bmatrix}
1 & 0 & 0 & 0 & 0 \\
0 & 0 & 0 & 1 & 0 \\
0 & 1 & 0 & 0 & 0 \\
0 & 0 & 1 & 0 & 0 \\
0 & 0 & 0 & 0 & 1
\end{bmatrix} \begin{bmatrix}
0 & 0 & 0 & \ast & \ast \\
0 & 0 & 0 & \ast & \ast \\
0 & 0 & 0 & \ast & \ast \\
0 & 0 & 0 & 0 & 0 \\
0 & 0 & 0 & 0 & 0
\end{bmatrix} \begin{bmatrix}
1 & 0 & 0 & 0 & 0 \\
0 & 0 & 0 & 1 & 0 \\
0 & 1 & 0 & 0 & 0 \\
0 & 0 & 1 & 0 & 0 \\
0 & 0 & 0 & 0 & 1
\end{bmatrix}^{-1}\right\} = \left\{\begin{bmatrix}
0 & \ast & 0 & 0 & \ast \\
0 & 0 & 0 & 0 & 0 \\
0 & \ast & 0 & 0 & \ast \\
0 & \ast & 0 & 0 & \ast \\
0 & 0 & 0 & 0 & 0
\end{bmatrix}\right\},
$$
and so
$$
\Nil_M = \left\{\begin{bmatrix}
0 & 0 & \ast & \ast & \ast \\
0 & 0 & \ast & \ast & \ast \\
0 & 0 & 0 & 0 & 0 \\
0 & 0 & 0 & 0 & 0 \\
0 & 0 & 0 & 0 & 0
\end{bmatrix}\right\} \cap \left\{\begin{bmatrix}
0 & \ast & 0 & 0 & \ast \\
0 & 0 & 0 & 0 & 0 \\
0 & \ast & 0 & 0 & \ast \\
0 & \ast & 0 & 0 & \ast \\
0 & 0 & 0 & 0 & 0
\end{bmatrix}\right\} = \left\{\begin{bmatrix}
0 & 0 & 0 & 0 & \ast \\
0 & 0 & 0 & 0 & 0 \\
0 & 0 & 0 & 0 & 0 \\
0 & 0 & 0 & 0 & 0 \\
0 & 0 & 0 & 0 & 0
\end{bmatrix}\right\},
$$
where $\ast$ denotes an arbitrary element of $\field$.
\end{eg}

\begin{prop}\label{compatible_matrices_pair}
Let $\alpha,\beta\models n$, and let $M\in\mathbb{N}^{k\times l}$ have type $(\alpha,\beta)$.
\begin{enumerate}[label=(\roman*), leftmargin=*]
\item\label{compatible_matrices_pair_enumeration} Let $\nil\in\gl_n$. Then $\nil\in\Nil_M$ if and only if $\nil_{r,s} = 0$ for all $1 \le r,s \le n$ satisfying at least one of the following two conditions:
\begin{itemize}[leftmargin=24pt, itemsep=2pt]
\item if $1 \le i,i' \le k$ are such that $r\in\block{\alpha}{i}$ and $s\in\block{\alpha}{i'}$, then $i \ge i'$; or
\item if $1 \le j,j' \le l$ are such that $\perm{M}^{-1}(r)\in\block{\beta}{j}$ and $\perm{M}^{-1}(s) \in \block{\beta}{j'}$, then $j \ge j'$.
\end{itemize}
In particular,
$$
\Nil_M\cong\prod_{1 \le i < i' \le k,\hspace*{2pt} 1 \le j < j' \le l}\field^{M_{i,j}M_{i',j'}}.
$$
\item\label{compatible_matrices_pair_standard} The set of nilpotent endomorphisms strictly compatible with both $\Unit{\id}\in\PFl_\alpha$ and $\Unit{\perm{M}}\in\PFl_\beta$ is $\Nil_M$.
\item\label{compatible_matrices_pair_flags}
Let $F\in\PFl_\alpha$ and $F'\in\PFl_\beta$ with $F\relpos{M}F'$. Then the set of nilpotent endomorphisms strictly compatible with both $F$ and $F'$ is conjugate to $\Nil_M$.
\end{enumerate}
\end{prop}

\begin{proof}
Part \ref{compatible_matrices_pair_enumeration} follows by an explicit calculation. Part \ref{compatible_matrices_pair_standard} follows from \cref{compatible_matrices}. Part \ref{compatible_matrices_pair_flags} follows from \cref{compatible_matrices_adjoint,relative_position_invariance}\ref{relative_position_invariance_action}.
\end{proof}

\begin{cor}\label{forward_probabilities_concrete}
Set $\field := \qinvF$. Let $\alpha,\beta\models n$, let $M\in\Mats{\alpha}{\beta}$, and let $(T,T')\in\Tabs{\alpha}{\beta}$. Then
$$
\pr{M}{T,T'} = \frac{|\{\nil \in \Nil_M :  \JF(\nil;\Unit{\id}) = T \text{ and } \JF(\nil;\Unit{\perm{M}}) = T'\}|}{|\Nil_M|},
$$
where above, $\Unit{\id}\in\PFl_\alpha$ and $\Unit{\perm{M}}\in\PFl_\beta$.
\end{cor}

\begin{proof}
This follows from taking $F = \Unit{\id}$ and $F' = \Unit{\perm{M}}$ in \cref{probabilities_interpretation}\ref{probabilities_interpretation_forward}, and using \cref{compatible_matrices_pair}\ref{compatible_matrices_pair_standard}.
\end{proof}

\begin{eg}\label{eg_probabilities_forward}
Set $\field := \qinvF$. We use \cref{forward_probabilities_concrete} to calculate the forward probabilities $\pr{M}{\cdot}$, where $M := \scalebox{0.8}{$\begin{bmatrix}1 & 1 \\ 2 & 1\end{bmatrix}$}$, as in \cref{eg_inversion_matrices}. We have $\perm{M} = 13425$ and $|\Nil_M| = q^{-1}$. We can write an arbitrary element $\nil\in\Nil_M$ as
$$
\nil = \begin{bmatrix}
0 & 0 & 0 & 0 & a \\
0 & 0 & 0 & 0 & 0 \\
0 & 0 & 0 & 0 & 0 \\
0 & 0 & 0 & 0 & 0 \\
0 & 0 & 0 & 0 & 0
\end{bmatrix}, \quad \text{ where } a\in\field.
$$
Let $T$ and $T'$ denote $\JF(\nil;\Unit{\id})$ and $\JF(\nil;\Unit{\perm{M}})$, respectively. We consider two cases, depending on whether $a$ is zero.

{\itshape Case 1: $a\neq 0$.} The number of such $\nil$ is $q^{-1}-1$. In this case, we have $T = \ytableausmall{1 & 1 & 2 & 2 \\ 2}$ and $T' = \ytableausmall{1 & 1 & 1 & 2 \\ 2}$, so
$$
\prdisplay{M}{\ytableausmall{1 & 1 & 2 & 2 \\ 2},\ytableausmall{1 & 1 & 1 & 2 \\ 2}} = \frac{q^{-1}-1}{q^{-1}} = 1 - q.
$$

{\itshape Case 2: $a=0$.} The number of such $\nil$ is $1$. In this case, we have $T = \ytableausmall{1 & 1 & 2 & 2 & 2}$ and $T' = \ytableausmall{1 & 1 & 1 & 2 & 2}$, so
$$
\prdisplay{M}{\ytableausmall{1 & 1 & 2 & 2 & 2},\ytableausmall{1 & 1 & 1 & 2 & 2}} = \frac{1}{q^{-1}} = q.
$$

The remaining forward probability of the form $\pr{M}{\cdot}$ is
\begin{gather*}
\prdisplay{M}{\ytableausmall{1 & 1 & 2 \\ 2 & 2},\ytableausmall{1 & 1 & 1 \\ 2 & 2}} = 0.\qedhere
\end{gather*}

\end{eg}

\subsection{Proof of the \texorpdfstring{$q$}{q}-Burge correspondence}\label{sec_double_counting}
We now prove that the forward and backward probabilities defined in \cref{sec_probabilities} give a probabilistic bijection proving the Cauchy identity for $q$-Whittaker functions (\cref{q-cauchy_intro}), whenever $1/q$ is a prime power. (By \cref{remark_infinitely_many}, this implies the identity holds for all $q$.)

If $M$ is a nonnegative-integer matrix whose entries sum to $n$, we let $\qbinomsmall{n}{M}$ denote the $q$-multinomial coefficient from \cref{defn_q-analogue}, by regarding $M$ as a weak composition of $n$. We also recall the set of triples $\triples{M}$ from \cref{defn_triples}.
\begin{lem}\label{matrices_count}
Set $\field := \qinvF$. Let $\alpha,\beta\models n$, and let $M\in\Mats{\alpha}{\beta}$. Then
\begin{gather*}
|\triples{M}| = q^{-n^2+n+\ncon{\alpha}+\ncon{\beta}}\qbinom{n}{M}.
\end{gather*}

\end{lem}

\begin{proof}
We enumerate all $(F,F',\nil)\in\triples{M}$. The number of $F\in\PFl_\alpha$ equals
$$
q^{-e_2(\alpha)}\qbinom{n}{\alpha},
$$
by \cref{q_enumeration}\ref{q_enumeration_P}. Given $F\in\PFl_\alpha$, the number of $F'\in\PFl_\beta$ with $F\relpos{M}F'$ equals
$$
q^{-\sum_{1 \le i' \le i \le k,\hspace*{1pt} 1 \le j < j' \le l}M_{i,j}M_{i',j'}}\frac{\prod_{i=1}^k\qfac{\alpha_i}}{\prod_{1 \le i \le k,\hspace*{1pt} 1 \le j \le l}\qfac{M_{i,j}}},
$$
by \cref{relative_position_invariance}\ref{relative_position_invariance_schubert} and \cref{alternative_orbit_decomposition_q}, where $M\in\mathbb{N}^{k\times l}$. Finally, given $F\in\PFl_\alpha$ and $F'\in\PFl_\beta$ with $F\relpos{M}F'$, the number of nilpotent endomorphisms $\nil$ strictly compatible with both $F$ and $F'$ equals
$$
q^{-\sum_{1 \le i < i' \le k,\hspace*{1pt} 1 \le j < j' \le l}M_{i,j}M_{i',j'}},
$$
by \cref{compatible_matrices_pair}. By \eqref{beta_sum_matrix}, the product of these three terms equals
$$
q^{-e_2(\alpha)-e_2(\beta)}\qbinom{n}{M}.
$$
Applying \cref{n_formula} gives the result.
\end{proof}

We will need the following consequence of \cref{jordan_stabilizer}:
\begin{lem}\label{jordan_form_count}
Set $\field := \qinvF$. For $\lambda\vdash n$, the number of nilpotent $\nil\in\gl_n$ with $\JF(\nil) = \lambda$ equals
\begin{gather*}
\frac{q^{-n^2+n+2\ncon{\lambda}}(1-q)^{n-\lambda_1}\qfac{n}}{\prod_{i\ge 1}\qfac{\lambda_i - \lambda_{i+1}}}.
\end{gather*}

\end{lem}

\begin{proof}
By \cref{transitive_action} and the orbit-stabilizer theorem, the number of such $\nil\in\gl_n$ is $\frac{|\GL_n|}{|\Stab_{\GL_n}\hspace*{-1pt}(\nil)|}$. The result then follows from \cref{jordan_stabilizer,q_enumeration}\ref{q_enumeration_GL}, using the fact that $n+2\ncon{\lambda} = \sum_{i\ge 1}\lambda_i^2$.
\end{proof}

\begin{cor}\label{tableaux_count}
Set $\field := \qinvF$. Let $(T,T')\in\Tabs{\alpha}{\beta}$, where $\lambda\vdash n$ is the shape of both $T$ and $T'$, and let $\nil\in\gl_n$ be nilpotent with $\JF(\nil) = \lambda$.
\begin{enumerate}[label=(\roman*), leftmargin=*]
\item\label{tableaux_count_matrix} For $M\in\Mats{\alpha}{\beta}$, the number of pairs of partial flags $(F,F')\in\PFl_\alpha\times\PFl_\beta$ with relative position $M$ which are both strictly compatible with $\nil$ and satisfy $\JF(\nil;F) = T$ and $\JF(\nil;F') = T'$ equals
$$
\frac{q^{n^2-n-2\ncon{\lambda}}(1-q)^{-n+\lambda_1}\prod_{i\ge 1}\qfac{\lambda_i - \lambda_{i+1}}}{\qfac{n}}|\triples[T,T']{M}\hspace*{-1pt}|.
$$
\item\label{tableaux_count_all} The number of pairs of partial flags $(F,F')\in\PFl_\alpha\times\PFl_\beta$ which are both strictly compatible with $\nil$ and satisfy $\JF(\nil;F) = T$ and $\JF(\nil;F') = T'$ equals
\begin{gather*}
\frac{q^{n^2-n-2\ncon{\lambda}}(1-q)^{-n+\lambda_1}\prod_{i\ge 1}\qfac{\lambda_i - \lambda_{i+1}}}{\qfac{n}}|\triples[T,T']{\alpha,\beta}\hspace*{-1pt}| = q^{-2\ncon{\lambda}+\ncon{\alpha}+\ncon{\beta}}\qwt(T)\qwt(T').
\end{gather*}
\end{enumerate}

\end{cor}

\begin{proof}
For part \ref{tableaux_count_matrix}, note that the number of such $(F,F')$ does not depend on the given $\nil\in\gl_n$, by \cref{transitive_action,compatible_matrices_adjoint,relative_position_invariance}\ref{relative_position_invariance_action}. Therefore it equals
$$
\frac{|\triples[T,T']{M}\hspace*{-1pt}|}{|\{\nil'\in\gl_n : \nil' \text{ is nilpotent and } \JF(\nil') = \lambda\}|},
$$
and the result then follows from \cref{jordan_form_count}. Part \ref{tableaux_count_all} follows from part \ref{tableaux_count_matrix} and \cref{whittaker_expansion}\ref{whittaker_expansion_tableau}.
\end{proof}

We are now ready to prove our main result:
\begin{thm}\label{reversibility}
Set $\field := \qinvF$. Let $\alpha = (\alpha_1, \dots, \alpha_k),\beta = (\beta_1, \dots, \beta_l)\models n$. Then for all $M\in\Mats{\alpha}{\beta}$ and $(T,T')\in\Tabs{\alpha}{\beta}$, we have
\begin{align}\label{reversibility_equation}
\begin{gathered}
\frac{(1-q)^{-n}}{\prod_{1 \le i \le k,\hspace*{1pt} 1 \le j \le l}\qfac{M_{i,j}}}\hspace*{1pt}\pr{M}{T,T'} = \frac{q^{n^2-n-\ncon{\alpha}-\ncon{\beta}}(1-q)^{-n}}{\qfac{n}}|\triples[T,T']{M}\hspace*{-1pt}|\hspace*{96pt} \\
\hspace*{204pt}= \frac{(1-q)^{-\lambda_1}}{\prod_{i\ge 1}\qfac{\lambda_i - \lambda_{i+1}}}\qwt(T)\qwt(T')\hspace*{2pt}\prdual{M}{T,T'},
\end{gathered}
\end{align}
where $\lambda\vdash n$ is the shape of both $T$ and $T'$. Therefore $(\prnoarg,\prdualnoarg)$ gives a probabilistic bijection from $(\Mats{\alpha}{\beta},\wtf)$ to $(\Tabs{\alpha}{\beta},\wtfdual)$. Summing \eqref{reversibility_equation} over all $M\in\Mats{\alpha}{\beta}$ and $(T,T')\in\Tabs{\alpha}{\beta}$, we obtain \eqref{q-cauchy_coefficient_intro}, namely, the coefficient of $\mathbf{x}^\alpha\mathbf{y}^\beta$ in the Cauchy identity for $q$-Whittaker functions.
\end{thm}

\begin{proof}
The first equality in \eqref{reversibility_equation} follows from \cref{matrices_count}, and the second equality follows from \cref{tableaux_count}\ref{tableaux_count_all}. Now note that \eqref{reversibility_equation} implies condition \ref{defn_probabilistic_bijection_reversibility} of \cref{defn_probabilistic_bijection} for the data from \cref{defn_weights,defn_probabilities}. We have already observed that conditions \ref{defn_probabilistic_bijection_forward} and \ref{defn_probabilistic_bijection_backward} hold. Therefore $(\prnoarg,\prdualnoarg)$ gives a probabilistic bijection, as desired.
\end{proof}

\begin{eg}\label{eg_reversibility}
Set $\field := \qinvF$. Let $\alpha := (2,3)$ and $\beta := (3,2)$, and consider
$$
M := \begin{bmatrix}
1 & 1 \\
2 & 1
\end{bmatrix}, \quad T := \hspace*{4pt}\begin{ytableau}
1 & 1 & 2 & 2 \\
2
\end{ytableau}\hspace*{4pt}, \quad \text{ and } \quad T' := \hspace*{4pt}\begin{ytableau}
1 & 1 & 1 & 2 \\
2
\end{ytableau}\hspace*{4pt},
$$
so that $M\in\Mats{\alpha}{\beta}$ and $(T,T')\in\Tabs{\alpha}{\beta}$. Then \eqref{reversibility_equation} implies that
$$
\frac{(1-q)^{-5}}{\qfac{1}\qfac{1}\qfac{2}\qfac{1}}\pr{M}{T,T'} = \frac{(1-q)^{-4}}{\qfac{3}\qfac{1}}\qbinom{3}{2}\qbinom{3}{1}\prdual{M}{T,T'}.
$$
This agrees with \cref{eg_probabilities_forward,eg_probabilities_backward}, where we calculated that $\pr{M}{T,T'} = 1-q$ and $\prdual{M}{T,T'} = (1 + q + q^2)^{-1}$.
\end{eg}

We explain why \cref{reversibility} implies the Cauchy identity for $q$-Whittaker functions:
\begin{proof}[Proof of \cref{q-cauchy_intro}]
By \cref{lhs_equivalence}, it suffices to prove \eqref{q-cauchy_coefficient_intro}, which follows from \eqref{reversibility_equation} whenever $1/q$ is a prime power. Since \eqref{q-cauchy_coefficient_intro} is an identity of rational functions, it therefore holds for all $q$ (cf.\ \cref{remark_infinitely_many}).
\end{proof}

\begin{rmk}\label{other_bijections}
We point out that our $q$-Burge correspondence is not the first probabilistic bijection defined between $(\Mats{\alpha}{\beta},\wtf)$ and $(\Tabs{\alpha}{\beta},\wtfdual)$ in order to prove the Cauchy identity for $q$-Whittaker functions (\cref{q-cauchy_intro}). Namely, Matveev and Petrov \cite{matveev_petrov17} gave two such probabilistic bijections, one of which is a $q$-analogue of the Burge correspondence, and the other a $q$-analogue of the RSK correspondence (cf.\ \cref{sec_classical_burge}). They are based, respectively, on the $q$-column insertion of O'Connell and Pei \cite{o'connell_pei13}, and the $q$-row insertion of Borodin and Petrov \cite{borodin_petrov16}. Aigner and Frieden \cite{aigner_frieden} gave a common $(q,t)$-analogue of these two probabilistic bijections in the case of permutation matrices (i.e.\ when $\alpha = \beta = (1, \dots, 1)$), which is also based on a probabilistic insertion procedure. These probabilistic bijections are rather different from our $q$-Burge correspondence, because they are characterized by local growth rules, as we discuss further in \cref{sec_growth_diagrams}; in particular, our $q$-Burge correspondence is new. One implication of this discussion is that while insertion-based and flag variety-based algorithms coincide in the classical setting, their $q$-analogues are fundamentally different.

We also mention that Imamura, Mucciconi, and Sasamoto \cite[Theorem 1.5]{imamura_mucciconi_sasamoto} have recently given a proof of the Cauchy identity for $q$-Whittaker functions using a rather different combinatorial framework. It would be interesting to establish a connection with this work.
\end{rmk}

\begin{rmk}\label{probabilistic_interpretation}
Set $\field := \qinvF$. Given $\alpha,\beta\models n$, let $(F,F',\nil)$ denote a uniformly random element of $\triples{\alpha,\beta}$. Then the first equality of \eqref{reversibility_equation} has the following probabilistic interpretation: for all $M\in\Mats{\alpha}{\beta}$, we have
\begin{align}\label{probablistic_interpretation_equation}
\Pr\big((F,F') \text{ has relative position } M\big) = \frac{|\triples{M}|}{|\triples{\alpha,\beta}|} \hspace*{1pt}\propto\hspace*{1pt} \qbinom{n}{M}.
\end{align}
In particular, when $\alpha = \beta = (1, \dots, 1)$ (equivalently, $\PFl_\alpha = \PFl_\beta = \Fl_n$), the relative position of the pair of complete flags $(F,F')$ is uniformly distributed over all $n\times n$ permutation matrices. This can be seen directly, as follows. The dimension of the cell of complete flags in relative position $w\in \Sym_n$ to $\Unit{\id}$ is equal to the number of inversions of $w$, while the dimension of the affine space of nilpotent endomorphisms strictly compatible with both $\Unit{\id}$ and such a flag is equal to the number of non-inversions of $w$. The independence on the choice of $w$ then follows from the fact that the sum of the number of inversions and the number of non-inversions of $w$ is always $\binom{n}{2}$. Similarly, the second equality of \eqref{reversibility_equation} implies that for all $(T,T')\in\Tabs{\alpha}{\beta}$, we have
\begin{gather*}
\Pr\big(\hspace*{-2pt}\JF(\nil;F) = T \text{ and } \JF(\nil;F') = T'\big) = \frac{|\triples[T,T']{\alpha,\beta}|}{|\triples{\alpha,\beta}|} \hspace*{1pt}\propto\hspace*{1pt} \frac{(1-q)^{-\lambda_1}}{\prod_{i\ge 1}\qfac{\lambda_i - \lambda_{i+1}}}\qwt(T)\qwt(T'),
\end{gather*}
where $\lambda$ denotes the shape of both $T$ and $T'$.

Note that when $q=1$, the $q$-multinomial coefficient $\qbinomsmall{n}{M}$ in \eqref{probablistic_interpretation_equation} becomes the multinomial coefficient $\binom{n}{M}$. This is proportional to the probability \eqref{contingency_probability} of obtaining $M$ by applying the following procedure $n$ times, starting from the zero matrix: we choose a matrix entry uniformly at random, and add $1$ to it. The resulting distribution on $\Mats{\alpha}{\beta}$ (given by \eqref{contingency_probability}) has been extensively studied in the statistics literature, where such matrices are known as {\itshape contingency tables with margins $(\alpha,\beta)$}; see e.g.\ \cite{agresti92} for a survey. It would be interesting to find a more explicit connection between this body of work and \eqref{probablistic_interpretation_equation}.
\end{rmk}

\section{Combinatorics of the \texorpdfstring{$q$}{q}-Burge correspondence}\label{sec_combinatorics}

\noindent In this section, we explore the combinatorial aspects of the $q$-Burge correspondence (and the associated probabilities) introduced in \cref{sec_q_burge}. We begin by reviewing the classical Burge and RSK correspondences in \cref{sec_classical_burge}. In \cref{sec_burge_limit}, we explain how to obtain the classical Burge correspondence as the $q\to 0$ limit of the $q$-Burge correspondence. In \cref{sec_polynomiality}, we discuss polynomiality of the forward probabilities $\pr{M}{T,T'}$. In \cref{sec_growth_diagrams}, we define growth diagrams for the $q$-Burge correspondence, and show that they do not admit a characterization by local rules, in contrast to insertion-based probabilistic bijections.

\subsection{The classical Burge correspondence}\label{sec_classical_burge}

In this subsection we review the classical Burge correspondence (also known as the {\itshape column RSK correspondence}), first studied by Burge \cite{burge74}. Our exposition follows \cite[Appendix A]{fulton97}, to which we refer for further details.\footnote{We caution that our conventions are such that our matrix $M$ is the transpose of the matrix $A$ in \cite{fulton97}.}
\begin{defn}[{\cite[Section A.4.1]{fulton97}}]\label{defn_column_insertion}
Let $T$ be a semistandard tableau, and let $i\in\mathbb{Z}_{>0}$. We define the {\itshape column insertion of $i$ into $T$} as the semistandard tableau obtained from $T$ by performing the following steps. If $i$ is greater than all the entries of $T$ in column $1$, we place $i$ in a new box at the bottom of column $1$. Otherwise, we replace the smallest entry $j\ge i$ in column $1$ with $i$, and insert $j$ into column $2$ by the same procedure. We continue until we place a number in a new box at the bottom of some column.

Given $M\in\mathbb{N}^{k\times l}$, we define semistandard tableaux $\Ptab{M}$ and $\Qtab{M}$ of the same shape by the following procedure, starting with both $\Ptab{M}$ and $\Qtab{M}$ as the empty tableau. For $j = 1, \dots, l$ and $i = k, \dots, 1$ (i.e.\ we read $M$ column by column, left to right, and within each column from bottom to top), we column-insert $i$ into $\Ptab{M}$ and record $j$ in the new box of $\Qtab{M}$ a total of $M_{i,j}$ times, so that at each step $\Ptab{M}$ and $\Qtab{M}$ have the same shape. Note that if $M$ has type $(\alpha,\beta)$, then $\Ptab{M}$ has content $\alpha$ and $\Qtab{M}$ has content $\beta$. The map $M\mapsto (\Ptab{M},\Qtab{M})$ is called the {\itshape Burge correspondence}. For all $n\in\mathbb{N}$ and $\alpha,\beta\models n$, it restricts to a bijection
\begin{gather*}
\Mats{\alpha}{\beta} \to \Tabs{\alpha}{\beta}, \quad M\mapsto (\Ptab{M},\Qtab{M}).
\end{gather*}

\end{defn}

\begin{eg}\label{eg_burge}
Let $M := \scalebox{0.8}{$\begin{bmatrix}1 & 0 & 1 \\ 2 & 1 & 1 \\ 0 & 1 & 2\end{bmatrix}$}$. Then $\Ptab{M}$ is obtained by column-inserting $2$, $2$, $1$, $3$, $2$, $3$, $3$, $2$, and $1$ into the empty tableau, and $\Qtab{M}$ is obtained by recording $1$, $1$, $1$, $2$, $2$, $3$, $3$, $3$, and $3$. We get
\begin{gather*}
\Ptab{M} = \hspace*{4pt}\begin{ytableau}
1 & 1 & 2 & 2 & 2 \\
2 & 3 & 3 \\
3
\end{ytableau}\hspace*{4pt} \quad \text{ and } \quad \Qtab{M} = \hspace*{4pt}\begin{ytableau}
1 & 1 & 1 & 3 & 3 \\
2 & 2 & 3 \\
3
\end{ytableau}\hspace*{4pt}.\qedhere
\end{gather*}

\end{eg}

\begin{rmk}\label{RSK_to_burge}
The {\itshape Robinson--Schensted--Knuth (RSK) correspondence} is another family of bijections from $\Mats{\alpha}{\beta}$ to $\Tabs{\alpha}{\beta}$. It was introduced by Knuth \cite{knuth70}, based on work of Robinson \cite{robinson38} and Schensted \cite{schensted61}. It is defined similarly to the Burge correspondence, where we use {\itshape row insertion} instead of column insertion, and we read the columns of the given matrix $M\in\mathbb{N}^{k\times l}$ from top to bottom rather than bottom to top. (For this reason, the Burge correspondence is also known as the {\itshape column RSK correspondence}.)

A direct relationship between the RSK and Burge correspondences is explained in \cite[Section A.4.2]{fulton97}. Namely, suppose that $M\mapsto (T,T')$ under the RSK correspondence. If $M$ is a permutation matrix (i.e.\ $T$ and $T'$ are standard tableaux), then
\begin{align}\label{RSK_to_burge_standard}
T = \transpose{\Ptab{M}} \quad \text{ and } \quad T' = \transpose{\Qtab{M}}.
\end{align}
In general, let $M_{\text{rev}}$ (respectively, $M^{\text{rev}}$) denote the matrix obtained from $M$ by reversing the order of its rows (respectively, columns). Then
$$
T = \Ptab{M^\text{rev}} = \evac(\Ptab{M_\text{rev}}) \quad \text{ and } \quad T' = \Qtab{M_{\text{rev}}} = \evac(\Qtab{M^{\text{rev}}}),
$$
where $\evac$ is the {\itshape evacuation} map (also known as the {\itshape Sch\"{u}tzenberger involution}). For example, let $M$ be as in \cref{eg_burge}. Then $M_{\text{rev}} = \scalebox{0.8}{$\begin{bmatrix}0 & 1 & 2 \\ 2 & 1 & 1 \\ 1 & 0 & 1\end{bmatrix}$}$ and $M^{\text{rev}} = \scalebox{0.8}{$\begin{bmatrix}1 & 0 & 1 \\ 1 & 1 & 2 \\ 2 & 1 & 0\end{bmatrix}$}$, and
\begin{gather*}
T = \hspace*{4pt}\begin{ytableau}
1 & 1 & 2 & 2 & 2 & 3 & 3 \\
2 & 3
\end{ytableau}\hspace*{4pt} \quad \text{ and } \quad T' = \hspace*{4pt}\begin{ytableau}
1 & 1 & 1 & 2 & 2 & 3 & 3 \\
3 & 3
\end{ytableau}\hspace*{4pt}.
\end{gather*}

\end{rmk}

The Burge correspondence has the following remarkable transpose symmetry:
\begin{prop}[{\cite[Section A.4.1]{fulton97}}]\label{classical_burge_transpose}
Let $M$ be a nonnegative-integer matrix. Then $\Ptab{\transpose{M}} = \Qtab{M}$.
\end{prop}

Rosso \cite{rosso12} gave a description of the Burge correspondence using partial flag varieties and Jordan forms. In order to state it, we recall work of Spaltenstein \cite{spaltenstein76}.
\begin{defn}[{Spaltenstein \cite{spaltenstein76}}]\label{defn_springer_fiber}
Let $\field$ be algebraically closed. Let $\alpha\models n$, and let $\nil\in\gl_n$ be nilpotent. We define the {\itshape Spaltenstein variety}
$$
\PFl_\alpha^\nil := \{F\in\PFl_\alpha : F \text{ is strictly compatible with } \nil\};
$$
this was denoted $\mathcal{X}^\lambda_\alpha$ in \cref{remark_perverse}, where $\lambda := \JF(\nil)$. For $T\in\SSYT(\lambda,\alpha)$, we define
$$
\PFl_\alpha^{\nil,T} := \{F\in\PFl_\alpha^\nil : \JF(\nil;F) = T\}.
$$
Then the irreducible components of $\PFl_\alpha^\nil$ are precisely the Zariski closures of $\PFl_\alpha^{\nil,T}$, for $T\in\SSYT(\lambda,\alpha)$.
\end{defn}

\begin{thm}[{Rosso \cite[Theorem 4.1]{rosso12}}]\label{classical_burge_dual}
Let $\field$ be algebraically closed. Let $\alpha,\beta\models n$, let $M\in\Mats{\alpha}{\beta}$, and let $\nil\in\gl_n$ be nilpotent such that $\JF(\nil)$ is the shape of both $\Ptab{M}$ and $\Qtab{M}$. Then for all $(F,F')$ in some nonempty Zariski-open subset of $\overline{\PFl_\alpha^{\nil,\Ptab{M}}}\times\overline{\PFl_\beta^{\nil,\Qtab{M}}}$, we have $F\relpos{M}F'$.
\end{thm}

We will establish a dual version of \cref{classical_burge_dual} in \cref{classical_burge_rank}, in which we fix a pair of partial flags $(F,F')$ and consider a generic nilpotent endomorphism $\nil$, rather than fixing $\nil$ and considering a generic pair $(F,F')$.
\begin{rmk}\label{classical_burge_history}
In the case that $M$ is a permutation matrix (so that $\Ptab{M}$ and $\Qtab{M}$ are standard tableaux), \cref{classical_burge_dual,classical_burge_rank} also give a description of the RSK correspondence, using \eqref{RSK_to_burge_standard}. This special case was established in more or less equivalent forms by Gansner \cite[Theorem 2.1]{gansner81} (cf.\ \cite[Theorem 3.2]{saks80}), Spaltenstein \cite[Proposition II.9.8]{spaltenstein82}, and Steinberg \cite{steinberg88} (cf.\ \cite[Remark 5.7(c)]{steinberg76}). We refer to \cite[Section 6]{britz_fomin01} for further discussion. Analogues of \cref{classical_burge_dual,classical_burge_rank} for the general RSK correspondence (rather than the Burge correspondence) have not appeared in the literature.
\end{rmk}

\subsection{From the \texorpdfstring{$q$}{q}-Burge correspondence to the classical Burge correspondence}\label{sec_burge_limit}
In this subsection, we will show that the classical Burge correspondence is the $q\to 0$ limit of the $q$-Burge correspondence:
\begin{thm}\label{burge_limit}
Let $M\in\Mats{\alpha}{\beta}$ and $(T,T')\in\Tabs{\alpha}{\beta}$. Then
$$
\lim_{q\to 0}\pr{M}{T,T'} = \lim_{q\to 0}\prdual{M}{T,T'} = \begin{cases}
1, & \text{ if $(T,T') = (\Ptab{M}, \Qtab{M})$}; \\
0, & \text{ otherwise},
\end{cases}
$$
where the limit is taken over all $q > 0$ such that $1/q$ is a prime power.
\end{thm}

\begin{eg}\label{eg_burge_limit}
Let $M := \scalebox{0.8}{$\begin{bmatrix}1 & 1 \\ 2 & 1\end{bmatrix}$}$, $T := \ytableausmall{1 & 1 & 2 & 2 \\ 2}$, and $T' := \ytableausmall{1 & 1 & 1 & 2 \\ 2}$, as in \cref{eg_reversibility}, and recall that $\pr{M}{T,T'} = 1-q$ and $\prdual{M}{T,T'} = (1 + q + q^2)^{-1}$. Note that
$$
\lim_{q\to 0}\pr{M}{T,T'} = \lim_{q\to 0}\prdual{M}{T,T'} = 1,
$$
so \cref{burge_limit} implies that $(\Ptab{M},\Qtab{M}) = (T,T')$ (which we can verify directly).
\end{eg}

Along the way to proving \cref{burge_limit}, we obtain a dual formulation of Rosso's description of the classical Burge correspondence (where a {\itshape minor of order $r$} of a matrix is the determinant of an $r\times r$ submatrix):
\begin{thm}\label{classical_burge_rank}
Let $M\in\mathbb{N}^{k\times l}$ have type $(\alpha,\beta)$, and let $\nil\in\Nil_M$. Recall that certain entries of $\nil$ are zero according to \cref{compatible_matrices_pair}\ref{compatible_matrices_pair_enumeration}; regard the remaining entries as indeterminates $x_1, \dots, x_m$, where $\Nil_M \cong \field^m$. Let $\mathcal{F} \subseteq \mathbb{Z}[x_1, \dots, x_m]$ denote the finite set of all nonzero minors of all powers of all submatrices of $\nil$ of the form $\nil|_{\Unit{\id}_i}$ (for $1 \le i \le k$) or $\nil|_{\Unit{\perm{M}}_j}$ (for $1 \le j \le l$), in which some monomial appears with coefficient $\pm 1$. Then if $x_1, \dots, x_m$ are evaluated at elements of $\field$ such that $f(x_1, \dots, x_m)\neq 0$ (in $\field$) for all $f\in\mathcal{F}$, we have
\begin{gather*}
\JF(\nil;\Unit{\id}) = \Ptab{M} \quad \text{ and } \quad \JF(\nil;\Unit{\perm{M}}) = \Qtab{M}.
\end{gather*}

\end{thm}

We now turn to proving \cref{burge_limit,classical_burge_rank}. In \cref{classical_burge_rank}, the point of considering polynomials in which some coefficient is $\pm 1$ is that such polynomials are nonzero modulo every prime. This will allow us to give a proof of \cref{burge_limit} which is independent of the characteristic of $\qinvF$ (and hence to take limits over all $q > 0$ such that $1/q$ is a prime power, rather than a power of a fixed prime). This is all made possible by the following result of Poljak \cite{poljak89}:
\begin{thm}[{Poljak \cite{poljak89}}]\label{combinatorial_rank}
Let $n\in\mathbb{N}$, and let $S\subseteq [n]^2$. Define the $n\times n$ matrix $A$ by
$$
A_{i,j} := \begin{cases}
x_{i,j}, & \text{ if $(i,j)\in S$}; \\
0, & \text{ otherwise}
\end{cases} \quad \text{ for } 1 \le i,j \le n,
$$
where $x_{i,j}$ is an indeterminate. Suppose that $d\in\mathbb{N}$ and $r\in [n]$ such that $A^d$ has a minor of order $r$ which is nonzero, regarded as an element of $\mathbb{Z}[x_{i,j} : (i,j) \in S]$. Then there exists a nonzero minor of $A^d$ of order at least $r$ in which some monomial appears with coefficient $\pm 1$.
\end{thm}

We will also need two technical results about counting points over finite fields:
\begin{lem}\label{nonzero_polynomials}
Let $f\in\field[x_1, \dots, x_m]$ be a nonzero polynomial, and for $1 \le i \le m$, let $d_i$ denote the degree of $f$ in the variable $x_i$. Suppose that $\field$ is finite, and $|\field| \ge d_i$ for all $1 \le i \le m$. Then
\begin{gather*}
|\{y\in\field^m : f(y)\neq 0\}| \ge \prod_{i=1}^m(|\field| - d_i).
\end{gather*}

\end{lem}

\begin{proof}
This follows by induction on $m$.
\end{proof}

\begin{lem}\label{lang_weil}
Let $p$ be a prime, let $X$ be an irreducible variety defined over $\overline{\pF{p}}$, and let $Y$ be a nonempty Zariski-open subset of $X$. Then
\begin{align}\label{lang_weil_limit}
\lim_{q\to 0}\frac{|Y(\qinvF)|}{|X(\qinvF)|} = 1,
\end{align}
where the limit is taken over all $q > 0$ (of which there are infinitely many) such that $1/q$ is a power of $p$, and $X$ and $Y$ are both defined over $\qinvF$.
\end{lem}

\begin{proof}
Note that since $\overline{\pF{p}} = \bigcup_{i\ge 1}\pF{p^i}$, there indeed exist infinitely-many powers $1/q$ of $p$ such that $X$ and $Y$ are both defined over $\qinvF$. We assume $1/q$ is of this form. Let $d$ denote the dimension of $X$. Since $X$ is irreducible, $X\setminus Y$ has dimension at most $d-1$. By the Lang--Weil estimate \cite{lang_weil54}, we get
$$
|X(\qinvF)| = q^{-d} + O(q^{-d + \frac{1}{2}}) \quad \text{ and } \quad |(X\setminus Y)(\qinvF)| = O(q^{-d+1}) \quad \text{ as } q \to 0.
$$
Hence
$$
\frac{|Y(\qinvF)|}{|X(\qinvF)|} = 1 - \frac{|(X\setminus Y)(\qinvF)|}{|X(\qinvF)|} \to 1 \quad \text{ as } q \to 0,
$$
as desired.
\end{proof}

\begin{rmk}\label{lang_weil_remark}
We will apply \cref{lang_weil} in the setting of \cref{classical_burge_dual}. In this case, one can obtain the limit \eqref{lang_weil_limit} more directly (without relying on the Lang--Weil estimate), by using \cref{nonzero_polynomials} and the fact that  $\PFl_\alpha^\nil$ admits an affine paving \cite[Theorem 1]{fresse16}.
\end{rmk}

\begin{proof}[Proof of \cref{burge_limit,classical_burge_rank}]
We begin by showing that \cref{classical_burge_rank,burge_limit} hold with $(\Ptab{M},\Qtab{M})$ replaced by some pair of tableaux $(T_M,T_M')$. We will then show that $(T_M,T_M') = (\Ptab{M},\Qtab{M})$.

We adopt the setup of \cref{classical_burge_rank}. Note that $\JF(\nil;\Unit{\id})$ and $\JF(\nil;\Unit{\perm{M}})$ are determined by the Jordan form partitions of $\nil|_{\Unit{\id}_i}$ (for $1 \le i \le k$) and $\nil|_{\Unit{\perm{M}}_j}$ (for $1 \le j \le l$), respectively. In turn, by \cref{matrix_to_partition}, the Jordan form partition of a nilpotent matrix is determined by the ranks of its powers. Finally, the rank of a matrix is the maximum order of a nonzero minor. Therefore by \cref{combinatorial_rank}, $(\JF(\nil;\Unit{\id}),\JF(\nil;\Unit{\perm{M}}))$ is the same (for all fields $\field$) for all evaluations of $\nil$ over $\field$ such that every polynomial in $\mathcal{F}$ is nonzero; we denote it by $(T_M,T_M') \in \Tabs{\alpha}{\beta}$.

For \cref{burge_limit}, note that by \eqref{reversibility_equation}, we have $\lim_{q\to 0}\pr{M}{T,T'} = \lim_{q\to 0}\prdual{M}{T,T'}$ (assuming at least one of the limits exists). By the previous paragraph and \cref{nonzero_polynomials} (using the fact that every polynomial in the set $\mathcal{F}$ defined in \cref{classical_burge_rank} is nonzero modulo every prime), we obtain
\begin{align}\label{burge_limit_proof}
\lim_{q\to 0}\pr{M}{T,T'} = \lim_{q\to 0}\prdual{M}{T,T'} = \begin{cases}
1, & \text{ if $(T,T') = (T_M, T_M')$}; \\
0, & \text{ otherwise},
\end{cases}
\end{align}
where the limit is taken over all $q > 0$ such that $1/q$ is a prime power.

It remains to show that $(T_M,T_M') = (\Ptab{M},\Qtab{M})$. To this end, fix a prime $p$, and let $Y$ denote the nonempty Zariski-open subset of $X := \overline{\PFl_\alpha^{\nil,\Ptab{M}}}\times\overline{\PFl_\beta^{\nil,\Qtab{M}}}$ provided by \cref{classical_burge_dual} with $\field = \overline{\pF{p}}$. Then by \cref{probabilities_interpretation}\ref{probabilities_interpretation_backward}, we have
$$
\prdual{M}{\Ptab{M},\Qtab{M}} \ge \frac{|Y(\qinvF)|}{|X(\qinvF)|} \quad \text{ for all powers $1/q$ of $p$}.
$$
By \cref{lang_weil}, we get $\lim_{q\to 0}\prdual{M}{\Ptab{M},\Qtab{M}} = 1$, where the limit is taken over infinitely-many powers $1/q$ of $p$. Comparing this with \eqref{burge_limit_proof} gives $(T_M,T_M') = (\Ptab{M},\Qtab{M})$.
\end{proof}

We also point out that the transpose symmetry of the classical Burge correspondence (\cref{classical_burge_transpose}) extends to the $q$-Burge correspondence:
\begin{prop}\label{transpose_symmetry}
Let $M\in\Mats{\alpha}{\beta}$ and $(T,T')\in\Tabs{\alpha}{\beta}$. Then
\begin{gather*}
\pr{M}{T,T'} = \pr{\transpose{M}}{T',T} \quad \text{ and } \quad \prdual{M}{T,T'} = \prdual{\transpose{M}}{T',T}.
\end{gather*}

\end{prop}

\begin{proof}
This follows from \cref{relative_position_invariance}\ref{relative_position_invariance_inverse}.
\end{proof}

\subsection{Polynomiality of the \texorpdfstring{$q$}{q}-Burge probabilities}\label{sec_polynomiality}
In this subsection, we discuss the properties of the forward probabilities $\pr{M}{T,T'}$, regarded as functions of $q$. We have analogous considerations for the backward probabilities $\prdual{M}{T,T'}$. Given the direct relation \eqref{reversibility_equation}, it suffices to consider only the forward probabilities, and we make this choice because the backward probabilities are not always polynomials in $q$ (see \cref{eg_probabilities_backward}).

We present the following significant open problems. We anticipate that parts \ref{probabilities_problem_polynomial} and \ref{probabilities_problem_nonnegative} have positive answers.
\begin{prob}\label{probabilities_problem}
Let $M\in\Mats{\alpha}{\beta}$ and $(T,T')\in\Tabs{\alpha}{\beta}$.
\begin{enumerate}[label=(\roman*), leftmargin=*]
\item\label{probabilities_problem_polynomial} Is $\pr{M}{T,T'}$ a polynomial function of $q$, i.e., does there exist $f\in\mathbb{Q}[x]$ such that whenever $1/q$ is a prime power and $\field = \qinvF$, we have $\pr{M}{T,T'} = f(q)$? In this case, does $f$ have integer coefficients?
\item\label{probabilities_problem_nonnegative} In the case that $\pr{M}{T,T'}$ is a polynomial function of $q$, is this polynomial either identically zero, or positive for all $q\in (0,1)$?
\item\label{probabilities_problem_combinatorics} Find a combinatorial interpretation of $\pr{M}{T,T'}$.
\end{enumerate}
\end{prob}

Using an argument of Reineke \cite{reineke06}, we show that if $\pr{M}{T,T'}$ is a rational function of $q$, then it is a polynomial with rational coefficients.
\begin{lem}[{cf.\ \cite[Proposition 6.1]{reineke06}}]\label{rational_to_polynomial}
Let $M\in\Mats{\alpha}{\beta}$ and $(T,T')\in\Tabs{\alpha}{\beta}$. Suppose that there exists a rational function $f\in\mathbb{Q}(x)$ such that whenever $1/q$ is a prime power and $\field = \qinvF$, we have $\pr{M}{T,T'} = f(q)$. Then $f\in\mathbb{Q}[x]$.
\end{lem}

\begin{proof}
Define $g\in\mathbb{Q}(x)$ by
$$
g(x) := x^{\sum_{1 \le i < i' \le k,\hspace*{1pt} 1 \le j < j' \le l}M_{i,j}M_{i',j'}}f(1/x),
$$
where $M\in\mathbb{N}^{k\times l}$. Note that by \cref{forward_probabilities_concrete,compatible_matrices_pair}\ref{compatible_matrices_pair_enumeration}, we have
\begin{align}\label{rational_to_polynomial_equation}
g(1/q) = |\{\nil \in \Nil_M :  \JF(\nil;\Unit{\id}) = T \text{ and } \JF(\nil;\Unit{\perm{M}}) = T'\}|
\end{align}
whenever $1/q$ is a prime power and $\field = \qinvF$. We claim that it suffices to show that $g(x)\in\mathbb{Q}[x]$. Indeed, in this case, $f(x)$ is a Laurent polynomial in $x$. Since $f(q) \in [0,1]$ whenever $1/q$ is a prime power, by taking $1/q$ sufficiently large, we see that $f(x)$ is a polynomial in $x$.

Now we show that $g(x)\in\mathbb{Q}[x]$, following \cite[Proof of Proposition 6.1]{reineke06}. Write $g(x) = \frac{a(x)}{b(x)}$, where $a(x), b(x)\in\mathbb{Q}[x]$. By the Euclidean algorithm, we can write $a(x) = b(x)c(x) + d(x)$, where $c(x),d(x)\in\mathbb{Q}[x]$ and $\deg(d) < \deg(b)$. Then 
$$
g(x) = c(x) + \frac{d(x)}{b(x)}.
$$
We must show that $d = 0$. Suppose otherwise that $d\neq 0$. Take $m\in\mathbb{Z}_{>0}$ so that $mc(x)\in\mathbb{Z}[x]$, and take $1/q$ to be a prime power sufficiently large that
$$
0 < m\left|\frac{d(1/q)}{b(1/q)}\right| < 1.
$$
Then $m(g(1/q) - c(1/q))$ is not an integer, contradicting \eqref{rational_to_polynomial_equation}.
\end{proof}

\begin{rmk}\label{rank_conditions_remark}
By a similar argument as in the proof of \cref{classical_burge_rank}, we can express $\pr{M}{T,T'}$ (up to a power of $q$) in terms of sizes of affine algebraic sets over $\qinvF$. However, this is not sufficient to conclude that $\pr{M}{T,T'}$ is a polynomial in $q$. For example, the equation $x^2 - y^2 = 0$ has $2q^{-1}-1$ solutions in $\qinvF^2$ if $1/q$ is odd, and $q^{-1}$ solutions if $1/q$ is even.
\end{rmk}

\begin{rmk}\label{polynomial_extension}
Let $\alpha,\beta\models n$. We point out that if the forward probability $\pr{M}{T,T'}$ is a polynomial function of $q$ for all $M\in\Mats{\alpha}{\beta}$ and $(T,T')\in\Tabs{\alpha}{\beta}$ (whence the backward probability $\prdual{M}{T,T'}$ is a rational function of $q$ by \eqref{reversibility_equation}), then $(\prnoarg,\prdualnoarg)$ gives a probabilistic bijection from $(\Mats{\alpha}{\beta},\wtf)$ to $(\Tabs{\alpha}{\beta},\wtfdual)$ for all values of $q$. This follows from \cref{reversibility,defn_probabilistic_bijection}, using \cref{remark_infinitely_many}. (The polynomiality assumption is necessary in order to define $(\prnoarg,\prdualnoarg)$ when $1/q$ is not a prime power.)
\end{rmk}

\begin{rmk}\label{decomposition_remark}
We point out that \cref{probabilities_problem}\ref{probabilities_problem_nonnegative} has an interesting reformulation, due to the following result of Bernstein \cite{bernstein15}. Namely, let $f\in\mathbb{Q}[x]$. Then either $f$ is identically zero or $f(q)$ is positive for all $q\in (0,1)$ if and only if $f(x)$ is a sum of terms of the form $cx^a(1-x)^b$, where $a,b\in\mathbb{N}$ and $c \in \mathbb{Q}_{>0}$. Moreover, if $f\in\mathbb{Z}[x]$, then we may require that $c\in\mathbb{Z}_{>0}$ above. This suggests an approach to solving \cref{probabilities_problem}, by constructing a decomposition
$$
\{\nil \in \Nil_M :  \JF(\nil;\Unit{\id}) = T \text{ and } \JF(\nil;\Unit{\perm{M}}) = T'\} \cong \bigsqcup_{i=1}^d \big(\field^{a_i}\times (\field^\times)^{b_i}\big)
$$
for some $a_1, \dots, a_d, b_1, \dots, b_d\in\mathbb{N}$.
\end{rmk}

\begin{rmk}\label{limit_1}
We have described the $q\to 0$ limit of the $q$-Burge correspondence in \cref{sec_burge_limit}. Here we consider the $q\to 1$ limit. Let $M\in\Mats{\alpha}{\beta}$ and $(T,T')\in\Tabs{\alpha}{\beta}$, where $\alpha,\beta\models n$, and suppose that $\pr{M}{T,T'}$ is a polynomial function of $q$. Then we anticipate that the evaluation of this polynomial at $q=1$ equals $1$ if $(T,T') = (\tilde{T}, \tilde{T}')$, and equals zero otherwise, where $\tilde{T}$ and $\tilde{T}'$ denote the unique tableaux of shape $(n)$ with contents $\alpha$ and $\beta$, respectively.

We point out that this would follow from the probabilities $\pr{M}{T,T'}$ being nonnegative at $q=1$. To see this, note that by \cref{polynomial_extension}, $(\prnoarg,\prdualnoarg)$ gives a probabilistic bijection from $(\Mats{\alpha}{\beta},(1-q)^n\wtf)$ to $(\Tabs{\alpha}{\beta},(1-q)^n\wtfdual)$ at $q=1$ (where we have rescaled the weights by $(1-q)^n$ so that they are defined at $q=1$). Observe that $(1-q)^n\wtf(M)$ is nonzero at $q=1$ for all $M\in\Mats{\alpha}{\beta}$, while $(1-q)^n\wtfdual(T,T')$ is zero at $q=1$ for all $(T,T')\in\Tabs{\alpha}{\beta}$ with $(T,T') \neq (\tilde{T}, \tilde{T}')$. By \ref{defn_probabilistic_bijection_extra}, we see that if $\prnoarg$ is always nonnegative at $q=1$, then $\pr{M}{T,T'}$ is always zero for all $(T,T')\in\Tabs{\alpha}{\beta}$ with $(T,T') \neq (\tilde{T}, \tilde{T}')$. Then \ref{defn_probabilistic_bijection_forward} implies that $\pr{M}{\tilde{T}, \tilde{T}'}$ is always $1$.
\end{rmk}

We are able to address \cref{probabilities_problem} when $\alpha = \beta$ and $M$ is a diagonal matrix. This case is tractable because it reduces to enumerating pairs of partial flags $(F,F')$ with $F = F'$. We then use this result to prove \cref{whittaker_expansion_dual}; in fact, the two results are equivalent. See \cref{whittaker_dual_remark} for a discussion of related work.
\begin{prop}\label{probabilities_diagonal}
Set $\field := \qinvF$. Let $\lambda\vdash n$ and $\alpha = (\alpha_1, \dots, \alpha_k)\models n$, and let $M\in\mathbb{N}^{k\times k}$ be the diagonal matrix with diagonal entries $\alpha_1, \dots, \alpha_k$. Then for $T,T'\in\SSYT(\lambda,\alpha)$, we have $\pr{M}{T,T'} = \prdual{M}{T,T'} = 0$ if $T\neq T'$, and
\begin{gather*}
\pr{M}{T,T} =  \frac{q^{\ncon{\lambda} - \ncon{\alpha}}(1-q)^{n-\lambda_1}\prod_{i\ge 1}\qfac{\alpha_i}}{\prod_{i\ge 1}\qfac{\lambda_i - \lambda_{i+1}}}\qwt(T) \quad \text{ and } \quad \prdual{M}{T,T} = \frac{q^{\ncon{\lambda}-\ncon{\alpha}}}{\qwt(T)}.
\end{gather*}
In particular, $\pr{M}{T,T'}$ is a polynomial in $q$ with integer coefficients which is nonnegative for all $q\in [0,1]$.
\end{prop}

\begin{proof}
First we prove the formulas for $\prdual{M}{T,T'}$. Let $\nil\in\gl_n$ be nilpotent with $\JF(\nil) = \lambda$. By \cref{probabilities_interpretation}\ref{probabilities_interpretation_backward}, $\prdual{M}{T,T'}$ equals the probability that for a uniformly random pair of partial flags $(F,F')\in\PFl_\alpha\times\PFl_\alpha$ which are both strictly compatible with $\nil$ and satisfy $\JF(\nil;F) = T$ and $\JF(\nil;F') = T'$, we have $F=F'$. This probability equals $0$ if $T\neq T'$, and by \cref{whittaker_expansion}\ref{whittaker_expansion_tableau}, it equals $\frac{q^{\ncon{\lambda}-\ncon{\alpha}}}{\qwt(T)}$ if $T = T'$. The formulas for $\pr{M}{T,T'}$ then follow from \eqref{reversibility_equation}. Finally, the fact that $\pr{M}{T,T'}$ is a polynomial in $q$ with integer coefficients follows from \cref{rational_to_polynomial} and Gauss's lemma (or alternatively, by \cref{whittaker_weight_dual}), and it is manifestly nonnegative for all $q\in [0,1]$.
\end{proof}

\begin{proof}[Proof of \cref{whittaker_expansion_dual}]
We prove part \ref{whittaker_expansion_dual_tableau}, whence part \ref{whittaker_expansion_dual_coefficient} follows by summing over all $T\in\SSYT(\lambda,\alpha)$. Let $M$ be the diagonal matrix of type $(\alpha,\alpha)$. Then $\Nil_M = \Nil_\alpha$, and by \cref{forward_probabilities_concrete}, we have
$$
\pr{M}{T,T} = \frac{|\{\nil \in \Nil_\alpha :  \JF(\nil;\Unit{\id}) = T\}|}{|\Nil_\alpha|}.
$$
We have $|\Nil_\alpha| = q^{-e_2(\alpha)} = q^{-\binom{n}{2} + \ncon{\alpha}}$, by \cref{n_formula}. The result then follows from \cref{probabilities_diagonal}.
\end{proof}

\subsection{Growth diagrams and the absence of local rules}\label{sec_growth_diagrams}
In this subsection we define growth diagrams for the $q$-Burge correspondence, and give an example to demonstrate that they are {\itshape not} characterized by local growth rules.\footnote{We thank Gabriel Frieden for providing us with this example.} This property distinguishes our $q$-Burge correspondence from insertion-based generalizations of the RSK and Burge correspondences which have been studied in the literature (cf.\ \cref{other_bijections}), which are all characterized by local growth rules.
\begin{defn}\label{defn_growth_diagram}
Let $\alpha = (\alpha_1, \dots, \alpha_k), \beta = (\beta_1, \dots, \beta_l) \models n$. Given $(F,F',\nil)\in\triples{\alpha,\beta}$, we define the {\itshape growth diagram} $\growth{F,F',\nil}$ to be the $(k+1) \times (l+1)$ matrix of partitions defined by
$$
\growth{F,F',\nil}_{i,j} := \JF(\nil|_{F_i\cap F'_j}) \quad \text{ for all } 0 \le i \le k \text{ and } 0 \le j \le l.
$$
Note that the last column of $\growth{F,F',\nil}$ precisely determines $\JF(\nil;F)$, and the last row precisely determines $\JF(\nil;F')$. Now set $\field := \qinvF$, and let $M\in\Mats{\alpha}{\beta}$. We regard $\growth{F,F',\nil}$ for $(F,F',\nil)\in\triples{M}$ as a random variable, denoted $\growth{M}$, where $\triples{M}$ has the uniform distribution. Equivalently, by \cref{compatible_matrices_adjoint}, \cref{relative_position_invariance}\ref{relative_position_invariance_action}, and \cref{compatible_matrices_pair}\ref{compatible_matrices_pair_standard}, the distribution of $\growth{M}$ is that of $\growth{\Unit{\id},\Unit{\perm{M}},\nil}$ for $\nil\in\Nil_M$, where $\Nil_M$ has the uniform distribution.
\end{defn}

\begin{eg}\label{eg_growth_diagram}
Let $M\in\mathbb{N}^{4\times 4}$ be the permutation matrix of $3412\in\Sym_4$. We can write an arbitrary element $\nil\in\Nil_M$ as
$$
\nil = \begin{bmatrix}
0 & a & 0 & 0 \\
0 & 0 & 0 & 0 \\
0 & 0 & 0 & b \\
0 & 0 & 0 & 0
\end{bmatrix}, \quad \text{ where } a,b\in\field.
$$
Then the growth diagram $\growth{\Unit{\id},\Unit{3412},\nil}$ is
$$
\hspace*{4pt}\begin{blockarray}{c>{\hspace*{2pt}}ccccc<{\hspace*{2pt}}}
& 0 & \spn{\unit{3}} & \spn{\unit{3}, \unit{4}} & \spn{\unit{3}, \unit{4}, \unit{1}} & \field^4 \\[8pt]
\begin{block}{c[>{\hspace*{2pt}}ccccc<{\hspace*{2pt}}]}
\centering 0 & \emptytab & \emptytab & \emptytab & \emptytab & \emptytab \bigstrut[t]\\[10pt]
\spn{\unit{1}} & \emptytab & \emptytab & \emptytab & \ydiagramsmall{1} & \ydiagramsmall{1} \\[12pt]
\spn{\unit{1}, \unit{2}} & \emptytab & \emptytab & \emptytab & \ydiagramsmall{1} & \JF\hspace*{-2pt}\Big(\scalebox{0.8}{$\begin{bmatrix}0 & a \\ 0 & 0\end{bmatrix}$}\Big) \\[16pt]
\hspace*{6pt}\spn{\unit{1}, \unit{2}, \unit{3}}\hspace*{6pt} & \emptytab & \ydiagramsmall{1} & \ydiagramsmall{1} & \ydiagramsmall{2} & \JF\hspace*{-2pt}\bigg(\scalebox{0.8}{$\begin{bmatrix}0 & a & 0 \\ 0 & 0 & 0 \\ 0 & 0 & 0\end{bmatrix}$}\bigg) \\[24pt]
\field^4 & \emptytab & \ydiagramsmall{1} & \JF\hspace*{-2pt}\Big(\scalebox{0.8}{$\begin{bmatrix}0 & b \\ 0 & 0\end{bmatrix}$}\Big) & \JF\hspace*{-2pt}\bigg(\scalebox{0.8}{$\begin{bmatrix}0 & 0 & 0 \\ 0 & 0 & b \\ 0 & 0 & 0\end{bmatrix}$}\bigg) & \JF(\nil) \bigstrut[b]\\
\end{block}
\end{blockarray}\hspace*{4pt},\vspace*{-4pt}
$$
where we have labeled row $i$ by $\Unit{\id}_i$ and column $j$ by $\Unit{3412}_j$. We obtain $\growth{M}$ by taking $a$ and $b$ to be uniformly random elements of $\field := \qinvF$.
\end{eg}

Growth diagrams were introduced by Fomin \cite{fomin86,fomin95} in order to study the RSK correspondence. Namely, to each $M\in\mathbb{N}^{k\times l}$, he associated a (deterministic) $(k+1)\times (l+1)$ array of partitions, from which we can read off the pair of tableaux obtained under the RSK correspondence. An important feature of Fomin's growth diagrams is that they are characterized by {\itshape local growth rules}; this means that for all $i,j\ge 1$, the partition indexed by $(i,j)$ only depends on the partitions indexed by $(i-1,j-1)$, $(i-1,j)$, and $(i,j-1)$, along with the entry $M_{i,j}$ (and not on $(i,j)$ or the other entries of $M$). The combinatorial growth rules for the RSK and Burge correspondences are described by van Leeuwen \cite[Sections 3.1 and 3.2]{van_leeuwen05}.

Recall the insertion-based probabilistic bijections of Matveev and Petrov \cite{matveev_petrov17} and Aigner and Frieden \cite{aigner_frieden}, discussed in \cref{other_bijections}. They also have associated growth diagrams, and they are characterized by local growth rules (see the cited references, and also \cite{pei17}). This means that for all $i,j\ge 1$, the distribution of the partition indexed by $(i,j)$, conditioned on the values of the partitions indexed by $(i-1,j-1)$, $(i-1,j)$, and $(i,j-1)$, only depends on the entry $M_{i,j}$. We show that this property does {\itshape not} hold for our $q$-Burge correspondence:
\begin{prop}\label{non-local_growth}
The growth diagrams for the $q$-Burge correspondence do not admit local growth rules. Explicitly, let $M$ and $M'$ be the permutation matrices of $3412$ and $1432$, respectively. Note that $M_{4,4} = M'_{4,4} = 0$. Then conditioned on
$$
\growth{M}_{3,3},\hspace*{2pt} \growth{M'}_{3,3} = (2) \quad \text{ and } \quad \growth{M}_{3,4},\hspace*{2pt} \growth{M'}_{3,4},\hspace*{2pt} \growth{M}_{4,3},\hspace*{2pt} \growth{M'}_{4,3} = (2,1),
$$
we have that $\growth{M}_{4,4}$ equals $(2,2)$ with probability $1$, and $\growth{M'}_{4,4}$ equals $(3,1)$ with probability $1$.
\end{prop}

\begin{proof}
We verify the statement for $\growth{M}$; we can verify the statement for $\growth{M'}$ by a similar calculation. Let us adopt the setup of \cref{eg_growth_diagram}. Then the condition $\growth{M}_{3,3} = (2)$ is satisfied; the condition $\growth{M}_{3,4} = (2,1)$ means precisely that $a\neq 0$; and the condition $\growth{M}_{4,3} = (2,1)$ means precisely that $b\neq 0$. Therefore $\JF(\nil) = (2,2)$, and so $\growth{M}_{4,4}$ equals $(2,2)$ with probability $1$.
\end{proof}

We pose the following open problem, which may lead to an approach to solving \cref{probabilities_problem}:
\begin{prob}\label{growth_problem}
Describe growth diagrams for the $q$-Burge correspondence using (necessarily non-local) growth rules.
\end{prob}

\section{Socle filtrations for type \texorpdfstring{$A$}{A} preprojective algebras}\label{sec_quiver}

\noindent In this section, we study certain modules over the preprojective algebra of a type $A$ quiver (i.e.\ a path). We show that the algebra underlying these modules is closely related to that of pairs of partial flags and strictly compatible nilpotent endomorphisms; see \cref{dictionary} for a high-level dictionary between the two settings. This allows us to apply the $q$-Burge correspondence of \cref{sec_q_burge} to enumerate isomorphism classes of such modules, refined according to their socle filtrations. This develops a connection between the combinatorics of symmetric functions and tableaux on the one hand, and the representation theory of preprojective algebras on the other hand, which we believe is a promising topic for future work; see \cref{dynkin_remark}. We mention that a different connection between tableau combinatorics and quiver representations was introduced by Garver, Patrias, and Thomas \cite{garver_patrias_thomas}, and some of the ideas in this section appear independently in work of Dranowski, Elek, Kamnitzer, and Morton-Ferguson \cite{dranowski_elek_kamnitzer_morton-ferguson}.
\begin{table}[ht]
\begin{center}
\setlength{\tabcolsep}{6pt}
\begin{tabular}{| >{\vspace*{4pt}\centering\arraybackslash}m{3in}<{\vspace*{4pt}} | >{\vspace*{4pt}\centering\arraybackslash}m{3in}<{\vspace*{4pt}} |}
\cline{1-2}
{\bfseries partial flag varieties and nilpotent endomorphisms} & \rule[12pt]{0pt}{0pt}{\bfseries modules over the preprojective algebra $\ppa{Q(k,l)}$} \\ \hline
pair of partial flags $(F,F')$ & \rule[12pt]{0pt}{0pt}module $V$ over the path algebra $\pa{Q(k,l)}$ such that all maps are injective \\ \hline
relative position of $(F,F')$ & \rule[12pt]{0pt}{0pt}isomorphism class of $V$ \\ \hline
nilpotent endomorphism $\nil$ strictly compatible with both $F$ and $F'$ & \rule[12pt]{0pt}{0pt}$\ppa{Q(k,l)}$-module $\double{V}$ which restricts to $V$ as a $\pa{Q(k,l)}$-module \\ \hline
pair of tableaux $(\JF(\nil;F),\JF(\nil;F'))$ & \rule[12pt]{0pt}{0pt}socle filtration of $\double{V}$ \\ \hline
\end{tabular}\vspace*{8pt}
\caption{A high-level dictionary between the objects of \cref{sec_q_burge,sec_quiver}. See \cref{module_relative_position}, \cref{preprojective_bijection}, and \cref{preprojective_jordan} for precise statements.}
\label{dictionary}
\end{center}
\end{table}

\subsection{Background on quiver representations}\label{sec_quiver_background}
We give some background on quiver representations and preprojective algebras. We refer to \cite{kirillov16} for further details.

\begin{defn}\label{defn_quiver}
A {\itshape (finite) quiver} (or {\itshape directed graph}) is pair $Q = (Q_0, Q_1)$, where $Q_0$ is a finite set, and $Q_1$ is a finite multisubset of  $Q_0\times Q_0$. We call the elements of $Q_0$ {\itshape vertices}, and the elements of $Q_1$ {\itshape arrows}. For $a = (i,j)\in Q_1$, we call $i$ the {\itshape source} of $a$, denoted $\source{a}$, and $j$ the {\itshape target} of $a$, denoted $\target{a}$. By definition, we only work with finite quivers.

A {\itshape path} in $Q$ is a directed walk. (We emphasize that unlike the standard use of the term {\itshape path} in graph theory, we do not require non-repetition of vertices and arrows.) We regard the vertices $Q_0$ as paths which have no arrows. The {\itshape path algebra} $\pa{Q}$ of $Q$ is the $\field$-algebra defined as follows. As a vector space, $\pa{Q}$ has a basis indexed by the paths in $Q$; the product $p'p$ of two paths $p$ and $p'$ is the concatenation of $p$ and $p'$ if the last vertex of $p$ is the first vertex of $p'$, and zero otherwise. (Note that our convention is to read paths from right to left.) As a $\field$-algebra, $\pa{Q}$ is generated by $Q_0\sqcup Q_1$. Note that $\pa{Q}$ is finite-dimensional if and only if it has no directed cycles.
\end{defn}

\begin{eg}\label{eg_quiver}
Consider the quiver $Q = (Q_0,Q_1)$ with $Q_0 := \{0,1,2\}$ and $Q_1 := \{a,b\}$, where $a := (2,1)$ and $b := (1,0)$. We may depict $Q$ as follows:
$$
Q = \hspace*{4pt}\begin{tikzpicture}[baseline=(current bounding box.center)]
\tikzstyle{out1}=[inner sep=0,minimum size=1.2mm,circle,draw=black,fill=black,semithick]
\pgfmathsetmacro{\s}{2.40};
\foreach \x in {0,...,2}{
\node[out1](v\x)at(-\x*\s,0)[label={[below=3pt]$\x$}]{};}
\path[-latex',thick](v2)edge node[below=-2pt]{\vphantom{$b$}$a$}(v1) (v1)edge node[below=-2pt]{$b$}(v0);
\end{tikzpicture}\hspace*{4pt}.
$$
In the path algebra $\pa{Q}$, we have that $ba$ is the path from $2$ to $0$, and $ab$ is zero. In fact, $\pa{Q}$ has the basis $\{0,1,2,a,b,ba\}$.
\end{eg}

We will consider (left) $\pa{Q}$-modules, which we assume are finite-dimensional. Such modules are also known as representations of $Q$:
\begin{defn}\label{defn_quiver_representation}
Let $Q = (Q_0,Q_1)$ be a quiver. A {\itshape representation} $V$ of $Q$ is a tuple indexed by $Q_0\sqcup Q_1$, where $V(i)$ (for $i\in Q_0$) is a (finite-dimensional) vector space, and $V(a)$ (for $a\in Q_0$) is a linear map from $V(\source{a})$ to $V(\target{a})$. We define the {\itshape dimension vector} of $V$ as
$$
\dim(V) := (\dim(V(i)))_{i\in Q_0} \in \mathbb{N}^{Q_0}.
$$
We may equivalently regard $V$ as the $\pa{Q}$-module
$$
\bigoplus_{i\in Q_0}V(i),
$$
where $i\in Q_0\subseteq\pa{Q}$ acts as projection onto $V(i)$, and $a\in Q_1\subseteq\pa{Q}$ acts as projection onto $V(\source{a})$ followed by $V(a)$. (Recall that $\pa{Q}$ is generated by $Q_0\sqcup Q_1$, so this defines a $\pa{Q}$-module structure.) Every (finite-dimensional) $\pa{Q}$-module arises in this way, uniquely up to isomorphism \cite[Theorem 1.7]{kirillov16}. We will freely identify $\pa{Q}$-modules with representations.

We call a $\pa{Q}$-module {\itshape indecomposable} if it is not a direct sum of two nonzero submodules. By the Krull--Schmidt theorem (cf.\ \cite[Theorem 1.11]{kirillov16}), every $\pa{Q}$-module is a direct sum of indecomposable modules, and the direct sum is unique up to reordering and isomorphism.
\end{defn}

\begin{eg}\label{eg_quiver_representation}
Let $Q$ be as in \cref{eg_quiver}. Consider the representation
$$
V := \hspace*{4pt}\begin{tikzpicture}[baseline=(current bounding box.base)]
\tikzstyle{out1}=[inner sep=0,minimum size=1.2mm,circle,draw=black,fill=black,semithick]
\pgfmathsetmacro{\s}{2.40};
\node[out1](v2)at(0*\s,0)[label={[below=3pt]$2$},label={[above=-1pt]$\field$}]{};
\node[out1](v1)at(1*\s,0)[label={[below=3pt]$1$},label={[above=-1pt]$\field^2$}]{};
\node[out1](v0)at(2*\s,0)[label={[below=3pt]$0$},label={[above=-1pt]$\field$}]{};
\path[-latex',thick](v2)edge node[above=0pt]{\scalebox{0.8}{$\begin{bmatrix}1 \\ -1\end{bmatrix}$}}(v1) (v1)edge node[above=0pt]{\scalebox{0.8}{$\begin{bmatrix}1 & 1\end{bmatrix}$}}(v0);
\end{tikzpicture}\hspace*{4pt}
$$
of $Q$, which we also regard as a $\pa{Q}$-module. Then $V$ is isomorphic to the direct sum of indecomposable $\pa{Q}$-modules
\begin{gather*}
\hspace*{4pt}\begin{tikzpicture}[baseline=(current bounding box.center)]
\tikzstyle{out1}=[inner sep=0,minimum size=1.2mm,circle,draw=black,fill=black,semithick]
\pgfmathsetmacro{\s}{2.40};
\node[out1](v2)at(0*\s,0)[label={[below=3pt]$2$},label={[above=-1pt]$\field$}]{};
\node[out1](v1)at(1*\s,0)[label={[below=3pt]$1$},label={[above=-1pt]$\field$}]{};
\node[out1](v0)at(2*\s,0)[label={[below=3pt]$0$},label={[above=-1pt]$0$}]{};
\path[-latex',thick](v2)edge node[above=-1pt]{$\id$}(v1) (v1)edge node[above=-1pt]{$0$}(v0);
\end{tikzpicture}\hspace*{4pt} \quad\oplus\quad \hspace*{4pt}\begin{tikzpicture}[baseline=(current bounding box.center)]
\tikzstyle{out1}=[inner sep=0,minimum size=1.2mm,circle,draw=black,fill=black,semithick]
\pgfmathsetmacro{\s}{2.40};
\node[out1](v2)at(0*\s,0)[label={[below=3pt]$2$},label={[above=-1pt]$0$}]{};
\node[out1](v1)at(1*\s,0)[label={[below=3pt]$1$},label={[above=-1pt]$\field$}]{};
\node[out1](v0)at(2*\s,0)[label={[below=3pt]$0$},label={[above=-1pt]$\field$}]{};
\path[-latex',thick](v2)edge node[above=-1pt]{$0$}(v1) (v1)edge node[above=-1pt]{$\id$}(v0);
\end{tikzpicture}\hspace*{4pt}.\qedhere
\end{gather*}

\end{eg}

\begin{defn}\label{defn_preprojective_algebra}
Let $Q = (Q_0, Q_1)$ be a quiver. For every arrow $a\in Q_1$, let $\rev{a}$ denote a reverse arrow, from $\target{a}$ to $\source{a}$. We define the {\itshape double} $\double{Q}$ of $Q$ as the quiver with vertices $Q_0$ and arrows $Q_1 \sqcup \rev{Q_1}$. Note that $\double{Q}$ only depends on the underlying undirected graph of $Q$.

Let $\epsilon: Q_1 \to \field^\times$ be any function. The {\itshape preprojective algebra} $\ppa{Q}$ of $Q$ is defined to be the path algebra $\pa{\double{Q}}$ of $\double{Q}$ modulo the two-sided ideal generated by
$$
\sum_{a\in Q_1}\epsilon(a)(\rev{a}a - a\rev{a}).
$$
This construction does not depend on $\epsilon$, up to isomorphism \cite[Lemma 5.4]{kirillov16}, and therefore $\ppa{Q}$ only depends on the underlying undirected graph of $Q$. We prefer to specify $Q$ and require that $\epsilon$ is the constant function $1$, since we will consider $\ppa{Q}$-modules with additional constraints depending on $Q$. Note that in this case, a $\ppa{Q}$-module is precisely a representation $V$ of $\double{Q}$ satisfying
\begin{align}\label{preprojective_conditions}
\sum_{\substack{a\in Q_1,\\ \source{a}=i}}V(\rev{a})\circ V(a) = \sum_{\substack{a\in Q_1,\\ \target{a}=i}}V(a)\circ V(\rev{a}) \quad \text{ for all } i\in Q_0.
\end{align}
Also, since $\pa{Q}$ is a subalgebra of $\ppa{Q}$, every $\ppa{Q}$-module is also a $\pa{Q}$-module. The associated representation of $Q$ is obtained by forgetting the linear maps $V(\rev{a})$ for all $a\in Q_1$.
\end{defn}

\begin{eg}\label{eg_preprojective_algebra}
We adopt the setup of \cref{eg_quiver_representation}. Let $\double{V}$ denote an arbitrary representation of $\double{Q}$ which restricts to $V$ as a representation of $Q$. Then we can write
$$
\double{V} = \hspace*{4pt}\begin{tikzpicture}[baseline=(current bounding box.base)]
\tikzstyle{out1}=[inner sep=0,minimum size=1.2mm,circle,draw=black,fill=black,semithick]
\pgfmathsetmacro{\s}{2.40};
\node[out1](v2)at(0*\s,0)[label={[below=3pt]$2$},label={[above=-1pt]$\field$}]{};
\node[out1](v1)at(1*\s,0)[label={[below=3pt]$1$},label={[above=-1pt]$\field^2$}]{};
\node[out1](v0)at(2*\s,0)[label={[below=3pt]$0$},label={[above=-1pt]$\field$}]{};
\path[-latex',thick](v2)edge[bend left=15] node[above=0pt]{\scalebox{0.8}{$\begin{bmatrix}1 \\ -1\end{bmatrix}$}}(v1) (v1)edge[bend left=15] node[above=0pt]{\scalebox{0.8}{$\begin{bmatrix}1 & 1\end{bmatrix}$}}(v0) (v1)edge[bend left=15] node[below=0pt]{\scalebox{0.8}{$\begin{bmatrix}a & b\end{bmatrix}$}}(v2) (v0)edge[bend left=15] node[below=0pt]{\scalebox{0.8}{$\begin{bmatrix}c \\ d\end{bmatrix}$}}(v1);
\end{tikzpicture}\hspace*{4pt}\quad \text{ for some } a,b,c,d\in\field.
$$
Such a $\double{V}$ satisfies \eqref{preprojective_conditions} (i.e.\ $\double{V}$ is a $\ppa{Q}$-module) if and only if $a = b = c = -d$.
\end{eg}

\begin{defn}\label{defn_socle}
Let $A$ be a (possibly noncommutative) ring. We call an $A$-module {\itshape simple} if it is nonzero and its only submodules are $0$ and itself. Given an $A$-module $V$, we define its {\itshape socle} $\soc(V)$ to be the sum of all simple submodules of $V$. We define the {\itshape socle filtration} (or {\itshape Loewy series}) of $V$ to be the nested sequence of $A$-modules
$$
\socf{V}{0} = 0 \subseteq \socf{V}{1} = \soc(V) \subseteq \socf{V}{2} \subseteq \socf{V}{3} \subseteq \cdots \subseteq V,
$$
determined by
$$
\socf{V}{j}/\socf{V}{j-1} = \soc(V/\socf{V}{j-1}) \quad \text{ for } j = 1, 2, 3, \dots.
$$

Now suppose that $A = \ppa{Q}$ is the preprojective algebra of a Dynkin quiver $Q$. (More generally, we can take $A$ to be the quotient of the path algebra of any quiver by an {\itshape admissible ideal}; cf.\ \cite[Section III.2]{assem_simson_skowronski06}.) Then we can explicitly describe the simple $A$-modules, and hence the socle filtration of $V$, as follows. For $i\in Q_0$, let $S_i$ be the $A$-module given as a representation by
$$
S_i(i) := \field, \quad S_i(j) := 0 \text{ for all } j\in Q_0\setminus\{i\}, \quad\text{and}\quad S_i(a) := 0 \text{ for all } a\in \double{Q}_1.
$$
Then the simple $A$-modules (up to isomorphism) are precisely $S_i$ for $i\in Q_0$ \cite[Corollary 5.10]{kirillov16}. Letting $\arrowsum$ be the two-sided ideal of $A$ generated by $\double{Q}_1$, we have
$$
\socf{V}{j} = \{v\in V : \arrowsum^j(v) = 0\} \quad \text{ for all } j \in\mathbb{N}.
$$
Equivalently, $\socf{V}{j}$ consists of all vectors which are sent to zero after applying any $j$ arrows of $\double{Q}$.
\end{defn}

\begin{eg}\label{eg_socle}
Let $\double{V}$ be the $\ppa{Q}$-module from \cref{eg_preprojective_algebra}, with $a = b = c = -d$ nonzero. Then the socle filtration of $\double{V}$ is
\begin{center}
\begin{tabular}{rcccl}
\rule[14pt]{0pt}{0pt} & $\vdots$ & & $\vdots$ \\
\rule[16pt]{0pt}{0pt} & $\soc^3(\double{V})$ & $=$ & $\double{V}$ \\
\rule[28pt]{0pt}{0pt} & $\soc^2(\double{V})$ & $=$ & $\hspace*{4pt}\begin{tikzpicture}[baseline=(current bounding box.base)]
\tikzstyle{out1}=[inner sep=0,minimum size=1.2mm,circle,draw=black,fill=black,semithick]
\pgfmathsetmacro{\s}{2.40};
\node[out1](v2)at(0*\s,0)[label={[below=3pt]$2$},label={[above=-1pt]$\field$}]{};
\node[out1](v1)at(1*\s,0)[label={[below=3pt]$1$},label={[above=-1pt]$\spn{\unit{1}-\unit{2}}$}]{};
\node[out1](v0)at(2*\s,0)[label={[below=3pt]$0$},label={[above=-1pt]$\field$}]{};
\path[-latex',thick](v2)edge[bend left=15] (v1) (v1)edge[bend left=15] (v0) (v1)edge[bend left=15] (v2) (v0)edge[bend left=15] (v1);
\end{tikzpicture}\hspace*{4pt}$ \\
\rule[30pt]{0pt}{0pt} & $\soc^1(\double{V})$ & $=$ & $\hspace*{4pt}\begin{tikzpicture}[baseline=(current bounding box.base)]
\tikzstyle{out1}=[inner sep=0,minimum size=1.2mm,circle,draw=black,fill=black,semithick]
\pgfmathsetmacro{\s}{2.40};
\node[out1](v2)at(0*\s,0)[label={[below=3pt]$2$},label={[above=-1pt]$0$}]{};
\node[out1](v1)at(1*\s,0)[label={[below=3pt]$1$},label={[above=-1pt]$\spn{\unit{1}-\unit{2}}$}]{};
\node[out1](v0)at(2*\s,0)[label={[below=3pt]$0$},label={[above=-1pt]$0$}]{};
\path[-latex',thick](v2)edge[bend left=15] (v1) (v1)edge[bend left=15] (v0) (v1)edge[bend left=15] (v2) (v0)edge[bend left=15] (v1);
\end{tikzpicture}\hspace*{4pt}$ \\
\rule[18pt]{0pt}{0pt} & $\soc^0(\double{V})$ & $=$ & $0$ & ,
\end{tabular}
\end{center}\vspace*{4pt}
where the linear map for each arrow above is given by restricting the corresponding linear map in $\double{V}$.
\end{eg}

\subsection{Path quivers}\label{sec_path_quivers}
We now focus our attention on path quivers with a unique sink.
\begin{defn}\label{defn_quiver_A}
Given $k,l\in\mathbb{Z}_{>0}$, let $Q(k,l)$ denote the path quiver on $k+l-1$ vertices with vertices labeled and arrows oriented as follows:
$$
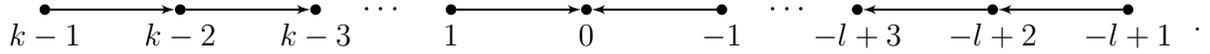
\hspace*{4pt}\begin{tikzpicture}[baseline=(current bounding box.center)]
\tikzstyle{out1}=[inner sep=0,minimum size=1.2mm,circle,draw=black,fill=black,semithick]
\pgfmathsetmacro{\s}{1.80};
\node[out1](v1)at(0*\s,0)[label={[below=3pt]\vphantom{$k-3$}$k-1$}]{};
\node[out1](v2)at(1*\s,0)[label={[below=3pt]\vphantom{$k-3$}$k-2$}]{};
\node[out1](v3)at(2*\s,0)[label={[below=3pt]\vphantom{$k-3$}$k-3$}]{};
\node[out1](v4)at(3*\s,0)[label={[below=3pt]\vphantom{$k-3$}$1$}]{};
\node[out1](v5)at(4*\s,0)[label={[below=3pt]\vphantom{$k-3$}$0$}]{};
\node[out1](v6)at(5*\s,0)[label={[below=3pt]\vphantom{$k-3$}$-1$}]{};
\node[out1](v7)at(6*\s,0)[label={[below=3pt]\vphantom{$k-3$}$-l+3$}]{};
\node[out1](v8)at(7*\s,0)[label={[below=3pt]\vphantom{$k-3$}$-l+2$}]{};
\node[out1](v9)at(8*\s,0)[label={[below=3pt]\vphantom{$k-3$}$-l+1$}]{};
\path[-latex',thick](v1)edge(v2) (v2)edge(v3) (v4)edge(v5) (v9)edge(v8) (v8)edge(v7) (v6)edge(v5);
\node[inner sep=0]at(2.5*\s,0){$\cdots$};
\node[inner sep=0]at(5.5*\s,0){$\cdots$};
\end{tikzpicture}\hspace*{4pt}.
$$
In particular, $Q(k,l)$ is of Dynkin type $A_{k+l-1}$.

The indecomposable $\pa{Q(k,l)}$-modules are indexed (up to isomorphism) by subsets $S\subseteq Q(k,l)_0$ of consecutive vertices (see e.g.\ \cite[Theorem 3.18]{kirillov16}); the corresponding representation $V$ is defined by
$$
V(i) := \begin{cases}
\field, & \text{ if $i\in S$}; \\
0, & \text{ otherwise}
\end{cases} \quad \text{ and } \quad V(a) := \begin{cases}
\id, & \text{ if $s(a),t(a)\in S$}; \\
0, & \text{ otherwise},
\end{cases}
$$
for all $i\in Q(k,l)_0$ and $a\in Q(k,l)_1$. Given $1 \le i \le k$ and $1 \le j \le l$, we let $\indec{i}{j}$ denote the module $V$ corresponding to the subset $S = \{k-i, k-i-1, \dots, -l+j\}$:
$$
\indec{i}{j} = \hspace*{4pt}\begin{tikzpicture}[baseline=(current bounding box.base)]
\tikzstyle{out1}=[inner sep=0,minimum size=1.2mm,circle,draw=black,fill=black,semithick]
\pgfmathsetmacro{\s}{1.80};
\node[out1](v1)at(0*\s,0)[label={[below=3pt]\vphantom{$k-3$}$k-i+1$},label={[above=-1pt]\vphantom{$0\field$}$0$}]{};
\node[out1](v2)at(1*\s,0)[label={[below=3pt]\vphantom{$k-3$}$k-i$},label={[above=-1pt]\vphantom{$0\field$}$\field$}]{};
\node[out1](v3)at(2*\s,0)[label={[below=3pt]\vphantom{$k-3$}$k-i-1$},label={[above=-1pt]\vphantom{$0\field$}$\field$}]{};
\node[out1](v7)at(3.5*\s,0)[label={[below=3pt]\vphantom{$k-3$}$-l+j+1$},label={[above=-1pt]\vphantom{$0\field$}$\field$}]{};
\node[out1](v8)at(4.5*\s,0)[label={[below=3pt]\vphantom{$k-3$}$-l+j$},label={[above=-1pt]\vphantom{$0\field$}$\field$}]{};
\node[out1](v9)at(5.5*\s,0)[label={[below=3pt]\vphantom{$k-3$}$-l+j-1$},label={[above=-1pt]\vphantom{$0\field$}$0$}]{};
\path[-latex',thick](v1)edge node[above=-1pt]{$0$}(v2) (v2)edge node[above=-1pt]{$\id$}(v3) (v9)edge node[above=-1pt]{$0$}(v8) (v8)edge node[above=-1pt]{$\id$}(v7);
\node[inner sep=0]at(-0.5*\s,0){$\cdots$};
\node[inner sep=0]at(2.75*\s,0){$\cdots$};
\node[inner sep=0]at(6*\s,0){$\cdots$};
\end{tikzpicture}\hspace*{4pt}.
$$
Note that the modules $V_{i,j}$ are precisely the indecomposable $\pa{Q(k,l)}$-modules $V$ (up to isomorphism) with $V(0)\neq 0$, or equivalently, such that $V(a)$ is injective for all $a\in Q(k,l)_1$.

Given $M\in\mathbb{N}^{k\times l}$, we define the $\pa{Q(k,l)}$-module
$$
\dec{M} := \bigoplus_{1 \le i \le k,\hspace*{2pt} 1 \le j \le l}\indec{i}{j}^{\oplus M_{i,j}}.
$$
Note that the modules $\dec{M}$ are precisely the (finite-dimensional) $\pa{Q(k,l)}$-modules $V$ (up to isomorphism) such that $V(a)$ is injective for all $a\in Q(k,l)_1$. In fact, $\dec{M}(a)$ is the identity map onto its image, for all $a\in Q(k,l)_1$.
\end{defn}

We relate the modules $\dec{M}$ to pairs of partial flags and relative position.
\begin{defn}\label{defn_module_to_flags}
Let $V$ be a (finite-dimensional) $\pa{Q(k,l)}$-module such that $V(a)$ is injective for all $a\in Q(k,l)_1$. Let $F = (F_0, \dots, F_k)$ be the partial flag of subspaces of $V(0)\cong\field^n$ such that $F_0 := 0$, and $F_i$ (for $1 \le i \le k$) is the image of $V(k-i)$ under the path from vertex $k-i$ to vertex $0$. We similarly define $F' = (F'_0, \dots, F'_l)$ as the partial flag of subspaces of $V(0)\cong\field^n$ such that $F'_0 := 0$, and $F'_j$ (for $1 \le j \le l$) is the image of $V(-l+j)$ under the path from vertex $-l+j$ to vertex $0$. We denote the pair of partial flags $(F,F')$ by $\flags{V}$.
\end{defn}

\begin{lem}\label{module_relative_position}
Let $V$ be a $\pa{Q(k,l)}$-module such that $V(a)$ is injective for all $a\in Q(k,l)_1$. Then $V\cong\dec{M}$, where $M\in\mathbb{N}^{k\times l}$ is the relative position of $\flags{V}$.
\end{lem}

\begin{proof}
This follows from \cref{relative_position_properties}.
\end{proof}

\begin{eg}\label{eg_module_to_flags}
Let $(k,l) := (2,3)$, and let $M := \scalebox{0.8}{$\begin{bmatrix}1 & 0 & 1 \\ 1 & 1 & 1\end{bmatrix}$}\in\mathbb{N}^{k\times l}$. Then by choosing suitable bases, we can write
$$
\dec{M} = \hspace*{4pt}\begin{tikzpicture}[baseline=(current bounding box.base)]
\tikzstyle{out1}=[inner sep=0,minimum size=1.2mm,circle,draw=black,fill=black,semithick]
\pgfmathsetmacro{\s}{3.00};
\node[out1](v1)at(0*\s,0)[label={[below=3pt]$1$},label={[above=-2pt]$\spn{\unit{1},\unit{2}}$}]{};
\node[out1](v0)at(1*\s,0)[label={[below=3pt]$0$},label={[above=-1pt]$\field^5$}]{};
\node[out1](v-1)at(2*\s,0)[label={[below=3pt]$-1$},label={[above=-2pt]$\spn{\unit{1},\unit{3},\unit{4}}$}]{};
\node[out1](v-2)at(3*\s,0)[label={[below=3pt]$-2$},label={[above=-2pt]$\spn{\unit{1},\unit{3}}$}]{};
\path[-latex',thick](v1)edge node[above=-1pt]{$\id$}(v0) (v-2)edge node[above=-1pt]{$\id$}(v-1) (v-1)edge node[above=-1pt]{$\id$}(v0);
\end{tikzpicture}\hspace*{4pt}.
$$
We have $\flags{\dec{M}} = (F,F')$, where
$$
F_0 = 0,\hspace*{2pt} F_1 = \spn{\unit{1},\unit{2}},\hspace*{2pt} F_2 = \field^5 \quad \text{ and } \quad F'_0 = 0,\hspace*{2pt} F'_1 = \spn{\unit{1}, \unit{3}},\hspace*{2pt} F'_2 = \spn{\unit{1},\unit{3},\unit{4}},\hspace*{2pt} F'_3 = \field^5.
$$
Note that $M$ is the relative position of $(F,F')$, in agreement with \cref{module_relative_position}. In fact, we have labeled the basis vectors of $\field^5$ so that $F = \Unit{\id}$ and $F' = \Unit{\perm{M}}$. 
\end{eg}

We introduce notation for the dimension vector of $\dec{M}$:
\begin{defn}\label{defn_compositions_to_dimension}
Given $\alpha = (\alpha_1, \dots, \alpha_k), \beta = (\beta_1, \dots, \beta_l)\models n$, we define $\dimcomp{\alpha}{\beta}\in\mathbb{N}^{\{k-1, k-2, \dots, -l+1\}}$ by
$$
\dimcomp{\alpha}{\beta}_{k-i} := \alpha_1 + \cdots + \alpha_i \text{ for } 1 \le i \le k, \quad \dimcomp{\alpha}{\beta}_{-l+j} := \beta_1 + \cdots + \beta_j \text{ for } 1 \le j \le l.
$$
That is,
$$
\dimcomp{\alpha}{\beta} = (\alpha_1, \alpha_1 + \alpha_2, \dots, \alpha_1 + \cdots + \alpha_k = n = \beta_1 + \cdots + \beta_l, \dots, \beta_1 + \beta_2, \beta_1).
$$
Note that $\dimcomp{\alpha}{\beta} = \dim(\dec{M})$ for all $M\in\mathbb{N}^{k\times l}$ of type $(\alpha,\beta)$.
\end{defn}

\begin{eg}\label{eg_compositions_to_dimension}
Let $M := \scalebox{0.8}{$\begin{bmatrix}1 & 0 & 1 \\ 1 & 1 & 1\end{bmatrix}$}$, as in \cref{eg_module_to_flags}. Then $M$ has type $(\alpha,\beta)$, where $\alpha := (2,3)$ and $\beta := (2,1,2)$. We have
\begin{gather*}
\dimcomp{\alpha}{\beta} = (2,5,3,2) = \dim(\dec{M})\in\mathbb{N}^{\{1,0,-1,-2\}}.\qedhere
\end{gather*}

\end{eg}

We now relate modules over the preprojective algebra of $Q(k,l)$ to nilpotent endomorphisms.
\begin{defn}\label{defn_nilpotent_to_preprojective}
Let $M\in\mathbb{N}^{k\times l}$, and let $\nil$ be a nilpotent endomorphism of $\dec{M}(0)$ strictly compatible with both partial flags of $\flags{\dec{M}}$. We define the $\ppa{Q(k,l)}$-module $\decnil{M}{\nil}$ to be the representation of $\double{Q(k,l)}$ which restricts to $\dec{M}$ as a representation of $Q(k,l)$, and
$$
\decnil{M}{\nil}(\rev{a}) := \begin{cases}
\nil|_{\dec{M}(\target{a})}, & \text{ if $s(a) > 0$}; \\
-\nil|_{\dec{M}(\target{a})}, & \text{ if $s(a) < 0$}
\end{cases} \quad \text{ for } a\in Q(k,l)_1.
$$
We can verify that \eqref{preprojective_conditions} indeed holds.
\end{defn}

\begin{eg}\label{eg_nilpotent_to_preprojective}
Let $M\in\mathbb{N}^{2\times 3}$, and let $(F,F')$ denote $\flags{\dec{M}}$. Then $\decnil{M}{\nil}$ is given by the following representation of $\double{Q(2,3)}$:
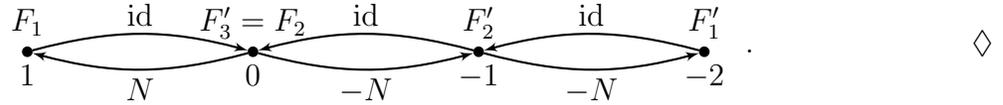
\begin{gather*}
\hspace*{4pt}\begin{tikzpicture}[baseline=(current bounding box.center)]
\tikzstyle{out1}=[inner sep=0,minimum size=1.2mm,circle,draw=black,fill=black,semithick]
\pgfmathsetmacro{\s}{3.00};
\node[out1](v1)at(0*\s,0)[label={[below=3pt]$1$},label={[above=-1pt]\vphantom{$F'_1$}$F_1$}]{};
\node[out1](v0)at(1*\s,0)[label={[below=3pt]$0$},label={[above=-1pt]\vphantom{$F'_1$}$F'_3 = F_2$}]{};
\node[out1](v-1)at(2*\s,0)[label={[below=3pt]$-1$},label={[above=-1pt]\vphantom{$F'_1$}$F'_2$}]{};
\node[out1](v-2)at(3*\s,0)[label={[below=3pt]$-2$},label={[above=-1pt]\vphantom{$F'_1$}$F'_1$}]{};
\path[-latex',thick](v1)edge[bend left=15] node[above=-1pt]{$\id$}(v0) (v0)edge[bend left=15] node[below=-1pt]{\vphantom{$-\nil$}$\nil$}(v1) (v-1)edge[bend right=15] node[above=-1pt]{$\id$}(v0) (v0)edge[bend right=15] node[below=-1pt]{$-\nil$}(v-1) (v-2)edge[bend right=15] node[above=-1pt]{$\id$}(v-1) (v-1)edge[bend right=15] node[below=-1pt]{$-\nil$}(v-2);
\end{tikzpicture}\hspace*{4pt}.\qedhere
\end{gather*}

\end{eg}

\begin{thm}\label{preprojective_bijection}
Let $M\in\mathbb{N}^{k\times l}$. Then the map
$$
\nil \mapsto \decnil{M}{\nil}
$$
is a bijection from the set of nilpotent endomorphisms of $\dec{M}(0)$ strictly compatible with both partial flags of $\flags{\dec{M}}$, to the set of $\ppa{Q(k,l)}$-modules (not considered up to isomorphism) which restrict to $\dec{M}$ as a $\pa{Q(k,l)}$-module. Each such set is also in bijection with $\Nil_M$, and can be regarded as an affine space over $\field$ of dimension $\sum_{1 \le i < i' \le k,\hspace*{1pt} 1 \le j < j' \le l}M_{i,j}M_{i',j'}$.
\end{thm}

\begin{proof}
We prove that the map $\nil \mapsto \decnil{M}{\nil}$ is a bijection; the last statement then follows from \cref{compatible_matrices_pair}. The map is well-defined and injective, so it suffices to prove that it is surjective. Let $V$ be a $\ppa{Q(k,l)}$-module which restricts to $\dec{M}$ as a $\pa{Q(k,l)}$-module. Define $\nil := V(\rev{(1,0)})$ if $k\ge 2$, and $\nil := 0$ otherwise. We will show that $V = \decnil{M}{\nil}$.

Write $\flags{\dec{M}} = (F,F')$. Applying \eqref{preprojective_conditions} for $i = 1, \dots, k-1$, we find that $\nil$ is nilpotent and strictly compatible with $F$, and $V(\rev{i+1,i}) = \nil|_{F_{k-i}}$. Similarly, applying \eqref{preprojective_conditions} for $i = 0, -1, \dots, -l+1$, we find that $\nil$ is strictly compatible with $F'$, and $V(\rev{i-1,i}) = -\nil|_{F'_{l+i}}$. Hence $V = \decnil{M}{\nil}$, as desired.
\end{proof}

\subsection{Reverse plane partitions}\label{sec_RPP}
We recall some background on reverse plane partitions of rectangles, which we will use to study socle filtrations of modules of the form $\decnil{M}{\nil}$. We refer to \cite{hopkins14} for further details on reverse plane partitions.
\begin{defn}\label{defn_RPP}
Given $k,l\in\mathbb{N}$, let $\rect{k}{l}\subseteq\mathbb{Z}^2$ denote the set of $(i,j)\in\mathbb{Z}^2$ satisfying
$$
k-1 \ge i \ge -l+1, \quad |i|+1 \le j \le \min(2k-i-1,2l+i-1), \quad j \equiv |i|+1 \hspace*{-4pt}\pmod{2}.
$$
We depict $\rect{k}{l}$ in the plane as the following $k\times l$ rectangle:
$$
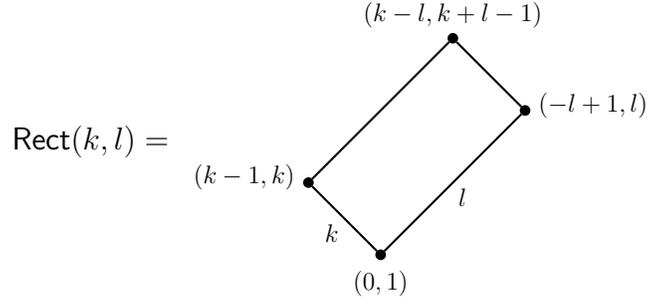
\rect{k}{l} = \hspace*{2pt}\begin{tikzpicture}[baseline=(current bounding box.center)]
\tikzstyle{out1}=[inner sep=0,minimum size=1.2mm,circle,draw=black,fill=black,semithick]
\pgfmathsetmacro{\s}{0.96};
\node[out1](vS)at(0*\s,1*\s)[label={[below=4pt]\scalebox{0.8}{$(0,1)$}}]{};
\node[out1](vW)at(-1*\s,2*\s)[label={[left=1pt]\scalebox{0.8}{$(k-1,k)$}}]{};
\node[out1](vE)at(2*\s,3*\s)[label={[right=1pt]\scalebox{0.8}{$(-l+1,l)$}}]{};
\node[out1](vN)at($(vW)+(vE)-(vS)$)[label={[above=-2pt]\scalebox{0.8}{$(k-l,k+l-1)$}}]{};
\path[thick](vS)edge node[below left=-3pt]{\scalebox{0.8}{$k$}}(vW) edge node[below right=-3pt]{\scalebox{0.8}{$l$}}(vE) (vN)edge(vW) edge(vE);
\end{tikzpicture}\hspace*{2pt}.
$$
We emphasize that in our depiction, the horizontal coordinate decreases from left to right. (Our conventions are consistent with \cref{defn_quiver_A,defn_socle}, which will become apparent in \cref{sec_socle}.)

A {\itshape reverse plane partition} $R = (R(i,j))_{(i,j)\in\rect{k}{l}}$ is a filling of $\rect{k}{l}$ with nonnegative integers, whose entries weakly decrease from southeast to northwest and from southwest to northeast:
$$
R(i,j) \ge R(i+1,j+1) \text{ and } R(i,j) \ge R(i-1,j+1) \quad \text{ for all } (i,j)\in\rect{k}{l}.
$$
For $i\in\mathbb{Z}$ with $k-1 \ge i \ge -l+1$, let $d_i\in\mathbb{N}$ be the sum of all entries of $R$ of the form $R(i,\cdot)$. Note that
$$
d_{k-1} \le \cdots \le d_1 \le d_0 \ge d_{-1} \ge \cdots \ge d_{-l+1},
$$ 
so there exist $\alpha,\beta\models n$ such that $\dimcomp{\alpha}{\beta} = d$. We call $(\alpha,\beta)$ the {\itshape type} of $R$, denoted $\type(R)$. We call the partition
$$
(R(0,1), R(0,3), R(0,5), \dots, R(0,\min(2k-1,2l-1)))\vdash d_0
$$
the {\itshape central partition} of $R$.
\end{defn}

\begin{eg}\label{eg_RPP}
Let $(k,l) := (2,3)$, $\lambda := (4,1)$, $\alpha := (2,3)$, and $\beta := (2,1,2)$. Then
$$
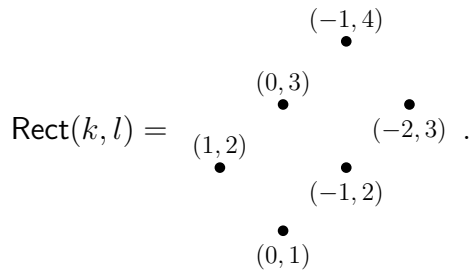
\rect{k}{l} = \hspace*{2pt}\begin{tikzpicture}[baseline=(current bounding box.center)]
\tikzstyle{out1}=[inner sep=0,minimum size=1.2mm,circle,draw=black,fill=black,semithick]
\pgfmathsetmacro{\s}{0.84};
\node[out1](v01)at(0*\s,1*\s)[label={[below=3pt]\scalebox{0.8}{$(0,1)$}}]{};
\node[out1](v-12)at(1*\s,2*\s)[label={[below=3pt]\scalebox{0.8}{$(-1,2)$}}]{};
\node[out1](v-23)at(2*\s,3*\s)[label={[below=3pt]\scalebox{0.8}{$(-2,3)$}}]{};
\node[out1](v12)at(-1*\s,2*\s)[label={[above=-3pt]\scalebox{0.8}{$(1,2)$}}]{};
\node[out1](v03)at(0*\s,3*\s)[label={[above=-3pt]\scalebox{0.8}{$(0,3)$}}]{};
\node[out1](v-14)at(1*\s,4*\s)[label={[above=-3pt]\scalebox{0.8}{$(-1,4)$}}]{};
\end{tikzpicture}\hspace*{2pt}.
$$
There are precisely two reverse plane partitions with central partition $\lambda$ and type $(\alpha,\beta)$:
\begin{gather*}
\scalebox{0.8}{\begin{tabular}{cccc}
& & $0$ & \\
& $1$ & & $2$ \\
$2$ & & $3$ & \\
& $4$ & &
\end{tabular}} \quad \text{ and } \quad \scalebox{0.8}{\begin{tabular}{cccc}
& & $1$ & \\
& $1$ & & $2$ \\
$2$ & & $2$ & \\
& $4$ & &
\end{tabular}}.\qedhere
\end{gather*}

\end{eg}

We now associate a reverse plane partition to a pair of semistandard tableaux of the same shape.
\begin{defn}[{\cite[Section 1]{hopkins14}}]\label{defn_tableaux_to_RPP}
Let $k\in\mathbb{N}$, and let $T\in\SSYT(\lambda)$ have entries bounded above by $k$. The {\itshape Gelfand--Tsetlin pattern} corresponding to $T$ is the array of numbers
$$
\hspace*{4pt}\begin{tabular}{ccccccccc}
& & & & $T^{(1)}_1$ & & & & \\
& & & $T^{(2)}_1$ & & $T^{(2)}_2$ & & & \\
& & $T^{(3)}_1$ & & $T^{(3)}_2$ & & $T^{(3)}_3$ & & \\
& \reflectbox{$\ddots$} & & \hspace*{-8pt}$\ddots$\hspace*{8pt}\reflectbox{$\ddots$}\hspace*{-8pt} & & \hspace*{-8pt}$\ddots$\hspace*{8pt}\reflectbox{$\ddots$}\hspace*{-8pt} & & $\ddots$ & \\
$T^{(k)}_1$ & & $T^{(k)}_2$ & & $\cdots$ & & $T^{(k)}_{k-1}$ & & $T^{(k)}_k$
\end{tabular}\hspace*{4pt}.
$$
Now let $l\in\mathbb{N}$, and let $T'\in\SSYT(\lambda)$ have entries bounded above by $l$. Note that the Gelfand--Tsetlin patterns corresponding to $T$ and $T'$ have the same last row (up to trailing zeros), namely $\lambda$. We rotate the two arrays by $90^\circ$ counterclockwise, reflect the array corresponding to $T'$ over a vertical line, and identify its leftmost column with the rightmost column of the (rotated) array corresponding to $T$. We view the resulting array as being indexed by the rectangle $\rect{k}{l}$, by deleting the entries outside of the rectangle; since the Young diagram of $\lambda$ has at most $\min(k,l)$ rows, such entries are zeros. This produces a reverse plane partition, denoted $\RPP{T}{T'}$.
\end{defn}

\begin{prop}\label{bijection_tableaux_to_RPP}
Let $\lambda\vdash n$, and let $\alpha, \beta\models n$. Then
$$
(T,T') \mapsto \RPP{T}{T'}
$$
is a bijection from $\SSYT(\lambda,\alpha)\times\SSYT(\lambda,\beta)$ to the set of reverse plane partitions with central partition $\lambda$ and type $(\alpha,\beta)$.
\end{prop}

\begin{proof}
This follows from the definitions; see \cite[Section 1]{hopkins14} for further details.
\end{proof}

\begin{eg}
Let $\lambda := (4,1)$, $\alpha := (2,3)$, and $\beta := (2,1,2)$, as in \cref{eg_RPP}. Then the bijection in \cref{bijection_tableaux_to_RPP} is given as follows:
\begin{gather*}
\RPPbig{\ytableausmall{1 & 1 & 2 & 2 \\ 2}}{\ytableausmall{1 & 1 & 2 & 3 \\ 3}} = \scalebox{0.8}{\begin{tabular}{cccc}
& & $0$ & \\
& $1$ & & $2$ \\
$2$ & & $3$ & \\
& $4$ & &
\end{tabular}}, \quad \RPPbig{\ytableausmall{1 & 1 & 2 & 2 \\ 2}}{\ytableausmall{1 & 1 & 3 & 3 \\ 2}} = \scalebox{0.8}{\begin{tabular}{cccc}
& & $1$ & \\
& $1$ & & $2$ \\
$2$ & & $2$ & \\
& $4$ & &
\end{tabular}}.\qedhere
\end{gather*}

\end{eg}

\subsection{Socle filtrations}\label{sec_socle}
In this subsection, we refine \cref{preprojective_bijection} by showing that the bijection $\nil\mapsto\decnil{M}{\nil}$ preserves associated reverse plane partitions.
\begin{lem}\label{socle_facts}
Let $M\in\mathbb{N}^{k\times l}$, let $(F,F') := \flags{\dec{M}}$, let $\nil$ be a nilpotent endomorphism of $\dec{M}(0)$ strictly compatible with both $F$ and $F'$, and let $i\in Q(k,l)_0$.
\begin{enumerate}[label=(\roman*), leftmargin=*]
\item\label{socle_facts_beginning} We have $\socf{\decnil{M}{\nil}}{j}(i) = 0$ for all $0 \le j \le |i|$.
\item\label{socle_facts_parity} We have
$$
\socf{\decnil{M}{\nil}}{j+1}(i) = \socf{\decnil{M}{\nil}}{j}(i) = \begin{cases}
\ker((\nil|_{F_{k-i}})^{\frac{j-i+1}{2}}), & \text{ if $i \ge 0$}; \\
\ker((\nil|_{F'_{l+i}})^{\frac{j+i+1}{2}}), & \text{ if $i \le 0$}
\end{cases}
$$
for all $j \ge |i|+1$ such $j\equiv |i|+1\hspace*{-2pt}\pmod{2}$.
\item\label{socle_facts_end} We have $\socf{\decnil{M}{\nil}}{j}(i) = \decnil{M}{\nil}(i)$ for all $j \ge \min(2k-i-1,2l+i-1)$.
\end{enumerate}
\end{lem}

\begin{proof}
This follows from the definitions.
\end{proof}

\begin{defn}\label{defn_socle_to_RPP}
Let $V$ be a $\ppa{Q(k,l)}$-module such that $V(a)$ is injective for all $a\in Q(k,l)_1$. We let $\socRPP{V}$ denote the filling of the rectangle $\rect{k}{l}$ defined by
$$
\socRPP{V}(i,j) := \dim(\socf{V}{j}(i)/\socf{V}{j-1}(i)) \quad \text{ for all } (i,j)\in\rect{k}{l}.
$$
(We will show in \cref{preprojective_jordan} that $\socRPP{V}$ is a reverse plane partition.) By \cref{preprojective_bijection}, $V$ is isomorphic to $\decnil{M}{\nil}$ for some $M$ and $\nil$. Then by \cref{socle_facts}, we can recover the dimension vectors for the socle filtration of $V$ from $\socRPP{V}$:
$$
\dim(\socf{V}{j}(i)) = \sum_{\substack{1 \le j' \le j, \\ (i,j')\in\rect{k}{l}}}\socRPP{V}(i,j') \quad \text{ for all } i\in Q(k,l)_0 \text{ and } j\in\mathbb{N}.
$$
In particular, $\socRPP{V}$ has type $(\alpha,\beta)$, where $\alpha$ and $\beta$ are such that $\dim(V) = \dimcomp{\alpha}{\beta}$.
\end{defn}

\begin{eg}\label{eg_socle_to_RPP}
We adopt the setup of \cref{eg_module_to_flags}. Let $\nil\in\gl_5$ denote an arbitrary nilpotent endomorphism strictly compatible with both $F$ and $F'$; equivalently, $\nil\in\Nil_M$, by \cref{compatible_matrices_pair}\ref{compatible_matrices_pair_standard}. Then we can write
$$
\nil = \begin{bmatrix}
0 & 0 & 0 & a & b \\
0 & 0 & 0 & 0 & 0 \\
0 & 0 & 0 & 0 & 0 \\
0 & 0 & 0 & 0 & 0 \\
0 & 0 & 0 & 0 & 0
\end{bmatrix}, \quad \text{ where } a,b\in\field.
$$
Using \cref{eg_nilpotent_to_preprojective}, we find that $\socRPP{\decnil{M}{\nil}}$ is one of three reverse plane partitions, depending on the values of $a$ and $b$:
$$
\scalebox{0.8}{\begin{tabular}{cccc}
& & $1$ & \\
& $1$ & & $2$ \\
$2$ & & $2$ & \\
& $4$ & &
\end{tabular}}\hspace*{4pt} \text{ if $a\neq 0$}; \hspace*{16pt} \scalebox{0.8}{\begin{tabular}{cccc}
& & $0$ & \\
& $1$ & & $2$ \\
$2$ & & $3$ & \\
& $4$ & &
\end{tabular}}\hspace*{4pt} \text{ if $a=0$ and $b\neq 0$}; \hspace*{16pt}
\scalebox{0.8}{\begin{tabular}{cccc}
& & $0$ & \\
& $0$ & & $2$ \\
$2$ & & $3$ & \\
& $5$ & &
\end{tabular}}\hspace*{4pt} \text{ if $a=b=0$}.
$$
We see that $\socRPP{\decnil{M}{\nil}}$ has type $(\alpha,\beta)$, where $\alpha := (2,3)$ and $\beta := (2,1,2)$, as in \cref{eg_compositions_to_dimension}.
\end{eg}

\begin{rmk}\label{minuscule_poset}
In the context of \cref{defn_socle_to_RPP}, it is natural to regard $\rect{k}{l}$ as the Hasse diagram of the {\itshape minuscule poset} associated to vertex $0$ of $Q(k,l)$; see \cite{proctor84} and \cite[Section 8.3]{green13} for further details.
\end{rmk}

\begin{thm}\label{preprojective_jordan}
Let $M\in\mathbb{N}^{k\times l}$, let $(F,F') := \flags{\dec{M}}$, and let $\nil$ be a nilpotent endomorphism of $\dec{M}(0)$ strictly compatible with both $F$ and $F'$. Then
$$
\RPP{\JF(\nil;F)}{\JF(\nil;F')} = \socRPP{\decnil{M}{\nil}}.
$$
In particular, $\socRPP{\decnil{M}{\nil}}$ is a reverse plane partition.
\end{thm}

\begin{proof}
We prove the equality of fillings of $\rect{k}{l}$ at all indices of the form $(i,\cdot)$ with $i\ge 0$; the equality at indices with $i\le 0$ follows by a similar argument. Let $T := \JF(\nil;F)$. It suffices to show that
$$
T^{(k-i)}_j = \dim(\socf{\decnil{M}{\nil}}{i+2j-1}(i)/\socf{\decnil{M}{\nil}}{i+2j-2}(i)) \quad \text{ for all } j\in\mathbb{Z}_{>0}.
$$
Both sides equal $\dim(\ker((\nil|_{F_{k-i}})^j)/\ker((\nil|_{F_{k-i}})^{j-1}))$, by \cref{matrix_to_partition,socle_facts}.
\end{proof}

\subsection{Enumerating modules over \texorpdfstring{$\ppa{Q(k,l)}$}{the preprojective algebra}}\label{sec_socle_enumeration}
Using \cref{preprojective_jordan}, we now apply the results of \cref{sec_q_burge} to enumerate $\ppa{Q(k,l)}$-modules $V$ such that $V(a)$ is injective for all $a\in Q(k,l)_1$. We recall that we are counting isomorphism classes $[V]$ of modules, each weighted by $\frac{1}{|\!\Aut(V)|}$.

We make the following definition in analogy with \cref{defn_triples}:
\begin{defn}\label{defn_isomorphism_classes}
Let $\alpha = (\alpha_1, \dots, \alpha_k), \beta = (\beta_1, \dots, \beta_l)\models n$. We let $\ppareps{\alpha,\beta}$ denote the set of all isomorphism classes $[V]$ of $\ppa{Q(k,l)}$-modules $V$, such that $\dim(V) = \dimcomp{\alpha}{\beta}$ and $V(a)$ is injective for all $a\in Q(k,l)_1$. Given $M\in\Mats{\alpha}{\beta}$, we let $\ppareps{M}$ denote the subset of $\ppareps{\alpha,\beta}$ of all $[V]$ such that $V$ is isomorphic to $\dec{M}$ as a $\pa{Q(k,l)}$-module.

Now let $\RPPs{\alpha}{\beta}$ denote the set of all reverse plane partitions of type $(\alpha,\beta)$. Given $R\in\RPPs{\alpha}{\beta}$, we let $\ppareps[R]{\alpha,\beta}$ and $\ppareps[R]{M}$ denote the subsets of $\ppareps{\alpha,\beta}$ and $\ppareps{M}$, respectively, of all $[V]$ such that $\socRPP{V} = R$.
\end{defn}

Note that an endomorphism of a $\pa{Q}$- or $\ppa{Q}$-module $V$ is given by a tuple $(f_i)_{i\in Q_0}$  (satisfying appropriate compatibility conditions), where each $f_i$ is a linear endomorphism of $V(i)$. We can explicitly describe such tuples for $\dec{M}$ and $\decnil{M}{\nil}$:
\begin{lem}\label{module_stabilizer}
Let $M\in\mathbb{N}^{k\times l}$, and let $(F,F') := \flags{\dec{M}}$.
\begin{enumerate}[label=(\roman*), leftmargin=*]
\item\label{module_stabilizer_path_algebra} We have the group isomorphism
$$
\Stab_{\GL_n}\hspace*{-1pt}(F,F') \xrightarrow{\cong} \Aut(\dec{M}),\quad g \mapsto (g|_{\dec{M}(i)})_{i\in Q(k,l)_0}.
$$
\item\label{module_stabilizer_preprojective} Let $\nil$ be a nilpotent endomorphism of $\dec{M}(0)$ strictly compatible with both $F$ and $F'$. Then we have the group isomorphism
\begin{gather*}
\Stab_{\GL_n}\hspace*{-1pt}(F,F',\nil) \xrightarrow{\cong} \Aut(\decnil{M}{\nil}),\quad g \mapsto (g|_{\decnil{M}{\nil}(i)})_{i\in Q(k,l)_0}.
\end{gather*}
\end{enumerate}

\end{lem}

\begin{proof}
This follows from the definitions.
\end{proof}

\cref{module_stabilizer} leads to the following concrete description of $\ppareps{M}$, which is useful in performing calculations. We recall that the subspace $\Nil_M$ of $\gl_n$ was defined in \cref{defn_inversion_matrices}, and that $\GL_n$ acts on $\gl_n$ by conjugation (see \cref{defn_GL_gl}).
\begin{cor}\label{ppareps_concrete_description}
Let $M$ be a nonnegative-integer matrix. Then the distinct elements of $\ppareps{M}$ are precisely the isomorphism classes
$$
[\decnil{M}{\nil}] \quad \text{ for } \nil\in\Nil_M/\Stab_{\GL_n}\hspace*{-1pt}(\Unit{\id}, \Unit{\perm{M}}),
$$
where $\flags{\dec{M}} = (\Unit{\id}, \Unit{\perm{M}})$.
\end{cor}

\begin{proof}
This follows from \cref{compatible_matrices_pair}\ref{compatible_matrices_pair_standard}, \cref{preprojective_bijection}, and \cref{module_stabilizer}\ref{module_stabilizer_path_algebra}.
\end{proof}

\begin{eg}\label{eg_isomorphism_classes}
Let us describe $\ppareps[R]{M}$, where
$$
M := \begin{bmatrix}
0 & 1 & 0 & 0 \\
1 & 0 & 0 & 0 \\
0 & 0 & 0 & 1 \\
0 & 0 & 1 & 0
\end{bmatrix} \quad \text{ and } \quad R := \scalebox{0.8}{\begin{tabular}{ccccccc}
& & & $0$ & & & \\
& & $0$ & & $0$ & & \\
& $0$ & & $0$ & & $0$ & \\
$1$ & & $1$ & & $1$ & & $1$ \\
& $2$ & & $2$ & & $2$ & \\
& & $2$ & & $2$ & & \\
& & & $2$ & & &
\end{tabular}}.
$$
We have
$$
\dec{M} = \hspace*{4pt}\begin{tikzpicture}[baseline=(current bounding box.base)]
\tikzstyle{out1}=[inner sep=0,minimum size=1.2mm,circle,draw=black,fill=black,semithick]
\pgfmathsetmacro{\s}{2.04};
\node[out1](v3)at(0*\s,0)[label={[below=3pt]$3$},label={[above=-2pt]$\spn{\unit{1}}$}]{};
\node[out1](v2)at(1*\s,0)[label={[below=3pt]$2$},label={[above=-2pt]$\spn{\unit{1},\unit{2}}$}]{};
\node[out1](v1)at(2*\s,0)[label={[below=3pt]$1$},label={[above=-2pt]$\spn{\unit{1},\unit{2},\unit{3}}$}]{};
\node[out1](v0)at(3*\s,0)[label={[below=3pt]$0$},label={[above=-1pt]$\field^4$}]{};
\node[out1](v-1)at(4*\s,0)[label={[below=3pt]$-1$},label={[above=-2pt]$\spn{\unit{2},\unit{1},\unit{4}}$}]{};
\node[out1](v-2)at(5*\s,0)[label={[below=3pt]$-2$},label={[above=-2pt]$\spn{\unit{2},\unit{1}}$}]{};
\node[out1](v-3)at(6*\s,0)[label={[below=3pt]$-3$},label={[above=-2pt]$\spn{\unit{2}}$}]{};
\path[-latex',thick](v3)edge(v2) (v2)edge(v1) (v1)edge(v0) (v-3)edge(v-2) (v-2)edge(v-1) (v-1)edge(v0);
\end{tikzpicture}\hspace*{4pt},
$$
where each map is the identity onto its image. By \cref{module_stabilizer}\ref{module_stabilizer_path_algebra}, we have
$$
\Aut(\dec{M}) \cong \Stab_{\GL_4}\hspace*{-1pt}(\Unit{\id}, \Unit{\perm{M}}) = \left\{\begin{bmatrix}
\ast & 0 & \ast & \ast \\
0 & \ast & \ast & \ast \\
0 & 0 & \ast & 0 \\
0 & 0 & 0 & \ast
\end{bmatrix}\right\}\subseteq\GL_4,
$$
where $\ast$ denotes an arbitrary element of $\field$.

We can write an arbitrary element $\nil\in\Nil_M$ as
$$
\nil = \begin{bmatrix}
0 & 0 & a & b \\
0 & 0 & c & d \\
0 & 0 & 0 & 0 \\
0 & 0 & 0 & 0
\end{bmatrix}, \quad \text{ where } a,b,c,d\in\field.
$$
Then we calculate that $\socRPP{\decnil{M}{\nil}} = R$ if and only if $ad \neq bc$. By \cref{ppareps_concrete_description}, the distinct elements of $\ppareps[R]{M}$ are precisely $[\decnil{M}{\nil}]$, for such $\nil$ modulo the action of $\Stab_{\GL_4}\hspace*{-1pt}(\Unit{\id}, \Unit{\perm{M}})$:
\begin{align}\label{eg_isomorphism_classes_equation}
\nil = \scalebox{0.8}{$\begin{bmatrix}
0 & 0 & 1 & 0 \\
0 & 0 & 0 & 1 \\
0 & 0 & 0 & 0 \\
0 & 0 & 0 & 0
\end{bmatrix}$},\hspace*{4pt}
\scalebox{0.8}{$\begin{bmatrix}
0 & 0 & 1 & 1 \\
0 & 0 & 0 & 1 \\
0 & 0 & 0 & 0 \\
0 & 0 & 0 & 0
\end{bmatrix}$},\hspace*{4pt}
\scalebox{0.8}{$\begin{bmatrix}
0 & 0 & 0 & 1 \\
0 & 0 & 1 & 0 \\
0 & 0 & 0 & 0 \\
0 & 0 & 0 & 0
\end{bmatrix}$},\hspace*{4pt}
\scalebox{0.8}{$\begin{bmatrix}
0 & 0 & 0 & 1 \\
0 & 0 & 1 & 1 \\
0 & 0 & 0 & 0 \\
0 & 0 & 0 & 0
\end{bmatrix}$},\hspace*{4pt}
\scalebox{0.8}{$\begin{bmatrix}
0 & 0 & 1 & 0 \\
0 & 0 & 1 & 1 \\
0 & 0 & 0 & 0 \\
0 & 0 & 0 & 0
\end{bmatrix}$},\hspace*{4pt}
\scalebox{0.8}{$\begin{bmatrix}
0 & 0 & 1 & 1 \\
0 & 0 & 1 & a \\
0 & 0 & 0 & 0 \\
0 & 0 & 0 & 0
\end{bmatrix}$},
\end{align}
where $a\in\field\setminus\{1\}$. This shows that $\ppareps[R]{M}$ may have size at least $|\field|$. In particular, we cannot necessarily recover a $\ppa{Q(k,l)}$-module up to isomorphism from $M$ and $R$. In fact, the general situation is even worse than in this example. When $k+l\ge 7$, the preprojective algebra $\ppa{Q(k,l)}$ is {\itshape wild}, meaning that its indecomposable representations occur in families requiring an arbitrarily high number of parameters to index them \cite[Lemma 3.5 and Theorem 3.7]{erdmann_skowronski06}.
\end{eg}

We have the following enumerative consequence of \cref{module_stabilizer}\ref{module_stabilizer_preprojective}:
\begin{lem}\label{module_to_triple_count}
Let $M$ be a nonnegative-integer matrix, and let $R$ be a reverse plane partition, such that $\type(M) = \type(R)$. Write $R = \RPP{T}{T'}$, as in \cref{bijection_tableaux_to_RPP}. Then if $\field$ is finite, we have
\begin{gather*}
\sum_{[V]\in\ppareps[R]{M}}\frac{1}{|\!\Aut(V)|} = \frac{|\triples[T,T']{M}\hspace*{-1pt}|}{|\GL_n|}.
\end{gather*}

\end{lem}

\begin{proof}
By \cref{ppareps_concrete_description,preprojective_jordan}, the sum on the left-hand side equals
$$
\sum_\nil \frac{1}{|\!\Aut(\decnil{M}{\nil})|},
$$
where the sum is over all $\nil\in\Nil_M/\Stab_{\GL_n}\hspace*{-1pt}(\Unit{\id}, \Unit{\perm{M}})$ such that $\JF(\nil;\Unit{\id}) = T$ and $\JF(\nil;\Unit{\perm{M}}) = T'$. By \cref{module_stabilizer}\ref{module_stabilizer_preprojective} and the orbit-stabilizer theorem, this equals
$$
\sum_\nil \frac{|\GL_n \cdot (\Unit{\id}, \Unit{\perm{M}}, \nil)|}{|\GL_n|},
$$
which equals the right-hand side by \cref{compatible_matrices_adjoint,relative_position_invariance}\ref{relative_position_invariance_action}.
\end{proof}

We now prove the main result of this section. We recall that $\ncon{\cdot}$ was defined in \cref{defn_n}, $\qwt(\cdot)$ was defined in \cref{defn_whittaker}, $\Tabs{\alpha}{\beta}$ was defined in \cref{defn_triples}, and $\prnoarg$ and $\prdualnoarg$ were defined in \cref{defn_probabilities}:
\begin{thm}\label{preprojective_enumeration}
Set $\field := \qinvF$. Let $\alpha = (\alpha_1, \dots, \alpha_k),\beta = (\beta_1, \dots, \beta_l)\models n$.
\begin{enumerate}[label=(\roman*), leftmargin=*]
\item\label{preprojective_enumeration_term} For all $M\in\Mats{\alpha}{\beta}$ and $R\in\RPPs{\alpha}{\beta}$, we have
\begin{multline*}
\sum_{[V]\in\ppareps[R]{M}}\frac{1}{|\!\Aut(V)|} = \frac{q^{n+\ncon{\alpha}+\ncon{\beta}}(1-q)^{-n}}{\prod_{1 \le i \le k,\hspace*{1pt} 1 \le j \le l}\qfac{M_{i,j}}}\hspace*{1pt}\pr{M}{T,T'} \\
= \frac{q^{n+\ncon{\alpha}+\ncon{\beta}}(1-q)^{-\lambda_1}}{\prod_{i\ge 1}\qfac{\lambda_i - \lambda_{i+1}}}\qwt(T)\qwt(T')\hspace*{2pt}\prdual{M}{T,T'},
\end{multline*}
where $R = \RPP{T}{T'}$, as in \cref{bijection_tableaux_to_RPP}.
\item\label{preprojective_enumeration_relative_position} For all $M\in\Mats{\alpha}{\beta}$, we have
$$
\sum_{[V]\in\ppareps{M}}\frac{1}{|\!\Aut(V)|} = \frac{q^{n+\ncon{\alpha}+\ncon{\beta}}(1-q)^{-n}}{\prod_{1 \le i \le k,\hspace*{1pt} 1 \le j \le l}\qfac{M_{i,j}}}.
$$
\item\label{preprojective_enumeration_RPP} For all $R\in\RPPs{\alpha}{\beta}$, we have
$$
\sum_{[V]\in\ppareps[R]{\alpha,\beta}}\frac{1}{|\!\Aut(V)|} = \frac{q^{n+\ncon{\alpha}+\ncon{\beta}}(1-q)^{-\lambda_1}}{\prod_{i\ge 1}\qfac{\lambda_i - \lambda_{i+1}}}\qwt(T)\qwt(T'),
$$
where $R = \RPP{T}{T'}$, as in \cref{bijection_tableaux_to_RPP}.
\item\label{preprojective_enumeration_sum} We have
\begin{multline*}
\sum_{[V]\in\ppareps{\alpha,\beta}}\frac{1}{|\!\Aut(V)|} = \sum_{M\in\Mats{\alpha}{\beta}}\frac{q^{n+\ncon{\alpha}+\ncon{\beta}}(1-q)^{-n}}{\prod_{1 \le i \le k,\hspace*{1pt} 1 \le j \le l}\qfac{M_{i,j}}} \\
= q^{n+\ncon{\alpha}+\ncon{\beta}}\sum_{(T,T')\in\Tabs{\alpha}{\beta}}\frac{(1-q)^{-\lambda_1}}{\prod_{i\ge 1}\qfac{\lambda_i - \lambda_{i+1}}}\qwt(T)\qwt(T'),
\end{multline*}
which equals $q^{n+\ncon{\alpha}+\ncon{\beta}}$ times \eqref{q-cauchy_coefficient_intro}, where above, $\lambda\vdash n$ denotes the shape of both $T$ and $T'$.
\end{enumerate}

\end{thm}

\begin{proof}
Part \ref{preprojective_enumeration_term} follows from \cref{module_to_triple_count}, \cref{q_enumeration}\ref{q_enumeration_GL}, and \eqref{reversibility_equation}. The remaining parts then follow by summation, using \cref{bijection_tableaux_to_RPP}.
\end{proof}

We observe that the forward and backward probabilities $\prnoarg$ and $\prdualnoarg$ defining the $q$-Burge correspondence can be described directly in terms of $\ppa{Q(k,l)}$-modules and reverse plane partitions. To state this result, when $\field$ is finite, we endow $\ppareps{\alpha,\beta}$ (and its subsets) with the measure in which the probability of obtaining the isomorphism class $[V]\in\ppareps{\alpha,\beta}$ is proportional to $\frac{1}{|\!\Aut(V)|}$.
\begin{prop}\label{probabilities_quiver}
Set $\field := \qinvF$. Let $\alpha = (\alpha_1, \dots, \alpha_k),\beta = (\beta_1, \dots, \beta_l)\models n$, let $M\in\Mats{\alpha}{\beta}$, let $R\in\RPPs{\alpha}{\beta}$, and write $R = \RPP{T}{T'}$, as in \cref{bijection_tableaux_to_RPP}.
\begin{enumerate}[label=(\roman*), leftmargin=*]
\item\label{probabilities_quiver_forward} Let $V$ denote a random element of $\ppareps{M}$. Then
$$
\pr{M}{T,T'} = \Pr\big(\socRPP{V} = R\big).
$$
\item\label{probabilities_quiver_backward} Let $V$ denote a random element of $\ppareps[R]{\alpha,\beta}$. Then
\begin{gather*}
\prdual{M}{T,T'} = \Pr\big(V \textnormal{ is isomorphic to $\dec{M}$ as a $\pa{Q(k,l)}$-module}).
\end{gather*}
\end{enumerate}

\end{prop}

\begin{proof}
This follows from \cref{probabilities_interpretation,module_to_triple_count}, using \cref{module_relative_position}, \cref{preprojective_bijection}, and \cref{preprojective_jordan}.
\end{proof}

While we find \cref{preprojective_enumeration} most significant for its enumerative implications, we can also use part \ref{preprojective_enumeration_term} to calculate the forward and backward probabilities $\prnoarg$ and $\prdualnoarg$, by enumerating $\ppareps[R]{M}$. We demonstrate this with two examples. It may also be possible to use part \ref{preprojective_enumeration_term} to address \cref{probabilities_problem}, though we ourselves have not been able to do this.
\begin{eg}\label{eg_preprojective_enumeration}
We adopt the setup of \cref{eg_reversibility}. Let $R := \RPP{T}{T'} = \scalebox{0.8}{\begin{tabular}{ccc}
& $1$ & \\
$2$ & & $3$ \\
& $4$ & 
\end{tabular}}$. By \cref{preprojective_enumeration}\ref{preprojective_enumeration_term}, we have
$$
\sum_{[V]\in\ppareps[R]{M}}\frac{1}{|\!\Aut(V)|} = \frac{q^{13}(1-q)^{-5}}{\qfac{1}\qfac{1}\qfac{2}\qfac{1}}\hspace*{1pt}\pr{M}{T,T'} = \frac{q^{13}(1-q)^{-4}}{\qfac{3}\qfac{1}}\qbinom{3}{2}\qbinom{3}{1}\prdual{M}{T,T'}.
$$
Let us show that the left-hand side above equals $q^{13}(1-q)^{-4}(1+q)^{-1}$, whence $\pr{M}{T,T'} = 1-q$ and $\prdual{M}{T,T'} = (1 + q + q^2)^{-1}$, in agreement with \cref{eg_reversibility}.

We have 
$$
\dec{M} = \hspace*{4pt}\begin{tikzpicture}[baseline=(current bounding box.base)]
\tikzstyle{out1}=[inner sep=0,minimum size=1.2mm,circle,draw=black,fill=black,semithick]
\pgfmathsetmacro{\s}{3.00};
\node[out1](v1)at(0*\s,0)[label={[below=3pt]$1$},label={[above=-2pt]$\spn{\unit{1},\unit{2}}$}]{};
\node[out1](v0)at(1*\s,0)[label={[below=3pt]$0$},label={[above=-1pt]$\field^5$}]{};
\node[out1](v-1)at(2*\s,0)[label={[below=3pt]$-1$},label={[above=-2pt]$\spn{\unit{1},\unit{3},\unit{4}}$}]{};
\path[-latex',thick](v1)edge node[above=-1pt]{$\id$}(v0) (v-1)edge node[above=-1pt]{$\id$}(v0);
\end{tikzpicture}\hspace*{4pt}.
$$
By \cref{module_stabilizer}\ref{module_stabilizer_path_algebra}, we have
$$
\Aut(\dec{M}) \cong \Stab_{\GL_5}\hspace*{-1pt}(\Unit{\id}, \Unit{\perm{M}}) = \left\{\begin{bmatrix}
\ast & \ast & \ast & \ast & \ast \\
0 & \ast & 0 & 0 & \ast \\
0 & 0 & \ast & \ast & \ast \\
0 & 0 & \ast & \ast & \ast \\
0 & 0 & 0 & 0 & \ast
\end{bmatrix}\right\}\subseteq\GL_5,
$$
where $\ast$ denotes an arbitrary element of $\field$. Also, recall the description of $\Nil_M$ from \cref{eg_inversion_matrices}. Then by \cref{ppareps_concrete_description}, we find that $\ppareps[R]{M}$ consists of a single isomorphism class, namely, $[\decnil{M}{\nil}]$ for
$$
\nil := \begin{bmatrix}
0 & 0 & 0 & 0 & 1 \\
0 & 0 & 0 & 0 & 0 \\
0 & 0 & 0 & 0 & 0 \\
0 & 0 & 0 & 0 & 0 \\
0 & 0 & 0 & 0 & 0
\end{bmatrix}.
$$
By \cref{module_stabilizer}\ref{module_stabilizer_preprojective}, we have
\begin{multline*}
\Aut(\decnil{M}{\nil}) \cong \Stab_{\GL_5}\hspace*{-1pt}(\Unit{\id}, \Unit{\perm{M}}, \nil) \\
= \{g\in\Stab_{\GL_5}\hspace*{-1pt}(\Unit{\id}, \Unit{\perm{M}}) : g_{1,1} = g_{5,5}\} \cong \GL_1^2 \times \GL_2 \times \field^7.
\end{multline*}
By \cref{q_enumeration}\ref{q_enumeration_GL}, we obtain
$$
\sum_{[V]\in\ppareps[R]{M}}\frac{1}{|\!\Aut(V)|} = \frac{1}{|\!\Aut(\decnil{M}{\nil})|} = q^{13}(1-q)^{-4}(1+q)^{-1},
$$
as desired.
\end{eg}

\begin{eg}\label{eg_isomorphism_classes_probability}
Set $\field := \qinvF$, and adopt the setup of \cref{eg_isomorphism_classes}. Write $R = \RPP{T}{T'}$, as in \cref{bijection_tableaux_to_RPP}; we have $T = T' = \ytableausmall{1 & 2 \\ 3 & 4}$. Let us use \cref{preprojective_enumeration}\ref{preprojective_enumeration_term} to calculate $\pr{M}{T,T'}$ and $\prdual{M}{T,T'}$:
$$
\sum_{[V]\in\ppareps[R]{M}}\frac{1}{|\!\Aut(V)|} = q^4(1-q)^{-4}\hspace*{1pt}\pr{M}{T,T'} = q^4(1-q)^{-2}(1+q)^3\hspace*{1pt}\prdual{M}{T,T'}.
$$

Let $\nil$ denote one of the matrices in \eqref{eg_isomorphism_classes_equation}. If $\nil = \scalebox{0.8}{$\begin{bmatrix}
0 & 0 & 1 & 0 \\
0 & 0 & 0 & 1 \\
0 & 0 & 0 & 0 \\
0 & 0 & 0 & 0
\end{bmatrix}$}$, then
$$
\Aut(\decnil{M}{\nil}) \cong \{g\in\Stab_{\GL_4}\hspace*{-1pt}(\Unit{\id}, \Unit{\perm{M}}) : g_{1,1} = g_{3,3} \text{ and } g_{2,2} = g_{4,4}\} \cong \GL_1^2 \times \field^4,
$$
and so $\frac{1}{|\!\Aut(\decnil{M}{\nil})|} = q^6(1-q)^{-2}$. Similarly, if $\nil = \scalebox{0.8}{$\begin{bmatrix}
0 & 0 & 0 & 1 \\
0 & 0 & 1 & 0 \\
0 & 0 & 0 & 0 \\
0 & 0 & 0 & 0
\end{bmatrix}$}$, then $\frac{1}{|\!\Aut(\decnil{M}{\nil})|} = q^6(1-q)^{-2}$. If $\nil$ is any other matrix in \eqref{eg_isomorphism_classes_equation}, then
$$
\Aut(\decnil{M}{\nil}) \cong \{g\in\Stab_{\GL_4}\hspace*{-1pt}(\Unit{\id}, \Unit{\perm{M}}) : g_{1,1} = g_{2,2} = g_{3,3} = g_{4,4}\} \cong \GL_1 \times \field^4,
$$
and so $\frac{1}{|\!\Aut(\decnil{M}{\nil})|} = q^5(1-q)^{-1}$. Therefore
$$
\sum_{[V]\in\ppareps[R]{M}}\frac{1}{|\!\Aut(V)|} = 2q^6(1-q)^{-2} + (q^{-1}+2)q^5(1-q)^{-1} = q^4(1-q)^{-2}(1+q).
$$
We obtain
\begin{gather*}
\pr{M}{T,T'} = (1-q)^2(1+q) \quad \text{ and } \quad \prdual{M}{T,T'} = (1+q)^{-2}.\qedhere
\end{gather*}

\end{eg}

\subsection{The classical Burge correspondence from socle filtrations}\label{sec_socle_to_burge}
We conclude this section by rephrasing the flag-variety based description of the Burge correspondence (\cref{classical_burge_rank}) in terms of $\ppa{Q(k,l)}$-modules and their socle filtrations.

\begin{thm}\label{classical_burge_socle}
Let $\field$ be algebraically closed. Let $M\in\mathbb{N}^{k\times l}$, and let $\mathcal{V}$ denote the set of all $\ppa{Q(k,l)}$-modules $V$ which restrict to $V_M$ as a $\pa{Q(k,l)}$-module. Recall from \cref{preprojective_bijection} that $\mathcal{V}$ is an affine space over $\field$. Then for all $V$ in some nonempty Zariski-open subset of $\mathcal{V}$, we have
\begin{gather*}
\socRPP{V} = \RPP{\Ptab{M}}{\Qtab{M}}.
\end{gather*}

\end{thm}

\begin{proof}
This follows from \cref{classical_burge_rank}, using \cref{preprojective_bijection,preprojective_jordan}.
\end{proof}

\begin{rmk}\label{dynkin_remark}
The description of the Burge correspondence in \cref{classical_burge_socle} suggests replacing $Q(k,l)$ with any Dynkin quiver $Q$ with a distinguished minuscule vertex, and examining the combinatorics of socle filtrations of certain generic $\ppa{Q}$-modules. We believe this would be a fruitful topic for further study.
\end{rmk}

\section{\texorpdfstring{$q$}{q}-Whittaker functions from Nakajima quiver varieties}\label{sec_nakajima}

\noindent In \cref{sec_quiver}, we showed that the $q$-Burge correspondence can be reinterpreted in terms of modules over the preprojective algebra of a path quiver. In this section, we show that the description of the $q$-Whittaker function in \cref{whittaker_expansion} can similarly be rephrased in terms of the Nakajima quiver variety of a path quiver. The exposition is entirely analogous, where we replace the quiver $Q(k,l)$ with the quiver on $k$ vertices obtained by deleting vertices $-1, \dots, -l+1$.

Quiver varieties were introduced by Nakajima \cite{nakajima94,nakajima98}; for the general definition and further background, see \cite[Chapter 10]{kirillov16}. We will study them only in the special case of interest to us, following \cite[Section 10.7]{kirillov16}.
\begin{defn}[{\cite[Section 10.7]{kirillov16}}]\label{defn_quiver_variety}
Let $\alpha = (\alpha_1, \dots, \alpha_k) \models n$. Define the path quivers on $k-1$ and $k$ vertices, respectively, as follows:
\begin{align*}
& Q := \hspace*{4pt}\begin{tikzpicture}[baseline=(current bounding box.center)]
\tikzstyle{out1}=[inner sep=0,minimum size=1.2mm,circle,draw=black,fill=black,semithick]
\pgfmathsetmacro{\s}{3.00};
\node[out1](v1)at(0*\s,0)[label={[below=3pt]\vphantom{$k-3$}$k-1$}]{};
\node[out1](v2)at(1*\s,0)[label={[below=3pt]\vphantom{$k-3$}$k-2$}]{};
\node[out1](v3)at(2*\s,0)[label={[below=3pt]\vphantom{$k-3$}$k-3$}]{};
\node[out1](v4)at(3*\s,0)[label={[below=3pt]\vphantom{$k-3$}$2$}]{};
\node[out1](v5)at(4*\s,0)[label={[below=3pt]\vphantom{$k-3$}$1$}]{};
\path[-latex',thick](v1)edge(v2) (v2)edge(v3) (v4)edge(v5);
\node[inner sep=0]at(2.5*\s,0){$\cdots$};
\end{tikzpicture}\hspace*{4pt}, \\[8pt]
& \tilde{Q} := \hspace*{4pt}\begin{tikzpicture}[baseline=(current bounding box.center)]
\tikzstyle{out1}=[inner sep=0,minimum size=1.2mm,circle,draw=black,fill=black,semithick]
\pgfmathsetmacro{\s}{3.00};
\node[out1](v1)at(0*\s,0)[label={[below=3pt]\vphantom{$k-3$}$k-1$}]{};
\node[out1](v2)at(1*\s,0)[label={[below=3pt]\vphantom{$k-3$}$k-2$}]{};
\node[out1](v3)at(2*\s,0)[label={[below=3pt]\vphantom{$k-3$}$k-3$}]{};
\node[out1](v4)at(3*\s,0)[label={[below=3pt]\vphantom{$k-3$}$1$}]{};
\node[out1](v5)at(4*\s,0)[label={[below=3pt]\vphantom{$k-3$}$0$}]{};
\path[-latex',thick](v1)edge(v2) (v2)edge(v3) (v4)edge(v5);
\node[inner sep=0]at(2.5*\s,0){$\cdots$};
\end{tikzpicture}\hspace*{4pt}.
\end{align*}
Given $F\in\PFl_\alpha$ and a nilpotent endomorphism $\nil\in\gl_n$ which are strictly compatible, define the $\pa{\double{\tilde{Q}}}$-module
$$
\qvpoint{F}{\nil} := \hspace*{4pt}\begin{tikzpicture}[baseline=(current bounding box.base)]
\tikzstyle{out1}=[inner sep=0,minimum size=1.2mm,circle,draw=black,fill=black,semithick]
\pgfmathsetmacro{\s}{3.00};
\node[out1](v1)at(0*\s,0)[label={[below=3pt]\vphantom{$k-3$}$k-1$},label={[above=-1pt]\vphantom{$F_{k-1}$}$F_1$}]{};
\node[out1](v2)at(1*\s,0)[label={[below=3pt]\vphantom{$k-3$}$k-2$},label={[above=-1pt]\vphantom{$F_{k-1}$}$F_2$}]{};
\node[out1](v3)at(2*\s,0)[label={[below=3pt]\vphantom{$k-3$}$k-3$},label={[above=-1pt]\vphantom{$F_{k-1}$}$F_3$}]{};
\node[out1](v4)at(3*\s,0)[label={[below=3pt]\vphantom{$k-3$}$1$},label={[above=-1pt]\vphantom{$F_{k-1}$}$F_{k-1}$}]{};
\node[out1](v5)at(4*\s,0)[label={[below=3pt]\vphantom{$k-3$}$0$},label={[above=-1pt]\vphantom{$F_{k-1}$}$F_k = \field^n$}]{};
\path[-latex',thick](v1)edge[bend left=15] node[above=-1pt]{$\id$}(v2) (v2)edge[bend left=15] node[below=-1pt]{\vphantom{$\nil$}$\nil$}(v1) (v2)edge[bend left=15] node[above=-1pt]{$\id$}(v3) (v3)edge[bend left=15] node[below=-1pt]{\vphantom{$\nil$}$\nil$}(v2) (v4)edge[bend left=15] node[above=-1pt]{$\id$}(v5) (v5)edge[bend left=15] node[below=-1pt]{\vphantom{$\nil$}$\nil$}(v4);
\node[inner sep=0]at(2.5*\s,0){$\cdots$};
\end{tikzpicture}\hspace*{4pt},
$$
and let $\qv{\alpha}$ denote the set of all such $\qvpoint{F}{\nil}$. Then $\qv{\alpha}$ is the {\itshape Nakajima quiver variety} associated to $Q$ with the parameters $\mathbf{v} := (\alpha_1, \alpha_1 + \alpha_2, \dots, \alpha_1 + \cdots + \alpha_{k-1})$ and $\mathbf{w} := (0, \dots, 0, n)$.
\end{defn}

\begin{rmk}\label{quiver_vs_preprojective_remark}
We point out that the elements of the Nakajima quiver variety $\qv{\alpha}$ satisfy the relation \eqref{preprojective_conditions} defining $\ppa{Q(k,0)}$, for all vertices $i$ except possibly $i=0$. In fact, $\qvpoint{F}{\nil}$ defines a $\ppa{Q(k,0)}$-module if and only if it satisfies \eqref{preprojective_conditions} at $i=0$, if and only if $\nil = 0$. This distinction is precisely the reason that in this section we work with the Nakajima quiver variety, rather than $\ppa{Q(k,0)}$-modules.
\end{rmk}

\begin{eg}\label{eg_quiver_variety}
Let $\alpha := (2,3)$. Then $\qv{\alpha}$ is the set of all representations of the form
$$
\qvpoint{F}{\nil} = \hspace*{4pt}\begin{tikzpicture}[baseline=(current bounding box.center)]
\tikzstyle{out1}=[inner sep=0,minimum size=1.2mm,circle,draw=black,fill=black,semithick]
\pgfmathsetmacro{\s}{3.00};
\node[out1](v1)at(0*\s,0)[label={[below=3pt]$1$},label={[above=-1pt]\vphantom{$V$}$V$}]{};
\node[out1](v0)at(1*\s,0)[label={[below=3pt]$0$},label={[above=-1pt]\vphantom{$V$}$\field^5$}]{};
\path[-latex',thick](v1)edge[bend left=15] node[above=-1pt]{$\id$}(v0) (v0)edge[bend left=15] node[below=-1pt]{\vphantom{$\nil$}$\nil$}(v1);
\end{tikzpicture}\hspace*{4pt},
$$
where $F = (0, V, \field^5)$, $V\in\Gr_{2,5}$, and $\nil\in\gl_5$ satisfies $\nil(\field^5) \subseteq V \subseteq \ker(\nil)$.
\end{eg}

We observe that the dimension vectors for the socle filtration of $\qvpoint{F}{\nil}$ are uniquely determined by $\JF(\nil;F)$:
\begin{lem}\label{quiver_variety_socle}
Let $\alpha = (\alpha_1, \dots, \alpha_k)\models n$, let $\qvpoint{F}{\nil}\in\qv{\alpha}$, and let $0 \le i \le k-1$.
\begin{enumerate}[label=(\roman*), leftmargin=*]
\item\label{quiver_variety_socle_beginning} We have $\socf{\qvpoint{F}{\nil}}{j}(i) = 0$ for all $0 \le j \le i$.
\item\label{quiver_variety_socle_parity} We have
$$
\socf{\qvpoint{F}{\nil}}{j+1}(i) = \socf{\qvpoint{F}{\nil}}{j}(i) = \ker((\nil|_{F_{k-i}})^{\frac{j-i+1}{2}})
$$
for all $j \ge i+1$ such $j\equiv i+1\hspace*{-2pt}\pmod{2}$.
\item\label{quiver_variety_socle_end} We have $\socf{\qvpoint{F}{\nil}}{j}(i) = \qvpoint{F}{\nil}(i)$ for all $j \ge 2k-i-1$.
\end{enumerate}
In particular, 
$$
\dim(\socf{\qvpoint{F}{\nil}}{j}(i)) = \sum_{1 \le j' \le \frac{j-i+1}{2}}T^{(k-i)}_{j'} \quad \text{ for all } j\in\mathbb{N},
$$
where $T := \JF(\nil;F)$.
\end{lem}

\begin{proof}
This follows from the definitions and \cref{matrix_to_partition}, as in \cref{sec_socle}.
\end{proof}

\begin{defn}\label{defn_quiver_variety_refined}
Let $\alpha\models n$. Given a semistandard tableau $T$ with content $\alpha$, we let $\qv{\alpha}^T$ denote the subset of $\qv{\alpha}$ of all points $\qvpoint{F}{\nil}$ such that $\JF(\nil;F) = T$ (i.e.\ $\qvpoint{F}{\nil}$ has socle filtration corresponding to $T$, as in \cref{quiver_variety_socle}). We have the decomposition
\begin{gather*}
\qv{\alpha} = \bigsqcup_{\lambda\vdash n,\hspace*{2pt} T\in\SSYT(\lambda,\alpha)}\qv{\alpha}^T.
\end{gather*}

\end{defn}

In the following statement, we recall that $\ncon{\cdot}$ was defined in \cref{defn_n} and $\qwt(\cdot)$ was defined in \cref{defn_whittaker}.
\begin{thm}\label{quiver_variety_enumeration}
Set $\field := \qinvF$. Let $\alpha\models n$.
\begin{enumerate}[label=(\roman*), leftmargin=*]
\item\label{quiver_variety_enumeration_tableau} For all $T\in\SSYT(\lambda,\alpha)$, we have
$$
|\qv{\alpha}^T| = \frac{q^{-n^2+n+\ncon{\lambda}+\ncon{\alpha}}(1-q)^{n-\lambda_1}\qfac{n}}{\prod_{i\ge 1}\qfac{\lambda_i - \lambda_{i+1}}}\qwt(T).
$$
\item\label{quiver_variety_enumeration_total} We have
\begin{gather*}
|\qv{\alpha}| = \sum_{\lambda\vdash n}\frac{q^{-n^2+n+\ncon{\lambda}+\ncon{\alpha}}(1-q)^{n-\lambda_1}\qfac{n}}{\prod_{i\ge 1}\qfac{\lambda_i - \lambda_{i+1}}}[\mathbf{x}^\alpha]W_\lambda(\mathbf{x};q).
\end{gather*}
\end{enumerate}

\end{thm}

\begin{proof}
Part \ref{quiver_variety_enumeration_tableau} follows from \cref{whittaker_expansion}\ref{whittaker_expansion_tableau}, using \cref{transitive_action,compatible_matrices_adjoint,jordan_form_count}. Part \ref{quiver_variety_enumeration_total} then follows by summation.
\end{proof}

\begin{eg}\label{eg_quiver_variety_enumeration}
Set $\field := \qinvF$, and let $\alpha := (2,3)$ and $T := \ytableausmall{1 & 1 & 2 & 2\\ 2}$. Then by \cref{quiver_variety_enumeration}\ref{quiver_variety_enumeration_tableau}, the size of $\qv{\alpha}^T$ is
$$
\frac{q^{-10}(1-q)^1\qfac{5}}{\qfac{3}\qfac{1}}\qbinom{3}{1} = q^{-10}(1-q)(1+q+q^2)(1+q+q^2+q^3)(1+q+q^2+q^3+q^4).
$$
We can also verify this directly using the setup of \cref{eg_quiver_variety}.
\end{eg}

\bibliographystyle{alpha}
\bibliography{ref}

\end{document}